\documentclass[11pt,amssymb]{amsart}
\usepackage{amssymb}
\DeclareFontFamily{OT1}{rsfs}{}
\DeclareFontShape{OT1}{rsfs}{n}{it}{<-> rsfs10}{}
\DeclareMathAlphabet{\mathscr}{OT1}{rsfs}{n}{it}

\newcommand{\Z}{{\mathbb Z}}

\newcommand{\C}{{\mathbb C}}
\newcommand{\Q}{{\mathbb Q}}
\newcommand{\Nrd}{\mathrm{Nrd}}

\newcommand{\cI}{\mathscr{I}}
\newcommand{\R}{{\mathbb R}}
\newcommand{\id}{\mathrm{id}}

\newcommand{\rk}{\mathrm{rk}}
\newcommand{\Br}{\mathrm{Br}}
\newcommand{\Cor}{\mathrm{Cor}}
\newcommand{\Res}{\mathrm{Res}}
\newcommand{\cF}{\mathscr{F}}
\newcommand{\Tr}{\mathrm{Tr}}
\newcommand{\cK}{\mathscr{K}}
\newcommand{\cL}{\mathscr{L}}

\newcommand{\cH}{\mathscr{H}}
\newcommand{\cB}{\mathcal{B}}

\newcommand{\cE}{\mathscr{E}}
\newcommand{\cG}{\mathscr{G}}

\newcommand{\cT}{\mathscr{T}}

\newcommand{\cO}{\mathcal{O}}

\newcommand{\cS}{\mathcal{S}}
\newcommand{\Int}{\mathrm{Int}\: }

\newcommand{\Ga}{\mathrm{Gal}}
\newtheorem{thm}{Theorem}[section]
\newtheorem{lemma}[thm]{Lemma}
\newtheorem{prop}[thm]{Proposition}
\newtheorem{cor}[thm]{Corollary}

\begin{document}

\title[Local-global principles]{Local-global principles for embedding of fields with involution
into simple algebras with involution}\author[Prasad]{Gopal Prasad}
\author[Rapinchuk]{Andrei S. Rapinchuk}

\address{Department of Mathematics, University of Michigan, Ann
Arbor, MI 48109}

\email{gprasad@umich.edu}

\address{Department of Mathematics, University of Virginia,
Charlottesville, VA 22904}

\email{asr3x@virginia.edu}
\maketitle
\centerline{\it Dedicated to Jean-Pierre Serre}
\begin{abstract}
In this paper we prove local-global principles for the existence of
an embedding $(E , \sigma) \hookrightarrow (A , \tau)$ of a given
global field $E$ endowed with an involutive automorphism $\sigma$
into a simple algebra $A$ given with an
involution $\tau$ in all situations except where $A$ is a matrix algebra of even
degree over a quaternion division algebra
and $\tau$ is orthogonal (Theorem A of the
introduction). Rather surprisingly, in the latter case we have a
result which in some sense is opposite to the local-global
principle, viz. algebras with involution locally isomorphic to
$(A , \tau)$ are distinguished by their maximal subfields invariant
under the involution (Theorem B of the introduction). These results
can be used in the study of classical groups over global fields. In
particular, we use Theorem B to complete the analysis of weakly
commensurable Zariski-dense $S$-arithmetic groups in all absolutely
simple algebraic groups of type different from $D_4$ which was
initiated in our paper \cite{PR2}. More precisely, we prove that in
a group of type $D_n,$ $n$ even $ > 4$, two weakly commensurable
Zariski-dense S-arithmetic subgroups are actually commensurable. As
indicated in \cite{PR2}, this fact leads to results about length-commesurable
and isospectral compact hyperbolic manifolds of dimension $4n+7$, with $n\geqslant 1$.
The appendix contains a Galois-cohomological interpretation of our
embedding theorems.
\end{abstract}

\section{Introduction}\label{S:I}

Let $A$ be a central simple algebra of dimension $n^2$ over a field $L$,
and let $\tau$ be an involution of
$A.$ Set $K = L^{\tau}.$ We recall that $\tau$ is said to be of the
{\it first} (resp.,\:{\it second}) kind if the restriction $\tau
\vert L$ is trivial (resp.,\:nontrivial); involutions of the second
kind are often called {\it unitary}. While dealing with central
simple algebras with involution of the first kind, we will always
assume that the center is a field  of characteristic $\ne 2$. If $\tau$ is an
involution of the first kind, then it is either of {\it symplectic}
type (if $\dim_L A^{\tau} = n(n-1)/2$) or of {\it orthogonal}
type (if $\dim_L A^{\tau} = n(n+1)/2$), cf.\,\cite{BoI}, Proposition 2.6.
Now, let $E$ be an $n$-dimensional commutative \'etale $L$-algebra
endowed with an automorphism $\sigma$ of order two such that
$\sigma \vert L = \tau \vert L.$ In this paper, we will
investigate the validity of the local-global principle for
the existence of an $L$-embedding $ \iota \colon (E , \sigma)
\hookrightarrow (A , \tau)$ of algebras with involution (i.e.,
satisfying $\iota \circ \sigma = \tau \circ \iota$) in case
$K$ is a global field. More precisely, if $K$ is a global field,
we say that the
local-global principle for embeddings holds (for a particular class
of commutative \'etale algebras with involution $(E , \sigma),$ or
for a particular class of central simple algebras with involution
$(A , \tau)$) if the existence of $(L \otimes_K K_v)$-embeddings
$$
\iota_v \colon (E \otimes_K K_v , \sigma \otimes \id_{K_v})
\hookrightarrow (A \otimes_K K_v , \tau \otimes \id_{K_v}) \ \
\text{for all} \ \ v \in V^K
$$
(here $V^K$ denotes the
set of all places of $K$) implies the
existence of an $L$-embedding $\iota \colon (E , \sigma)
\hookrightarrow (A , \tau)$ as above. We will only be interested in
the commutative \'etale $L$-algebras $E$ with involution $\sigma$ such that
\begin{equation}\label{E:I1}
\dim_K E^{\sigma} = \left\{\begin{array}{ccc} n & \text{if} & \sigma
\vert L \neq \id_{L} \\ \left[ \frac{n+1}{2} \right] & \text{if} &
\sigma \vert L = \id_{L} \end{array} \right.
\end{equation}
as the $\tau$-invariant maximal commutative \'etale subalgebras of
$A$ satisfying this condition (for $\sigma = \tau \vert E$)
correspond to the maximal $K$-tori of the associated (special)
unitary group $\mathrm{SU}(A , \tau)$ (cf.\:Proposition \ref{P:E11}).
So, (\ref{E:I1}) will be tacitly assumed to hold for all algebras
$(E , \sigma)$ considered in the paper (notice that (\ref{E:I1}) is
satisfied automatically if either $E$ is a field or $\sigma \vert L
\neq \id_L,$ cf.\:Proposition \ref{P:E10}).
\vskip1mm

It turns out that the local-global principle holds unconditionally
(i.e., without any additional restriction on $(E , \sigma)$) only if
$\tau$ is a symplectic involution of $A,$ and moreover, in this case,
provided that there exists an embedding $E \hookrightarrow A$ as
algebras without involutions, one needs to check the local
conditions only for real $v$ -- cf.\,Theorem \ref{T:Sym1} and
Corollary \ref{C:Sym1} for the precise statements. In most of the
other cases, the local-global principle holds if $E$ is a field
extension of $L$ (as opposed to a general commutative \'etale
$L$-algebra). The following theorem combines the
essential parts of Theorems \ref{T:U1}, \ref{T:O-101} and
\ref{T:M-1}.

\vskip3mm

\noindent {\bf Theorem A.} {\it Let $L$ be global field. Let $A$ be a central simple
$L$-algebra of dimension $n^2$ with an involution $\tau,$ and let
$E/L$ be a {\rm field extension} of degree $n$ endowed with an involutive
automorphism $\sigma$ such that $\sigma \vert L = \tau \vert L$.
Then the local-global principle for the
existence of an embedding $\iota \colon (E , \sigma) \hookrightarrow
(A , \tau)$ holds in each of the following situations:

\vskip2mm

{\rm (i)} $\tau$ is an involution of the second kind;

\vskip1mm

{\rm (ii)} $A = M_n(K)$, and $\tau$ is an orthogonal involution;

\vskip1mm

{\rm (iii)} \parbox[t]{11cm}{$A = M_m(D)$, where $D$ is a quaternion
division algebra, $m$ is odd, and $\tau$ is an orthogonal
involution.} }

\vskip3mm

Assertion (i) of the above theorem for $n$ odd was
established earlier in our paper \cite{PR1} (Proposition A.2 in
Appendix A) where it was used to compute the metaplectic kernel for
absolutely simple simply connected groups of outer type $A_n.$
The other assertions of Theorem A were
unknown prior to this work (however as this work progressed we became aware of
the fact that the questions about existence of local-global principles for
embeddings were raised in various contexts by different mathematicians).
The results of \S\S\,\ref{S:U}, \ref{S:O} and
\ref{S:OS} furnish local-global principles for embedding of commutative \'etale algebras with
involution in more
general situations. On the other hand, the examples constructed in
\S\S\,\ref{S:U} and \ref{S:OS} show that the local-global principle
may fail in general if $E$ is not a field.

The only case not covered by the above theorem is  $A = M_m(D)$,
where $D$ is a quaternion division algebra, $m$ is even, and $\tau$
is an orthogonal involution of $A$ (then the corresponding algebraic
group $\mathrm{SU}(A , \tau)$ is of type $D_m$). For us, this case
was, in fact, the main motivation to investigate the
local-global principle for embeddings since it is linked to a question
left open in the original version of our paper \cite{PR2}; this question has now
been resolved using Theorem B of this paper. The main focus in \cite{PR2} was
to determine when the ``weak commensurability" of arithmetic groups
implies their commensurability. Since the relevant definitions are
somewhat technical, we will postpone them until \S \ref{S:Ap}, and
instead discuss here a closely related problem whether two forms over a
number field $K,$ of an absolutely simple simply connected algebraic group $G,$
are $K$-isomorphic if they have the same
$K$-isomorphism classes of maximal $K$-tori. It was shown in
\cite{PR2}, Theorem 7.3, that the latter condition indeed forces the
forms to be $K$-isomorphic if the type of $G$ is different from
$A_n$ $(n > 1),$ $D_n$ $(n \geqslant 4)$ or $E_6.$ On the other
hand, in \S 9 of \cite{PR2} we developed a Galois-cohomological
construction of nonisomorphic $K$-forms having the same
$K$-isomorphism classes of maximal $K$-tori for each of the
following types: $A_n,$ $n>1$,  $D_{n}$ with $n$ odd $>1$, and $E_6.$ We will now
explain how examples of this kind (for classical types) can be
produced using Theorem~A.

Suppose we are able to construct two central simple
$L$-algebras $A_1$ and $A_2$  of
dimension $n^2$ endowed with involutions $\tau_1$ and $\tau_2$
of the same kind and
type such that

\vskip2mm

(a) $(A_1 , \tau_1)$ is not isomorphic to $(A_2 , \tau_2)$ or its
opposite;

\vskip1mm

(b) \parbox[t]{11cm}{for each $v \in V^K,$ the algebra $(A_1
\otimes_K K_v , \tau_1 \otimes \id_{K_v})$ is isomorphic as a $(L
\otimes_K K_v)$-algebra to either $(A_2 \otimes_K K_v , \tau_2 \otimes
\id_{K_v})$ or its opposite.}

\vskip2mm

\noindent Then the corresponding special unitary groups $G_i =
\mathrm{SU}(A_i , \tau_i)$ are not isomorphic over $K$ but are
isomorphic over $K_v$ for all $v \in V^K.$ Furthermore, any maximal
$K$-torus of $G_1$ corresponds to a maximal commutative \'etale
$\tau_1$-invariant subalgebra $E_1$ of  $A_1$ satisfying
(\ref{E:I1}). Condition (b) implies that for each $v \in V^K$, there
is an embedding $$(E_1 \otimes_K K_v , (\tau_1 \vert E_1) \otimes
\id_{K_v}) \hookrightarrow (A_2 \otimes_K K_v , \tau_2 \otimes
\id_{K_v})$$ of algebras with involution. So, if the local-global principle
for embeddings holds for $(E_1 , \tau_1 \vert E_1),$ there exists an
embedding $(E_1 , \tau_1 \vert E_1) \hookrightarrow (A_2 , \tau_2).$
Thus, under appropriate assumptions, we obtain that $A_1$ and $A_2$
have the same isomorphism classes of maximal commutative \'etale
subalgebras, invariant under the involutions and satisfying
(\ref{E:I1}), hence the groups $G_1$ and $G_2$ have the same isomorphism
classes of maximal $K$-tori.

It is simplest to implement this construction by taking for $A_1$
and $A_2$ suitable {\it division} algebras with involutions of the
second kind as then, by Theorem A\,(i), the local-global principle for
embeddings holds for all maximal commutative \'etale subalgebras
invariant under involutions. (This was actually done in Example 6.6
in \cite{PR2} for $n$ odd - the restriction on $n$ was due to the
fact that while working on \cite{PR2} we did not know if the local-global
principle for embeddings of fields holds for arbitrary $n$.) Along the
same lines, one can  construct, for each {\it odd} $m \geqslant 3,$
a central simple $K$-algebra $A$ of dimension $n^2$, with $n = 2m$,
and two orthogonal involutions $\tau_1$ and $\tau_2$ such that $(A ,
\tau_1) \not\simeq (A , \tau_2)$ but $(A \otimes_K K_v , \tau_1
\otimes \id_{K_v}) \simeq (A \otimes_K K_v , \tau_2 \otimes
\id_{K_v})$ for all $v \in V^K,$ and then use Theorem A\,(iii)  to conclude
that $(A , \tau_1)$ and $(A , \tau_2)$ have at least the same
isomorphism classes of maximal subfields invariant under the
involutions (existence of  involutions which give the same
isomorphism classes of {\it all} maximal commutative \'etale
subalgebras, invariant under the involutions and satisfying
(\ref{E:I1}), is more subtle and requires the Galois-cohomological
constructions described in \cite{PR2}, \S 9). Theorem A, however,
does not provide information that would allow one to construct
similar examples if $m$ is even. Rather surprisingly, it turned out
that such examples simply do not exist in this case, so in effect algebras
of dimension $n^2$,
with  $4| n$, endowed with
orthogonal involutions {\it are differentiated} by the
isomorphism classes of maximal commutative \'etale subalgebras
invariant under the involutions and satisfying (\ref{E:I1}) (and
even by the isomorphism classes of maximal invariant subfields).

\vskip2mm

\noindent {\bf Theorem B.} (i) {\it Let $A_1$ and $A_2 $ be two central
simple $K$-algebras, of dimension $n^2$, $n \geqslant 3,$ endowed
with orthogonal involutions $\tau_1$ and $\tau_2$ respectively. If $A_1$ and
$A_2$ have the same isomorphism classes of
$n$-dimensional commutative \'etale subalgebras invariant under the
involutions and satisfying {\rm (\ref{E:I1})} (i.e., for any
$n$-dimensional $\tau_1$-invariant commutative \'etale subalgebra
$E_1$ of $A_1$ satisfying {\rm (\ref{E:I1})}, there exists an
embedding $(E_1 , \tau_1 \vert E_1) \hookrightarrow (A_2 , \tau_2),$
and vice versa), then
$$
(A_1 \otimes_K K_v , \tau_1 \otimes \id_{K_v}) \simeq (A_2 \otimes_K
K_v , \tau_2 \otimes \id_{K_v}) \ \ \text{for all} \ \  v \in V^K,
$$
and hence, in particular, $A_1 \simeq A_2.$ If $n$ is even, then the
same conclusion holds if $(A_1 , \tau_1)$ and $(A_2 , \tau_2)$ just
have the same isomorphism classes of maximal {\rm fields}
invariant under the involutions.}

\vskip1mm

(ii) {\it Let $A$  be a central simple $K$-algebra with an
orthogonal involution $\tau,$ of dimension $n^2$ with $4|n.$ Let
$\cI = \cI(A , \tau)$ be the set of orthogonal involutions $\eta$ on $A$
such that $(A\otimes_K K_v , \tau \otimes \id_{K_v}) \simeq (A \otimes_K K_v ,
\eta \otimes \id_{K_v})$ for all $v \in V^K.$ Then given $\eta \in \cI$, one can find
an $\eta$-invariant maximal field
$E_{\eta}$ in $A$ so that if $\nu \in \cI$ is such that there exists
an embedding $(E_{\eta} , \eta \vert E_{\eta}) \hookrightarrow (A ,
\nu)$, then $(A , \eta) \simeq (A , \nu).$}

\vskip2mm

We notice that since $\cI$ in general contains more than one
isomorphism class (cf.\:\cite{LUG} in conjunction with Proposition \ref{P:G111} below),
the local-global principle does not hold even for embeddings of
fields with involution when $n$ is a multiple of four (cf.\:Remark
8.6).

\vskip1mm

Theorem B can be used to resolve the ambiguity left open in the
original version of \cite{PR2} for groups of type $D_{2r}:$ we show
in \S \ref{S:Ap} that at least when $r > 2,$ weak commensurability
of two arithmetic subgroups of an absolutely simple group of this
type implies their commensurability (see Theorem \ref{T:Ap1} below
for the precise formulation). To describe some geometric consequences of this result,
we will now recall the main geometric results of \cite{PR2}.  Given a connected absolutely simple real algebraic group $G$, let $X$ be the symmetric space of $G(\R)$ and $\Gamma_1$ and $\Gamma_2$ be two torsion-free lattices in the latter, at least one of which is arithmetic. Let $L(X/\Gamma_1)$ and $L(X/\Gamma_2)$ be the set of lengths of closed geodesics on $X/\Gamma_1$ and $X/\Gamma_2$ respectively. $X/\Gamma_1$ and $X/\Gamma_2$ are said to be {\it length-commensurable} if\, $\Q\cdot L(X/\Gamma_1) = \Q\cdot L(X/\Gamma_2)$. We have proved in \cite{PR2} that if either $X/\Gamma_1$ and $X/\Gamma_2$ are length-commensurable, or they are compact and  isospectral, and $G$ is of type other than $A_n$ ($n>1$), $D_n$ ($n\geqslant 4$) and $E_6$, then $X/\Gamma_1$ and $X/\Gamma_2$ are commensurable (i.e., they admit a common finite-sheeted cover). Theorem \ref{T:Ap1} of this paper allows us to draw the same conclusion if $G$ is of type $D_{2r}$ with $r>2$, for example, if $X$ is the hyperbolic space of dimension  $4r-1$, with $r>2$. It has been shown in \cite{PR2}, \S 9, that if $G$ is of type $A_r$, $D_{2r+1}$, $r>1$, or $E_6$, then the above conclusion fails in general.

\vskip1mm

In the Appendix, we interpret the problem of the existence of an
embedding $(E , \sigma) \hookrightarrow (A , \tau)$ in terms of
Galois cohomology and also relate it to the problem of finding a
rational point on a certain homogeneous space.

\vskip2mm

{\bf Notation.} For a field $K$, $\overline{K}$ will denote an algebraic closure. If $K$ is a global field,  $V^K$ will denote the set of all places of
$K$, and $V^K_r$ (resp., $V^K_f$) the set of real (resp., finite)
places.

\vskip3mm

{\bf Acknowledgments.} Both the authors were partially supported by
the NSF (grants DMS-0653512 and DMS-0502120), BSF (grant 2004083) and
the Humboldt Foundation.

It is a pleasure to thank Jean-Louis Colliot-Th\'el\`ene and
Jean-Pierre Tignol for their comments. We thank the
referee for suggestions that helped to improve the exposition.

\section{On commutative \'etale algebras with involution}\label{S:E}

In \S\S \ref{S:E}, \ref{S:G}, we collect, with partial proofs, some
known results about \'etale algebras and their embeddings into
central simple algebras. In these two sections, $L$ will denote an arbitrary infinite field.
Let $E$ be a commutative \'etale
$L$-algebra of dimension $n$.  Then $E = \prod_{i = 1}^r E_i$, where
$E_i/L$ is a separable field extension and $\sum_{i = 1}^r [E_i : L]
= n.$ As usual, for $x = (x_1, \ldots , x_r) \in E,$ we set
$N_{E/L}(x) = \prod_{i = 1}^r N_{E_i/L}(x_i).$ Let $\sigma$ be a
ring automorphism of $E$ of order two leaving $L$ invariant.
\begin{prop}\label{P:E10}
{\rm (1)} Assume that $\sigma \vert L \neq \id_L$ and set $K =
L^{\sigma}.$ Then $\dim_K E^{\sigma} = n$ and any $x \in E$ such
that $x\sigma(x) = 1$ is of the form $x = y\sigma(y)^{-1}$ for some
$y \in E^{\times}.$

\vskip.1mm

\noindent {\rm (2)} Let now $\sigma \vert L = \id_L,$ and assume
that $\displaystyle \dim_L E^{\sigma} = \left[ \frac{n+1}{2}
\right].$ If $x \in E$ satisfies $x\sigma(x) = 1$, then in each of
the following cases: (i) $n$ is even, or (ii) $n$ is odd and
$N_{E/L}(x) = 1,$ we have $x =y \sigma(y)^{-1}$ for some $y \in
E^{\times}.$
\end{prop}
\begin{proof}
(1) We have $E = E^{\sigma} \otimes_K L$ (cf.\:\cite{Bor}, AG 14.2),
so $\dim_K E^{\sigma} = n.$ Clearly, $E$ is a direct product of
$\sigma$-invariant subalgebras $R$ of one of the following types:
(a) $R$ is a separable field extension of $L,$ or (b) $R = R' \times R''$
with $R', R''$ being separable field extensions of $L$ interchanged by
$\sigma,$ and it is enough to prove the second assertion of (1) for each of
these types of algebras. In case (a), the claim follows from the Hilbert's Theorem
90. In case (b), we have $x = (x' , x'')$ with
$x'\sigma(x'') = 1_{R'}$ and $x''\sigma(x') = 1_{R''}.$ Set $y =
(x' , 1_{R''}).$ Then $x =y\sigma(y)^{-1},$ as required.

\vskip1mm

(2) Here $E$ is a direct product of $\sigma$-invariant subalgebras $R$
of the following three types: (a) $R$ is a separable field
extension of $L$ and $\sigma \vert R \neq \id_R;$ (b) same $R$ but
$\sigma \vert R = \id_R;$ (c) $R = R' \times R''$ where $R' , R''$
are separable field extensions of $L$ interchanged by $\sigma.$ In cases (a)
and (c), we have $\dim_L R^{\sigma} = (1/2) \dim_L R,$ and the same
argument as in (1) shows that any $x \in R$ satisfying $x\sigma(x) =
1$ is of the form $x = y\sigma(y)^{-1}$ for some $y \in
R^{\times},$ in particular, $N_{R/L}(x) = 1.$ The assumption
$\displaystyle \dim_L E^{\sigma} = \left[\frac{n + 1}{2} \right]$
implies that if $n$ is even, then $E$ does not have components of
type (b), and our assertion follows. If $n$ is odd, then there is
only one component of type (b), and this component is 1-dimensional,
i.e. $E = E' \times E''$ where $E'$ is a direct product of components of
types (a) and (c), and $E'' = L.$ Writing $x = (x' , x''),$ we observe
that $N_{E/L}(x) = 1$ implies that $x'' = 1,$ and our assertion
again follows.
\end{proof}

\vskip2mm

\begin{prop}\label{P:E12}
We assume that $L$ is not of characteristic $2$.
Let $E $ be a commutative \'etale $L$-algebra with an involution
$\sigma$ such that $\sigma \vert L = \id_L,$ with $n := \dim_L E$
even. Set $F = E^{\sigma}$ and assume that $\dim_L F = n/2.$ Then
there exists $d \in F^{\times}$ such that
$$
(E , \sigma) \simeq (F[x]/(x^2 - d) , \theta)
$$
where $\theta$ is defined by $x \mapsto -x.$
\end{prop}
\begin{proof}
We have seen in the proof of Proposition \ref{P:E10}(2) that $E$ is
a direct sum of $\sigma$-invariant subalgebras $R$ of type (a) or
(c) introduced therein, and it is enough to prove our claim for
algebras of each of those types. If $R$ is of type (a), then the
assertion is well-known. So, let $R = R' \times R''$ where $R'$ and
$R''$ are separable extensions of $L$ such that $\sigma(R') = R''.$
Then $F = R^{\sigma}$ coincides with $\{ (a , \sigma(a)) \vert a \in
R' \},$ using which it is easy to see that the map $F[x] \to E,$ $x
\mapsto (1 , -1),$ yields an isomorphism
$$
(F[x]/(x^2 - 1) , \theta) \simeq (E , \sigma),
$$
so we can take $d = 1.$
\end{proof}

\vskip2mm

Now, let $A$ be a central simple $L$-algebra with an
involution $\tau,$  $\dim_L A = n^2.$ Set $K = L^{\tau},$ and let $H
= \mathrm{U}(A , \tau)$ and $G = \mathrm{SU}(A , \tau)$ be the
corresponding algebraic $K$-groups. Given an $n$-dimensional
$\tau$-invariant (maximal) commutative  \'etale $L$-subalgebra $E$ of
$A,$ we consider the associated maximal $K$-torus
$\mathrm{R}_{E/K}(\mathrm{GL}_1) \subset
\mathrm{R}_{L/K}(\mathrm{GL}_{1 , A}),$ and then define the
corresponding $K$-tori
$$
S = (\mathrm{R}_{E/K}(\mathrm{GL}_1) \cap H)^{\circ} \ \ \text{and}
\ \ T = (\mathrm{R}_{E/K}(\mathrm{GL}_1) \cap G)^{\circ}
$$
in $H$ and $G,$ respectively.
\begin{prop}\label{P:E11}
$S$ is a  maximal torus in $H$ (resp., $T$ is a maximal torus in
$G$) if and only if {\rm (\ref{E:I1})} holds (for $\sigma =
\tau |_E$). Any maximal $K$-torus in $H$ (resp., $G$) corresponds to
an $n$-dimensional  $\tau$-invariant commutative \'etale $L$-subalgebra
$E$ of $A$ for which {\rm (\ref{E:I1})} holds.
\end{prop}
\begin{proof}
The involution $\tau$ induces an automorphism of
$\mathrm{R}_{E/K}(\mathrm{GL}_1),$ and we then get a homomorphism
$$
\varphi \colon \mathrm{R}_{E/K}(\mathrm{GL}_1) \longrightarrow S, \
\ x \mapsto \tau(y)y^{-1}.
$$
Clearly, $\ker \varphi = \mathrm{R}_{E^{\tau}/K}(\mathrm{GL}_1),$
yielding the bound
$$
\dim S \geqslant \dim_K E - \dim_K E^{\tau} = \dim_K E_{-1},
$$
where $E_{-1}$ is the $(-1)$-eigenspace of $\tau$ in $E$. On the other
hand, the Cayley-Dickson parametrization $s \mapsto (1 - s)(1 +
s)^{-1}$ gives an injective rational map of $S$ into the affine
space corresponding to $E_{-1},$ providing the opposite bound.
Therefore,
\begin{equation}\label{E:E100}
\dim S = \dim_K E - \dim_K E^{\tau} = \dim_K E_{-1}
\end{equation}
in all cases. If $\tau \vert L \neq \id_L$,  then, on the one hand,
$\dim_K E^{\tau} = n$ (Proposition \ref{P:E10}(1)), hence $\dim S =
n,$ and on the other hand, $\rk\: H = n.$ So, $S$ is a maximal torus
of $H.$ Furthermore, $\dim T \geqslant n - 1$ and $\rk G = n - 1,$
so $T$ is a maximal torus of $G.$ Now, suppose $\tau \vert L =
\id_L.$ Then $G = H^{\circ}$ and $S = T.$ If $n$ is even, then for
both orthogonal and symplectic involutions we have $\rk\: G = n/2,$
and in view of (\ref{E:E100}), the fact that $\dim S = n/2$ is
equivalent to $\dim_K E^{\tau} = n/2,$ i.e., to (\ref{E:I1}). In $n$ is
odd, then the involution is necessarily orthogonal and $\rk\: G = (n
- 1)/2.$ Then again from (\ref{E:E100}) we obtain that $\dim S =
(n-1)/2$ is equivalent to the assertion that
$\dim_K E^{\tau} = (n + 1)/2,$ which is
again (\ref{E:I1}).

Using the well-known description of the possibilities for $(A
\otimes_K \overline{K} , \tau \otimes \id_{\overline{K}}),$ one easily produces a
maximal torus $T_0$ of $G$ which generates an  $\overline{K}$-subalgebra
of dimension $n$ if $\sigma \vert L = \id_L,$ and of dimension $2n$
otherwise, and in the latter case this subalgebra is an algebra over
$L \otimes_K \overline{K}.$ Then in view of the conjugacy of maximal tori
(\cite{Bor}, 11.3), we see that the same is true for any maximal
torus. Now, if $T$ is a maximal $K$-torus of $G$, then the
Zariski-density of $T(K)$ in $T$ (\cite{Bor}, 8.14) implies that the
$K$-subalgebra $E$ of $A$ generated by $T(K)$ (which is
automatically \'etale and $\tau$-invariant) is an $n$-dimensional
$L$-algebra. Since $T$ is maximal, (\ref{E:I1}) holds for $E$
by the first part of the proof. The argument for
maximal tori in $H$ is similar.
\end{proof}

\vskip1mm

The connection between the subalgebras satisfying (\ref{E:I1}) and
the maximal tori of the corresponding unitary group can be used to
prove the following.
\begin{prop}\label{P:E50}
Let $A$ be a central simple algebra over a global field
$L,$ of dimension $n^2,$ with an involution $\tau,$ and let $G =
\mathrm{SU}(A , \tau).$ Suppose that we are given a finite set $V$
of places of $K = L^{\tau},$ and for each $v \in V$, an
$n$-dimensional $(\tau \otimes \id_{K_v})$-invariant commutative \'etale $(L
\otimes_K K_v)$-subalgebra $E(v)$ of  $A \otimes_K K_v$ satisfying
(\ref{E:I1}) of  \S \ref{S:I}. Then there exists an $n$-dimensional
$\tau$-invariant commutative \'etale $L$-subalgebra $E$ of  $A$ satisfying
(\ref{E:I1}) of  \S \ref{S:I} such that
$$
E(v) = g_v^{-1}(E \otimes_K K_v)g_v \ \ \text{with} \ \ g_v \in
G(K_v),
$$
in particular, $(E(v) , (\tau \otimes \id_{K_v}) \vert E(v)) \simeq
(E \otimes_K K_v , (\tau \vert E) \otimes \id_{K_v})$ as $L\otimes_K K_v$-algebras with involutions,  for all $v
\in V.$
\end{prop}
\begin{proof}
Corresponding to $E(v),$ there is a maximal $K_v$-torus $T(v)$ of
$G.$ Using weak approximation in the variety of maximal tori
of $G$ (cf.\:\cite{PlR}, Corollary~3 in \S 7.2), we can find a
maximal $K$-torus $T$ of $G$ such that for all $v\in V$, $T(v) = g_v^{-1} T g_v$ for
some $g_v \in G(K_v).$ By Proposition \ref{P:E11}, $T$ corresponds
to an $n$-dimensional $\tau$-invariant commutative \'etale $L$-subalgebra $E$ of
$A,$ which is as required (notice that since $g_v \in
G(K_v),$ the $K_v$-algebra isomorphism $a \mapsto g_v a g_v^{-1}$, $E(v) \to E \otimes_K
K_v$, respects involutions).
\end{proof}

\vskip1mm

Next, we will recall the definition of a class of maximal tori in a
given semi-simple group which will play an important role in \S
\ref{S:Ap} (cf.\:also \cite{PR4}, \cite{PR2}). Let $G$ be a connected
semi-simple group defined over a field $F.$ Fix  a maximal $F$-torus
$T$ of $G,$ and let $\Phi = \Phi(G , T)$ denote the corresponding
root system. Furthermore, let $F_T$ be the minimal splitting field
of $T$ (over $F$). Then the action of the Galois group $\Ga(F_T/F)$
on the character group $X(T)$  of $T$ induces an injective group homomorphism
$\theta_T \colon \Ga(F_T/F) \longrightarrow \mathrm{Aut}(\Phi).$ In
the sequel, we will identify the Weyl group $W(\Phi)$ of the root
system $\Phi$ with the Weyl group $W(G , T).$ We say that $T$ is
{\it generic} (over $F$) if $\theta_T(\Ga(F_T/F)) \supset W(G , T).$
\begin{prop}\label{P:E55}
Let $(A , \tau)$ be a central simple $L$-algebra with involution, of
dimension $n^2$, with $n > 2.$ Set $K = L^{\tau},$ and let $G =
\mathrm{SU}(A , \tau)$ be the corresponding algebraic $K$-group.
Furthermore, let $E$ be an $n$-dimensional
$\tau$-invariant commutative \'etale $L$-subalgebra of $A$ that satisfies (\ref{E:I1}) of
 \S 1, and let $T$ be the corresponding maximal $K$-torus of $G.$ Assume
that $T$ is generic over $K.$

\vskip2mm

\noindent $\bullet$ \parbox[t]{12cm}{If either $\tau$ is of the
first kind and $n$ is even, or $\tau$ is of the second kind, then $E$
is a field extension of $L.$}

\vskip2mm

\noindent $\bullet$ \parbox[t]{12cm}{If $\tau$ is of the first kind
and $n$ is odd, then $E = E' \times K$ where $E'$ is a field
extension of $K = L.$}
\end{prop}
\begin{proof}
Since the Weyl group acts on $X(T)\otimes_{\Z}\Q$
(nontrivially and) irreducibly, the assumption that $T$
is generic over $K$ implies that $T$ does not contain proper $K$-subtori
and is $K$-anisotropic. Assume that $\tau$ is of the first
kind. If $E$ is not as described in the statement of the proposition,
then (cf.\:the proof of Proposition \ref{P:E10}) there is a
nontrivial decomposition $E = E_1 \times E_2$ such that $E_2 \neq K$ and $E_1$
is either a $\tau$-stable field extension of $K$ such that $\tau \vert E_1$ is
nontrivial, or is of the form $E_1 = E' \times E''$ and $\tau$
interchanges $E'$ and $E''.$ But in the first case $T$ has a proper
$K$-subtorus corresponding to $E_1$, and in the second case a
1-dimensional $K$-split subtorus coming from the subalgebra $K
\times K \subset E' \times E'',$ which is impossible.

Let now $\tau$
be of the second kind. Then $E \simeq L \otimes_K F$ where $F =
E^{\tau}.$ Given a  $K$-subalgebra $F'$ of  $F$ of dimension $n',$
corresponding to it there is a $K$-subtorus of  $T$ of dimension
$n' - 1.$ As $T$ does not contain proper $K$-subtori, we conclude that $F$ does
not contain any proper $K$-subalgebra of dimension $>1$. Since by our assumption, $n > 2$, we see that $F$
must be a field extension of $K.$ To prove that $E$ is a field, we
need to show that $L$ and $F$ are linearly disjoint over $K.$
If $L$ and $F$ are not linearly disjoint over $K$,
$E$ contains a subalgebra of the form $L \otimes_K L$
(with the involution acting on the first factor). Corresponding to
this subalgebra, we have a $K$-torus $S \subset H = \mathrm{U}(A ,
\tau)$ which is $K$-isomorphic to $\mathrm{R}_{L/K}(\mathrm{GL}_1).$
Since $H/G \simeq \mathrm{R}^{(1)}_{L/K}(\mathrm{GL}_1)$ is
$K$-anisotropic, the 1-dimensional $K$-split subtorus of $S$ is
contained in $G,$ hence in $T,$ a contradiction.
\end{proof}

\vskip1mm

We will now formulate, for the convenience of future reference, two
propositions about embeddings of commutative \'etale algebras into central
simple algebras. The first proposition is a particular case of
Proposition 4.3 in \cite{C-KM}.
\begin{prop}\label{P:E1}
Let $A$ be a  central simple algebra of dimension $n^2$ over a field
$L,$ and let $E$ be an $n$-dimensional commutative \'etale
$L$-algebra. If $ E = \prod_{j = 1}^{\ell} E_j,$ where
$E_j$ is a (separable) field extension of $L,$ then $E$ admits an
$L$-embedding into $A$ if and only if each $E_j$ splits $A,$ or,
equivalently, $A \otimes_L E$ is a direct sum of matrix algebras
over field extensions of $L.$
\end{prop}

\begin{prop}\label{P:E2}
Let $A$ be a central simple algebra of dimension $n^2$ over a {\rm
global} field $L,$ and $E$ be an $n$-dimensional commutative \'etale
$L$-algebra. Then an $L$-embedding $\varepsilon \colon E
\hookrightarrow A$ exists if and only if for every $w \in V^L$ there
exists an $L_w$-embedding $\varepsilon_w \colon E \otimes_L L_w
\hookrightarrow A \otimes_L L_w.$
\end{prop}

This follows from Proposition \ref{P:E1} and the fact that for a
global field $F,$ the map $\Br(F) \longrightarrow  \bigoplus_{w \in
V^F} \Br(F_w)$ is injective (cf.\:\cite{Pie}, \S 18.4).

\section{Embeddings of commutative \'etale algebras with involution
into central simple algebras with involution}\label{S:G}

In this section, $L$ is an arbitrary field, $A$ is a central simple $L$-algebra of
dimension $n^2,$ and $\tau$ an involution on $A.$ Let $E$ be an $n$-dimensional
commutative \'etale
$L$-algebra with an involutive automorphism $\sigma$ such that
$\sigma \vert L = \tau \vert L$ and condition (\ref{E:I1}) of the
introduction holds. Let $F = E^{\sigma}$. Let
$\varepsilon \colon E \hookrightarrow A$ be an $L$-embedding
which may not respect the given involutions.
\begin{prop}\label{P:G1}
{\rm (cf.\:\cite{K}, \S 2.5)} There exists a $\tau$-symmetric $g\in
A^{\times}$ such that for
$$\theta = \tau \circ \Int g = \Int g^{-1} \circ \tau ,$$ we have
\begin{equation}\label{E:15}
\varepsilon(\sigma(x)) = \theta(\varepsilon(x)) \ \ \text{for all} \
\ x \in E,
\end{equation} i.e., \: $\varepsilon: (E,\sigma)\hookrightarrow (A,\theta )$ is an
$L$-embedding of algebras with involution.
\end{prop}
\begin{proof}
Since $\tau \circ \varepsilon \circ \sigma$ is an $L$-embedding of
$E$ into $A,$ according to the ``Skolem-Noether Theorem" for
commutative \'etale subalgebras of dimension $n$ (see \cite{K0},
Hilfssatz 3.5, or \cite{K}, p.\,37)\footnote{We would like to point
out the fact, apparently missing in the literature, that this form
of the Skolem-Noether Theorem immediately follows from ``Hilbert's Theorem
90." More precisely, let $A$ be a central simple $L$-algebra of
dimension $n^2,$ and let $E$ be a commutative \'etale $L$-algebra of dimension
$n.$ Let us show that given two $L$-embeddings $\iota_i \colon E
\hookrightarrow A$ for $i = 1 , 2,$ there exists $g \in A^{\times}$
such that $\iota_2(x) = g^{-1} \iota_1(x) g$ for all $x \in E.$ We
will use $\iota_i$ to also denote its natural extension $E
\otimes_L L_{sep} \hookrightarrow A \otimes_L L_{sep},$ where
$L_{sep}$ is a separable closure of $L.$ There exists $a \in E
\otimes_L L_{sep}$ whose characteristic polynomial $p(t)$ has $n$
distinct roots, and then $E \otimes_L L_{sep} = L_{sep}[a].$ The
matrices $\iota_1(a) , \iota_2(a) \in A \otimes_L L_{sep} =
M_n(L_{sep})$ have $p(t)$ as their common characteristic polynomial,
and are therefore conjugate to each other. It follows that there exists $h \in (A
\otimes_L L_{sep})^{\times}$ such that $\iota_2(x) = h^{-1}
\iota_1(x) h$ for all $x \in E \otimes_L L_{sep}.$ Then for any
$\theta \in \Ga(L_{sep}/L),$ the element $h\theta(h)^{-1}$
centralizes $\iota_1(E),$ and hence there exists $\xi_{\theta}
\in (E \otimes_L L_{sep})^{\times}$ such that $\iota_1(\xi_{\theta})
= h\theta(h)^{-1}.$ Then the family $\xi = \{ \xi_{\theta} \}$ is
a Galois 1-cocyle with values in $T(L_{sep}) = (E \otimes_L
L_{sep})^{\times},$ where $T = \mathrm{R}_{E/L}(\mathrm{GL}_1)$ in
the standard notations. Since $H^1(L , T) = \{1\}$ (``Hilbert's Theorem 90"),
there exists $t \in (E \otimes_L L_{sep})^{\times}$ such that
$\xi_{\theta} = t\theta(t)^{-1}$ for all $\theta \in
\Ga(L_{sep}/L).$ Set $g = \iota_1(t)^{-1}h \in (A \otimes_L
L_{sep})^{\times}.$ Then $\theta(g) = g$ for every $\theta,$ implying
that $g \in A^{\times}.$ At the same time, $\iota_2(x) =
g^{-1}\iota_1(x)g$ for all $x \in E,$ as
required.} there exists $g \in A^{\times}$ such that
$$
\varepsilon(x) = g^{-1} (\tau \circ \varepsilon \circ \sigma)(x) g \
\ \text{for all} \ \ x \in E.
$$
Substituting $\sigma(x)$ for $x,$ we obtain
\begin{equation}\label{E:HF2}
\varepsilon(\sigma(x)) = g^{-1} \tau(\varepsilon(x))g.
\end{equation}
Now
$$
\varepsilon(x) = g^{-1} (\tau \circ \varepsilon \circ \sigma)(x) g =
g^{-1} \tau\left(g^{-1}\tau(\varepsilon(x))g\right)g
=(g^{-1}\tau(g)) \varepsilon(x) (\tau(g)^{-1}g),$$ for all $x \in
E$. Since $\varepsilon(E)$ is its own centralizer in $A,$ we see
that
$$
g^{-1}\tau(g) = \varepsilon(a) \ \ \text{for some} \ \ a \in E.
$$
Furthermore,
$$
\varepsilon(\sigma(a)) = g^{-1}\tau(\varepsilon(a))g =
g^{-1}\tau(g^{-1}\tau(g))g = \tau(g)^{-1}g = \varepsilon(a^{-1}).
$$
Therefore, $a\sigma(a) = 1,$ so according to Proposition \ref{P:E10}, $a = b\sigma(b)^{-1}$ for some $b \in E^{\times}$ (one
needs to observe that if  $\sigma \vert L = \id_L$  and $n$ is odd,
$N_{E/L}(a) = \mathrm{Nrd}_{A/L}(g^{-1}\tau(g)) = 1$). Set
$h = g\varepsilon(b).$ Then we have
$$\varepsilon(\sigma(x)) = \varepsilon(b)^{-1} \varepsilon(\sigma(x))
\varepsilon(b) = h^{-1} \tau(\varepsilon(x)) h \ \ \text{for} \ \ x
\in E,$$ and, in addition,
$$
\tau(h) = \tau(\varepsilon(b))\tau(g) = g \varepsilon(\sigma(b))
g^{-1}\tau(g) = g\varepsilon(\sigma(b)a) = g\varepsilon(b)= h.
$$
So, we could have assumed from the very beginning that $g$ in
(\ref{E:HF2}) is $\tau$-symmetric. Then
$$
\theta : = \Int g^{-1} \circ \tau = \tau \circ \Int g
$$
is an involution, and it follows from (\ref{E:HF2}) that (\ref{E:15}) holds.
\end{proof}

\vskip2mm

Fix an involution $\theta = \tau \circ \Int g,$ where $\tau(g) = g,$
satisfying (\ref{E:15}).

\begin{thm}\label{T:G1}
The following conditions are equivalent:

\vskip2mm

\noindent { (i)} \parbox[t]{12cm}{There exists an $L$-embedding
$\iota \colon (E , \sigma) \to (A , \tau)$ of algebras with
involution.}

\vskip2mm

\noindent {(ii)} \parbox[t]{12cm}{There exists an $a \in
F^{\times}$ such that $(A , \theta_a) \simeq (A , \tau)$
as algebras with involution, where  for $x\in
F^{\times}$, we set $\theta_x = \theta\circ \Int
\varepsilon(x)= \tau\circ\Int(g\varepsilon(x)).$}

\vskip2mm

\noindent {(iii)} \parbox[t]{12cm}{$g\varepsilon(b) = \tau(h)h $ for
some $b \in F^{\times}$ and $h \in A^{\times}.$ }
\end{thm}
\begin{proof}
$(i) \Rightarrow (ii):$ Using the Skolem-Noether Theorem, we see
that there exists $s \in A^{\times}$, such that $\iota = \Int s
\circ \varepsilon.$ By our assumption, $\iota \circ \sigma = \tau
\circ \iota$ on $E,$ and by our construction of $\theta,$ we have
$\varepsilon \circ \sigma = \theta \circ \varepsilon$ on $E.$  Let
$\psi = \Int s$. Then
$$
\psi \circ \theta \circ\varepsilon = \psi \circ \varepsilon \circ
\sigma = \tau \circ \psi \circ \varepsilon \ \ \text{on} \ \ E.
$$
So, there exists $b \in E^{\times}$ such that
\begin{equation}\label{E:11}
\tau \circ \psi = \psi \circ \theta \circ \Int \varepsilon(b),
\end{equation}
i.e.,
\begin{equation}\label{E:12}
\tau \circ \psi = \psi \circ \theta_b.
\end{equation}
From
$$
\mathrm{id}_A = (\psi^{-1} \circ \tau \circ \psi)^2 = \left(\theta
\circ \Int \varepsilon(b) \right)^2 = \Int
\varepsilon(\sigma(b)^{-1} b),
$$
it follows that $t := \sigma(b)^{-1} b \in L,$ and clearly
$\sigma(t) = t^{-1}.$ If $\sigma \vert L = \mathrm{id}_L$, then $t =
\pm 1.$ However, if $t = - 1$, then $\theta_b$ is an involution of
type different from that of $\theta$ and $\tau$ (cf.\,\cite{BoI},
Proposition 2.7(3)), and (\ref{E:12}) would be impossible. So, $t =
1$ and $b \in F^{\times},$ as desired. If $\sigma \vert L
\neq \mathrm{id}_L$, then $N_{L/K}(t) = 1,$ and therefore by
Hilbert's Theorem 90, we can write
$$
t = \sigma(b)^{-1} b = \sigma(c)c^{-1} \ \ \text{for some} \ \ c \in
L^{\times}.
$$
Then $\sigma(bc) = bc$ and $\theta_b= \theta_{bc}$. Take $a = bc$.

\vskip2mm

$(ii) \Rightarrow (iii):$ Let $\varphi \colon (A , \theta_a) \to (A
, \tau)$ be an isomorphism of $L$-algebras with involution. Then
$\varphi = \Int h$ for some $h\in A^{\times}$. Equation (\ref{E:HF2}) implies that
$$
\varepsilon(a) = \varepsilon(\sigma(a)) = g^{-1}\tau(\varepsilon(a)) g,
$$
so
$$
\tau(g\varepsilon(a))  = \tau(\varepsilon(a))\tau(g)
=\tau(\varepsilon(a))g=
 g\varepsilon(a),
$$
i.e.,  $g\varepsilon(a)$ is $\tau$-symmetric. Using the equality
$\varphi \circ \theta_a = \tau \circ \varphi$ we obtain that
$$
\Int h \circ \theta_a = \Int h\circ\tau\circ \Int (g\varepsilon(a))
= \tau \circ \Int (\tau(h)^{-1}g\varepsilon(a)) = \tau\circ\Int h.
$$
Therefore, $(g\varepsilon(a))^{-1}\tau(h)h \in L^{\times},$ i.e.,
$\tau(h)h=\lambda g\varepsilon(a) $ for some $\lambda \in
L^{\times}.$ Since $g\varepsilon(a)$ is $\tau$-symmetric, $\lambda$
must lie  in $K^{\times}$. Let $b=a\lambda\in
F^{\times}$. Then  $g\varepsilon(b) = \tau(h)h$.

\vskip2mm

$(iii) \Rightarrow (i):$ Suppose $g\varepsilon(b) = \tau(h)h$ for
some $b\in F^{\times}$ and $h\in A^{\times}$. Set
$\varphi = \Int h.$ Then
$$
\varphi \circ \theta_b = \Int h \circ \tau \circ
\Int(g\varepsilon(b)) = \tau \circ \Int(\tau(h)^{-1}g\varepsilon(b))
= \tau \circ \Int h = \tau\circ \varphi.
$$
It follows that for $\iota = \varphi \circ \varepsilon$ we have
$$
\iota \circ \sigma = \varphi \circ \varepsilon \circ \sigma =
\varphi \circ \theta \circ \varepsilon = \varphi \circ \theta_b
\circ \varepsilon = \tau \circ \varphi \circ \varepsilon = \tau
\circ \iota,
$$
as required.
\end{proof}


\vskip1mm

We conclude this section with the following well-known fact.
\begin{prop}\label{P:G111}
Let $A = M_m(D)$, where $D$ is a central division algebra over $L$ endowed with an
involution $a \mapsto \bar{a}$, and define an
involution $x \mapsto x^*$ of $A$ by $(x_{ij}) \mapsto
(\overline{x_{ji}}).$ Let $\epsilon$ be either $+1$ or $-1$.
For $i=1, 2$, let $Q_i \in A^{\times}$
be such that $Q_i^* = \epsilon Q_i$,
and define involutions $\tau_i$ by $\tau_i(x) = Q^{-1}_i x^*
Q_i.$ Then $(A , \tau_1) \simeq (A , \tau_2)$ as $L$-algebras with
involution if and only if there exist $z \in A^{\times}$ and
$\lambda \in K^{\times}$ (where $K = L^{\tau}$) such that $Q_2 =
\lambda z^* Q_1 z.$
\end{prop}
\begin{proof}
Any $L$-algebra automorphism $\varphi \colon A \longrightarrow A$
is inner, i.e., it is of the form $x \mapsto z^{-1}xz$ for some $z \in A^{\times}.$
Furthermore, a direct computation shows that the condition $\tau_2(\varphi(x)) =
\varphi(\tau_1(x))$, for all $x \in A$, is equivalent to the fact that
$\lambda := (z^*)^{-1} Q_2 z^{-1} Q_1^{-1}$ belongs to $Z(A) = L.$
Then $Q_2 = \lambda z^*Q_1z,$ and applying $^*$ we obtain that
actually $\lambda \in K.$
\end{proof}

We notice that the matrix equation relating $Q_1$ and $Q_2$ says
that the associated  (skew)-hermitian forms are {\it similar,}
i.e., an appropriate scalar multiple of one is equivalent to the other.

\section{Algebras with an involution of the second kind}\label{S:U}

In this section, we will establish a local-global principle for
embedding of fields with an involutive automorphism into simple
algebras with an involution of the second kind, which is assertion (i) of
Theorem A (of the introduction). A partial result (with some extra
conditions) in this direction was obtained earlier in our paper
\cite{PR1}, Proposition A.2, and the argument below is a
modification of the argument given therein. What has not been
previously observed is that the local-global principle {\it fails}
for general commutative \'etale algebras (see Example 4.6
below).
\begin{thm}\label{T:U1}
Let $A $ be a central simple algebra over a global field
$L,$ of dimension $n^2,$ with an involution $\tau$ of the second
kind, $K=L^{\tau}$, and let $E/L$ be a field extension of degree $n$
provided with an involutive automorphism $\sigma$ such that $\tau
\vert L = \sigma \vert L.$ Suppose that for each $v \in V^K$ there
exists an $(L \otimes_K K_v)$-embedding
$$
\iota_v \colon (E \otimes_K K_v , \sigma \otimes \mathrm{id}_{K_v})
\hookrightarrow (A \otimes_K K_v , \tau \otimes \mathrm{id}_{K_v})
$$
of algebras with involutions. Then there exists an $L$-embedding
$$
\iota \colon (E , \sigma) \hookrightarrow (A , \tau)
$$
of algebras with involutions.
\end{thm}
\begin{proof}
First, we observe that the existence of $\iota_v$ for all $v \in
V^K$ implies the existence of an $L_w$-embedding $\varepsilon_w
\colon E \otimes_L L_w \hookrightarrow A \otimes_L L_w,$ for all $w
\in V^L.$ Indeed, fix a $w$ and let $v \in V^K$ be such that $w
\vert v.$ If $L \otimes_K K_v$ is a field, then it coincides with
$L_w,$ and then $\varepsilon_w = \iota_v$ is the required embedding.
On the other hand, if $L \otimes_K K_v$ is not a field, then $v$ has
two extension to $L,$ one of which is $w$ and the other will be
denoted $w'.$ We have
$$L \otimes_K K_v \simeq L_w \times L_{w'} \simeq K_v \times K_v ,$$
and
$$
E \otimes_K K_v \simeq E \otimes_L (L \otimes_K K_v) \simeq (E
\otimes_L L_w) \times (E \otimes_L L_{w'}).
$$
Furthermore,
\begin{equation}\label{E:U50}
A \otimes_K K_v \simeq A \otimes_L (L \otimes_K K_v) \simeq (A
\otimes_L L_w) \times (A \otimes_L L_{w'}).
\end{equation}
It follows that the restriction of $\iota_v$ to the component $E
\otimes_L L_w$ provides the required embedding $\varepsilon_w.$ Now,
by Proposition \ref{P:E2}, the existence of the embeddings
$\varepsilon_w$ for $w \in V^L$ implies the existence of an
$L$-embedding $\varepsilon \colon E \hookrightarrow A,$ which we
will fix.

\vskip1mm

Next, using Proposition \ref{P:G1}, we can find an involution
$\theta$ on $A$ of the form
$$
\theta = \tau \circ \Int g = \Int g^{-1} \circ \tau
$$
that satisfies $\theta(\varepsilon(x)) = \varepsilon(\sigma(x))$ for
all $x \in E.$ Then according to Theorem~\ref{T:G1}, an
$L$-embedding $\iota \colon (E , \sigma) \hookrightarrow (A , \tau)$
as algebras with involutions exists if and only if we can find $a
\in F^{\times},$ where $F = E^{\sigma},$ and $h \in A^{\times}$ so
that
\begin{equation}\label{E:U20}
g = \tau(h)h \varepsilon(a).
\end{equation}
For $v \in V^K,$ the existence of $\iota_v$ implies the existence of
$a_v \in (F \otimes_K K_v)^{\times}$ and  $h_v \in (A \otimes_K
K_v)^{\times}$ such that
\begin{equation}\label{E:U21}
g = \tau(h_v)h_v\varepsilon(a_v)
\end{equation}
(to avoid cumbersome notations, we write $\varepsilon$ and $\tau$
instead of $\varepsilon \otimes \mathrm{id}_{K_v}$ and $\tau \otimes
\mathrm{id}_{K_v}$). Indeed, if $L \otimes_K K_v$ is a field, this
immediately follows from Theorem \ref{T:G1}.

To treat the case where $L \otimes_K K_v$ is not a field, we first
note the following fact that will be used repeatedly: as in
(\ref{E:U50}), we have an isomorphism $A \otimes_K K_v \simeq A_1
\times A_2,$ where $A_1,$ $A_2$ are simple $K_v$-algebras, and
$\tau$ interchanges $A_1$ and $A_2.$ Thus, $A_2$ can be identified
with the opposite algebra $A_1^{\mathrm{op}},$ and moreover, this
identification can be chosen so that $\tau$ corresponds to the
exchange involution $(x_1 , x_2) \mapsto (x_2 , x_1).$ It follows
that any $\tau$-symmetric element in $A \otimes_K K_v$ (i.e.,\:any
element in $A^{\tau} \otimes_K K_v$) can be written in the form
$\tau(h_v)h_v$ for some $h_v \in A \otimes_K K_v.$\footnote{We
note here for future use that the the same argument shows that any
$\tau$-symmetric element in $A \otimes_K K_v$ with reduced norm 1
can be written in the form $\tau(h_v)h_v$ with $h_v \in A \otimes_K
K_v$ of reduced norm 1 - one only needs to observe that the natural
extension $\Nrd_{A \otimes_K K_v/L \otimes_K K_v}$ of the reduced
norm map $\Nrd_{A/L}$ coincides with $(\Nrd_{A_1/K_v} ,
\Nrd_{A_2/K_v})$ in terms of the above identification} In
particular, it follows that (\ref{E:U21}) has a solution with $a_v =
1.$

Taking reduced norms in (\ref{E:U21}), we obtain
\begin{equation}\label{E:U12}
\Nrd_{A/L}(g) = N_{F \otimes_KK_v/K_v}(a_v)N_{L \otimes_K K_v/K_v}(b_v) ,
\end{equation}
where $b_v = \Nrd_{A \otimes_K K_v/L \otimes_K K_v}(h_v).$ We will
now make use of the following.
\begin{prop}\label{P:U1}
Let $L/K$ be an abelian Galois extension of degree $m$ that
satisfies the Hasse norm principle (which is automatically the case
if $L/K$ is cyclic), and $F/K$ be a finite extension linearly
disjoint from $L$ over $K$. Then the pair $F$ and $L$ satisfies the
Hasse multinorm principle over $K,$ i.e.,
\begin{equation}\label{E:MN}
N_{F/K}(J_{F}) N_{L/K}(J_{L}) \cap K^{\times} =
N_{F/K}(F^{\times}) N_{L/K}(L^{\times}),
\end{equation}
where $J_F$ and $J_L$ denote the group of id\`eles  of $F$ and $L$ respectively.
\end{prop}
\begin{proof}
%
%
%
%
%
%
Let $E = FL$. By our assumption, the restriction map
$$
\Ga(E/F) \stackrel{\theta}{\longrightarrow} \Ga(L/K)
$$
is an isomorphism. Using the commutative diagram (cf.\:\cite{ANT},
Ch.\:VII, Proposition 4.3)
$$
\begin{array}{rcc}
J_{F} & \stackrel{\psi_{E/F}}{\longrightarrow} & \Ga(E/F) \\
N_{F/K} \downarrow &  & \downarrow \theta \\
J_K & \stackrel{\psi_{L/K}}{\longrightarrow} & \Ga(L/K),
\end{array}
$$
in which $\psi_{E/F}$ and $\psi_{L/K}$ are the corresponding
Artin maps, we see that $N_{F/K}$ induces an isomorphism
\begin{equation}\label{E:HF1}
J_{F}/F^{\times}N_{E/F}(J_E) \simeq J_K/K^{\times}
N_{L/K}(J_{L}).
\end{equation}
Now, suppose
$$
a = N_{F/K}(x) N_{L/K}(y)
$$
where $a \in K^{\times}$,  $x \in J_{F}$ and $y\in J_L$. Then
$$
N_{F/K}(x) = a N_{L/K}(y)^{-1}.
$$
So, it follows from the isomorphism (\ref{E:HF1}) that $x \in
F^{\times} N_{E/F}(J_E),$ i.e.
$$
x = x' N_{E/F}(z) \ \ \text{with} \ \ x' \in F^{\times}, \ z
\in J_E.
$$
Then
$$
aN_{F/K}(x')^{-1} = N_{L/K}(y)N_{E/K}(z)  =
N_{L/K}(yN_{E/L}(z)) \in N_{L/K}(J_{L}).
$$
Since $L/K$ satisfies the Hasse norm principle, we see that
$$
aN_{F/K}(x')^{-1} = N_{L/K}(y') \ \ \text{for some} \ \ y'
\in L^{\times},
$$
as required.
\end{proof}

\vskip2mm

Continuing with the notations introduced in the previous
proposition, we notice that given $z\in K^{\times},$ for any $v \in
V^K_f$ which is unramified in both $F$ and $L$, and $z$ is a unit in $K_v^{\times}$, $z$
is automatically the norm of a {\it unit}. Since all but finitely many $v
\in V^K_f$ satisfy the above conditions, we see that if for every $v \in V^K$,
$$
z\in N_{F \otimes_K K_v/K_v}((F \otimes_K K_v)^{\times}) N_{L
\otimes_K K_v/K_v}((L \otimes_K K_v)^{\times}),
$$
then actually
$$
z \in N_{F/K}(J_{F}) N_{L/K}(J_{L}).
$$
This remark in conjunction with (\ref{E:U12}) implies that
Proposition \ref{P:U1} can be applied in our situation with
$F = E^{\sigma}$, which yields
the existence of $a \in F^{\times},$ $b \in L^{\times}$ such that
\begin{equation}\label{E:U15}
\Nrd_{A/L}(g) = N_{F/K}(a) N_{L/K}(b) = \Nrd_{A/L}(\varepsilon(a)) N_{L/K}(b).
\end{equation}
We claim that a solution $(a , b)$ to (\ref{E:U15}) can be chosen so
that
\begin{equation}\label{E:U16}
g\varepsilon(a)^{-1} \in \Sigma(v) := \{ \tau(h_v)h_v \: \vert \:
h_v \in (A \otimes_K K_v)^{\times} \}
\end{equation}
and
\begin{equation}\label{E:U16a}
b \in \Theta(v) := \Nrd_{A \otimes_K K_v/L \otimes_K K_v}((A
\otimes_K K_v)^{\times})
\end{equation}
for all $v \in V^K_r$. To see this, we consider the $K$-torus
$$
T = \{ (x , y) \in \mathrm{R}_{F/K}(\mathrm{GL}_1) \times
\mathrm{R}_{L/K}(\mathrm{GL}_1) \ \vert \ N_{F/K}(x) N_{L/K}(y) = 1
\}.
$$
Fix a solution $(a , b)$ to (\ref{E:U15}). Then for $(a_v , b_v =
\Nrd_{A\otimes_KK_v/L\otimes_KK_v}(h_v)),$ where $(a_v , h_v)$ is a solution to (\ref{E:U21}), we
have
$$
t := (a_va^{-1} , b_vb^{-1})_{v \in V^K_r} \in T(V^K_r) := \prod_{v
\in V^K_r} T(K_v).
$$
Since $\Sigma(v) = \Sigma(v)^{-1}$ and $\Theta(v) = \Theta(v)^{-1}$
are open in $(A^{\tau} \otimes_K K_v)^{\times}$ and $(L \otimes_K
K_v)^{\times}$ respectively, the set $\Omega =
\prod_{v \in V^K_r} \Omega(v),$ where
$$
\Omega(v) = \{ (x , y) \in T(K_v) \ \vert \ x \in
\Sigma(v)g\varepsilon(a)^{-1}, \ \ y \in \Theta(v)b^{-1} \},
$$
is an open neighborhood of $t$ in $T(V^K_r).$ However, $T$
has the weak approximation property with respect to $V^K_r$
(cf.\:\cite{PlR}, Proposition 7.8, or \cite{Vo}, \S 11.5). So, $\Omega$
contains an element $(a_0 , b_0)\in T(K).$ Then
$$
\Nrd_{A/L}(g) = N_{F/K}(a_0a) N_{L/K}(b_0b)
$$
and $g\varepsilon (a_0a)^{-1} \in \Sigma(v)$ and $b_0b \in
\Theta(v),$ for all $v \in V^K_r.$ After replacing $a$ with $a_0a$, and $b$ with $b_0b$,
we will assume that $a \in F^{\times}$ and $b
\in L^{\times}$ satisfy (\ref{E:U15}), (\ref{E:U16}) and (\ref{E:U16a}). Then
it follows from Eichler's Norm Theorem (cf.\:\cite{PlR}, Theorem 1.13 and
\S 6.7) that there exists $h_0 \in A^{\times}$ such that
$\Nrd_{A/L}(h_0) = b.$ To complete the argument, we need the
following.
\begin{lemma}\label{L:U10}
Let $\mathscr{S}$ be the variety of $\tau$-symmetric elements in $M
= \mathrm{SL}_{1 , A}.$ If $x \in \mathscr{S}(K)$ is such that $x
\in \Sigma(v) = \{ \tau (h_v)h_v \:  | \:  h_v \in (A\otimes_K K_v
)^{\times} \}$ for all $v \in V^K_r$, then $x = \tau(h)h$ for some
$h \in M(K).$
\end{lemma}
\begin{proof}
We can write $x = \tau(y) y$ for some $y \in M(K_{sep}),$ where
$K_{sep}$ is a separable closure of $K.$ Then $\xi_{\gamma} :=
y\gamma(y)^{-1}$ for $\gamma \in \Ga(K_{sep}/K)$ defines a Galois
1-cocycle $\xi$ with values in $G = \mathrm{SU}(A , \tau).$ It is
enough to show that $\xi$ defines the trivial element of $H^1(K ,
G).$ Indeed, then there exists $z \in G(K_{sep})$ with the property
$$
\xi_{\gamma} = y\gamma(y)^{-1} = z^{-1}\gamma(z) \ \ \text{for all}
\ \gamma \in \Ga(K_{sep}/K).
$$
It follows that $h := zy \in M(K)$, and obviously, $x = \tau(h)h,$
as required. It is known that $H^1(K , G)$ is trivial if $K$ is
either a global function field \cite{Ha} or a totally imaginary
number field (cf.\:\cite{PlR}, \S 6.7), so our assertion follows
immediately. To prove the assertion in the general case, we will use
the Hasse principle for $G,$ i.e., the fact that the map
$$
H^1(K , G) \longrightarrow \prod_{v \in V^K_r} H^1(K_v , G)
$$
is injective (cf.\:\cite{PlR}, Theorem 6.6). So, it is enough to
show that the image of $\xi$ in $H^1(K_v , G)$ is trivial, for all
$v \in V^K_r,$ which, by the argument above, is equivalent to the
fact that $x = \tau(h_v)h_v$ for some $h_v \in M(K_v).$ But if $L
\otimes_K K_v$ is not a field, then according to the observation
made in a footnote above, any $x \in \mathscr{S}(K_v)$ can be
written in the form $\tau(h_v)h_v$ for some $h_v \in M(K_v),$ and
there is nothing to prove. Thus, it remains to consider the case
where $L \otimes_K K_v$ is a field (which, of course, coincides with
$\C$). Let $H = \mathrm{U}(A , \tau).$ The fact that $x \in
\Sigma(v)$ implies that the image of $\xi$ in $H^1(K_v , H)$ is
trivial, and it is enough to show that in this situation, the map
$H^1(K_v , G) \to H^1(K_v , H)$ has trivial kernel. But over $K_v =
\R,$ we have compatible isomorphisms
$$
H \simeq \mathrm{U}(f) \ \ \text{and} \ \ G \simeq \mathrm{SU}(f)
$$
for some nondegenerate hermitian form $f.$ The exact sequence
$$
1 \to \mathrm{SU}(f) \longrightarrow \mathrm{U}(f)
\stackrel{\det}{\longrightarrow} T \to 1,
$$
where $T= \mathrm{R}^{(1)}_{\C/\R}(\mathrm{GL}_1),$ gives rise to
the following exact cohomological sequence
$$
\mathrm{U}(f)({\R}) \stackrel{\det}{\longrightarrow} T({\R})
\longrightarrow H^1(\R , \mathrm{SU}(f)) \longrightarrow H^1(\R ,
\mathrm{U}(f)).
$$
Since the first map is obviously surjective, the third map has
trivial kernel, as required.
\end{proof}

\vskip1mm

We will now complete the proof of Theorem \ref{T:U1}. It follows
from our construction that $x =
\tau(h_0)^{-1}(g\varepsilon(a)^{-1})h_0^{-1}$ satisfies the
assumptions of Lemma \ref{L:U10}. So, it can be written in the form
$\tau(h)h$ for some $h \in A^{\times},$ and therefore the same is
true for $g\varepsilon(a)^{-1},$ yielding the required presentation
(\ref{E:U20}) for $g$.
\end{proof}

\vskip2mm

\noindent {\bf Remarks 4.4.} (1) In the notations of Lemma
\ref{L:U10}, for any $v \in V^K_f,$ we have $H^1(K_v , G) = \{1\},$
so the argument therein yields the following fact: any $x \in
\mathscr{S}(K_v)$ can be written in the form $\tau(h_v)h_v$ for some
$h_v \in (A \otimes_K K_v)^{\times}.$ We will use this observation
in the example below.

\vskip1mm

(2) Using Theorem \ref{T:U1}, it  has been proved in \cite{GG} that
if either $K$ is totally complex, or the degree $n$ of $A$ is odd,
there exists a cyclic Galois extension $F$ of $K$ such that $(F
\otimes_K L , \id_F\otimes\tau)$ embeds in $(A , \tau)$.

\vskip1mm

(3) Some sufficient conditions for the existence of $\iota_v$ at a
particular $v \in V^K$ are given in \cite{PR1}, Propositions A.3 and
A.4. We will use these conditions in the proof of the following
corollary.

\addtocounter{thm}{1}

\vskip2mm

\begin{cor}\label{C:U1}
Let $(A_1 , \tau_1)$ and $(A_2 , \tau_2)$ be two central simple
algebras with involutions of the second kind over a global field
$L.$ Assume that
$$
\dim_L A_1 = \dim_L A_2 =: n^2 \ \ \text{and} \ \ \tau_1 \vert L =
\tau_2 \vert L =: \tau.
$$
Then there exists a field extension $E/L$ of degree $n$ with an
involutive automorphism $\sigma$ satisfying $\sigma(L) = L$ and
$\sigma \vert L = \tau,$ such that $(E , \sigma)$ embeds into $(A_i
, \tau_i)$ as an algebra with involution, for $i = 1 , 2.$
\end{cor}
\begin{proof}
Let $G_i = \mathrm{SU}(A_i , \tau_i),$ and let $V_i$ be the finite
set of all $v \in V^K$ such that $G_i$ is not quasi-split over
$K_v$ (cf.\:\cite{PlR}, Theorem 6.7). Set $V = V_1 \cup V_2,$ and let
$$
S_1 = \{ v \in V \: \vert \: L \otimes_K K_v \simeq K_v \times K_v
\}, \ \ S_2 = V \setminus S_1.
$$
Pick an extension $F/K$ of degree $n$ which is linearly disjoint
from $L$ over $K$ and satisfies the following conditions: $F
\otimes_K K_v$ is a field for $v \in S_1,$ and $F \otimes_K K_v
\simeq K_v^n$ for $v \in S_2.$ Set $E = FL = F \otimes_K L$ and let
$\sigma$ be the involution $\id_F\otimes\tau$ of $E$.
Then it follows from Proposition A.3 (resp., Proposition A.4) in
\cite{PR1} that there exist embeddings $\iota^i_v \colon (E \otimes_K K_v ,
\sigma \otimes \id_{K_v}) \hookrightarrow (A_i \otimes_K K_v ,
\tau_i \otimes \id_{K_v})$ for $v \in S_1$ (resp., $v \in
S_2$) and $i = 1 , 2.$ On the other hand, for $v \notin V$ and
any $i = 1 , 2,$ the existence of $\iota^i_v$ follows from the fact
that $G_i$ is quasi-split over $K$ (cf.\:\cite{PlR}, p.\:340).
Applying Theorem \ref{T:U1}, we obtain the existence of embeddings
$\iota^i \colon (E , \sigma) \hookrightarrow (A_i , \tau_i)$, for $i
= 1 , 2.$
\end{proof}

\vskip2mm

\addtocounter{thm}{1}

We will now construct an example showing that the assertion of
Theorem~\ref{T:U1} does not extend to embeddings of \'etale
algebras.

\vskip2mm

\noindent {\bf Example 4.6.} Let $K$ be a number field. Pick $a \in
K^{\times} \setminus {K^{\times}}^ 2$ so that $a > 0$ in all real
completions of $K,$ and set $L = K(\sqrt{a}).$ Furthermore, pick two
nonarchimedean places $v_1 , v_2$ of $K$ so that $a \in
{K_{v_i}^{\times}}^2$ for $i = 1 , 2,$ and then pick $b \in
K^{\times}$ with the property $b \notin {K_{v_i}^{\times}}^2$ for $i =
1, 2.$ Set
$$
F_1 = K(\sqrt{b}) \ \ , \ \ F_2 = K(\sqrt{ab}),
$$
and let
$$
F = F_1L = F_2L = K(\sqrt{a} , \sqrt{b}).
$$
Let $\sigma_i \in \Ga(F/F_i)$ be the nontrivial automorphism for $i
= 1 , 2;$ notice that both $\sigma_1$ and $\sigma_2$ act
nontrivially on $L.$ Consider the commutative \'etale $L$-algebra $E
= F \times F$ with the involutive automorphism $\sigma = (\sigma_1 ,
\sigma_2);$ clearly, $ E^{\sigma} = F_1 \times F_2.$

Now, let $D_0$ be the quaternion division algebra over $K$ with local
invariant $1/2 \in \Q/\Z$ at $v_1$ and $v_2,$ and $0$ everywhere
else. Then both $F_1$ and $F_2$ are isomorphic to, and
henceforth will be identified with, maximal subfields of $D_0.$ Fix
a basis $1, {\mathbf i}, {\mathbf j}, {\mathbf k}$ of $D_0$
over $K$ such that ${\mathbf i}^2 = \alpha,$ ${\mathbf j}^2 = \beta$
for some $\alpha , \beta \in K^{\times}$, and
$\mathbf{i}\mathbf{j}=\mathbf{k}=-\mathbf{j}\mathbf{i}$. Let
$\delta$ be the standard involution of $D_0,$ and $D_0^+ = K$ and
$D_0^- = K{\mathbf i} + K{\mathbf j} + K{\mathbf k}$ be the spaces
of $\delta$-symmetric and $\delta$-skew-symmetric elements,
respectively. Let $D = D_0 \otimes_K L$ with the involution $\mu =
\delta \otimes \tau_0$, where $\tau_0$ is the nontrivial automorphism
of $L/K,$ and let $D^{\mu}$ be the set of $\mu$-symmetric elements.
\begin{lemma}\label{L:U2}
$\Nrd_{D/L}(D^{\mu}) = K.$
\end{lemma}
\begin{proof} We obviously have
$$
D^{\mu} = D_0^+ + \sqrt{a}D_0^- = K + \sqrt{a}(K{\mathbf i} +
K{\mathbf j} + K{\mathbf k}),
$$
from which it follows that ${\mathrm Nrd}_{D/L}(D^{\mu})$ is the set
of elements represented by $q = x_0^2 - a\alpha x_1^2 - a\beta x_2^2
+ a\alpha\beta x_3^2$ over $K.$ To show that this set coincides with
$K,$ it is enough to show that the quadratic form $q$ is indefinite
at all real places of $K.$ But by our construction, at those places
the algebra $D_0$ splits, so the form $\alpha x_1^2 + \beta x_2^2 -
\alpha\beta x_3^2$ is not negative definite. Since $a > 0,$ the same
is true for the form $a(\alpha x_1^2 + \beta x_2^2 - \alpha\beta
x_3^2),$ and the required fact follows.
\end{proof}

Now, we observe that
$$
F_1 \otimes_K L \simeq F_2 \otimes_K L \simeq F,
$$
and
$$
(F_1 \otimes_K L)^{\mu} = F_2 \ \ \text{and} \ \ (F_2 \otimes_K
L)^{\mu} = F_1.
$$
Thus, $F$ has two embeddings $\nu_{i} \colon F \to D,$ where $i = 1
, 2,$ such that $\nu_i(F)$ is $\mu$-invariant and
$$
\nu_1^{-1} \circ \mu\circ \nu_1 = \sigma_2 \ \ \text{and} \ \
\nu_2^{-1} \circ \mu \circ \nu_2 = \sigma_1.
$$
Consider the embedding
$$
\varepsilon \colon E = F \times F \to M_2(D) =: A, \ \
\varepsilon(x_1 , x_2) = \left(\begin{array}{cc} \nu_1(x_2) & 0 \\ 0
& \nu_2(x_1) \end{array} \right).
$$
It follows from our construction that if we endow $A$ with the
involution $\theta((x_{ij})) = (\mu(x_{ji}))$, then $\varepsilon
\colon (E , \sigma) \to (A , \theta)$ is an embedding of algebras
with involutions.

\vskip2mm

We now need to recall the following, which is actually Exercise 5.2
in \cite{ANT}.
\begin{lemma}\label{L:U1}
Let $F = K(\sqrt{a} , \sqrt{b})$ be a bi-quadratic extension of a
number field $K.$ Assume that for all $v \in V^K,$ the local degree
$[F_v  : K_v]$ is $\leqslant 2.$ Let $K_i = K(\sqrt{a_i})$ for $i =
1, 2, 3,$ be the three quadratic subfields of $F,$ and set
$$
N_i = N_{K_i/K}(K_i^{\times}) \ \ \text{and} \ \ N_i^v = N_{{K_i}_
v/K_v}({K_i} _v^{\times}) \ \text{for} \ v \in V^K.
$$
Then $N_1^v N_2^v N_3^v = K_v^{\times}$ for all $v \in V^K,$ but
$N_1N_2N_3 \neq K^{\times}.$
\end{lemma}
\begin{proof}
For those who did not have a chance to work out all the details in
Exercise 5.2 in \cite{ANT}, we briefly sketch the argument. First,
by our assumption, for any $v \in V^K,$ we have ${K_i}_ v = K_v$ for
at least one $i,$ and therefore $N_1^v N_2^v N_3^v = K_v^{\times}.$
Next, set $S_i = \{ v \in V^K \ \vert \ {K_i}_v = K_v \}.$ Then,
letting $(* , *)_v$ denote the Hilbert symbol over $K_v,$ we can
define the following homomorphism $\varphi \colon K^{\times} \to \{
\pm 1 \},$
$$
\varphi(x) = \prod_{v \in S_1} (a_2 , x)_v \stackrel{1)}{=} \prod_{v
\in S_1} (a_3 , x)_v \stackrel{2)}{=} \prod_{v \in S_2} (a_3 , x)_v
=
$$
$$
= \prod_{v \in S_2} (a_1 , x)_v = \prod_{v \in S_3} (a_1 , x)_v =
\prod_{v \in S_3} (a_2 , x)_v.
$$
We notice that  equality 1) follows from the fact that for $v \in
S_1$ we have $a_2a_3^{-1} \in {K_v^{\times}}^2.$ To prove equality 2),
we observe that by our assumption $V^K = S_1 \cup S_2 \cup S_3,$ so
the product formula for the Hilbert symbol combined with the facts
that $S_1 \cap S_2 \subset S_3$ and $a_3 \in {K_v^{\times}}^2$ for $v
\in S_3,$ yields
$$
1 =\prod_{v \in V^K} (a_3 , x)_v = \prod_{v \in S_1 \cup S_2} (a_3 ,
x)_v = \prod_{v \in S_1} (a_3 , x)_v \cdot \prod_{v \in S_2} (a_3 ,
x)_v,
$$
as required. All other equalities are established similarly. It
follows from the appropriate description of $\varphi$ that
$\varphi(N_i) = 1$ for all $i = 1, 2, 3.$ Thus, $\varphi(N_1N_2N_3)
= 1.$ On the other hand, it follows from Chebotarev's Density
Theorem that one can pick $u_1 \in S_1$ and $u_2 \notin S_1$ so that
$a_2 \notin K_{u_j}^{\times 2}$ for $j = 1 , 2.$ Using Exercise 2.16
in \cite{ANT}\footnote{For the reader's convenience, we recall the
statement of this result, which will be used again in \S \ref{S:O}:
Let $a \in K^{\times},$ and suppose that for each $v \in V^K,$ we
are given $\varepsilon_v \in \{ \pm 1 \}$ so that the following
three conditions are satisfied: (i) $\varepsilon_v = 1$ for all but finitely many
$v;$ (ii) $\prod_v \varepsilon_v = 1;$ (iii) for each $v \in
V^K,$ there exists $x_v \in K_v^{\times}$ such that $(a , x_v)_v =
\varepsilon_v.$ Then there exists $x \in K^{\times}$ such that $(a ,
x)_v = \varepsilon_v$ for all $v.$ }, we can find $x \in K^{\times}$
satisfying
$$
(a_2 , x)_{u_1} = (a_2 , x)_{u_2} = -1 \ \ \text{and} \ \ (a_2 ,
x)_u = 1 \ \text{for all} \ u \in V^K \setminus \{ u_1 , u_2 \}.
$$
Then $\varphi(x) = -1,$ implying that $N_1N_2N_3 \neq K^{\times}.$
\end{proof}

We will assume henceforth that $a , b \in K^{\times}$ are chosen so
that $F = K(\sqrt{a} , \sqrt{b})$ satisfies our previous assumptions
and those of Lemma~\ref{L:U1}, i.e., the local degree $[F_v : K_v]$
is $\leqslant 2$ for all $v \in V^K.$ (Explicit example: $K = \Q,$
$a = 13,$ $b = 17;$ then one can take for $v_1 , v_2$ the $p$-adic places
of $\Q$ corresponding to the primes $p = 3$ and $23$.) According to Lemma
\ref{L:U1},  one can choose $s \in K^{\times}$ so that
\begin{equation}\label{E:U10}
s \notin
N_{K(\sqrt{a})/K}(K(\sqrt{a})^{\times})N_{K(\sqrt{b})/K}(K(\sqrt{b})^{\times})
N_{K(\sqrt{ab})/K}(K(\sqrt{ab})^{\times})
\end{equation}
It follows from Lemma \ref{L:U2} that there exists $g \in
A^{\theta}$ such that $\mathrm{Nrd}_{A/L}(g) = s$ (in fact, we can
choose such a $g$ of the form $\mathrm{diag}(t , 1)$ where $t \in
D^{\mu}$). Consider the involution $\tau = \Int g \circ \theta.$ We
claim that the equation
\begin{equation}\label{E:U11}
g \varepsilon(x) = h \tau(h) \ \ \text{for} \ \ x \in
{(E^{\sigma})}^{\times}, \ \ h \in A^{\times},
\end{equation}
is solvable everywhere locally, but not globally. Then  one can
embed $(E \otimes_K K_v , \sigma \otimes \mathrm{id}_{K_v})$  into
$(A \otimes_K K_v , \tau \otimes \mathrm{id}_{K_v})$ for all $v \in
V^K,$ but one cannot embed $(E, \sigma)$ into $(A , \tau).$

First, suppose (\ref{E:U11}) holds for some $x \in (E^{\sigma})^{\times}$ and
$h \in A^{\times}.$ Since $E^{\sigma} = K(\sqrt{b}) \times
K(\sqrt{ab}),$ taking reduced norms, we obtain
$$
s = \mathrm{Nrd}_{A/L}(g)  \hskip8cm
$$
$$
\ \ \ \ \ \ \ \ \ \ \in
N_{K(\sqrt{a})/K}(K(\sqrt{a})^{\times})N_{K(\sqrt{b})/K}(K(\sqrt{b})^{\times})
N_{K(\sqrt{ab})/K}(K(\sqrt{ab})^{\times}),
$$
which contradicts (\ref{E:U10}).

Now, fix $v \in V^K.$ If $v \in V^K_r$, then by our construction $L
\otimes_K K_v$ is not a field. Then every $\tau$-symmetric element
in $(A \otimes_K K_v)^{\times}$ can be written in the form
$\tau(h_v)h_v$ for some $h_v \in (A \otimes_K K_v)^{\times},$ and
there is nothing to prove. So, assume now that $v \in V^K_f.$ Since
$v$ splits in at least one of the extensions $K(\sqrt{a}),$
$K(\sqrt{b})$ and $K(\sqrt{ab})$, and $E^{\sigma} = K(\sqrt{b}) \times
K(\sqrt{ab}),$ we see that there exits $s_v \in (E^{\sigma}
\otimes_K K_v)^{\times}$ and $t_v \in (L \otimes_K K_v)^{\times}$
such that
$$
\Nrd_{A/L}(g) = N_{E^{\sigma} \otimes_K K_v/K_v}(s_v) N_{L \otimes_K
K_v/K_v}(t_v).
$$
Furthermore, the homomorphism of reduced norm $$\Nrd_{A \otimes_K
K_v/L \otimes_K K_v} \colon (A \otimes_K K_v)^{\times} \to (L
\otimes_K K_v)^{\times}$$ is surjective, so there exists $z_v \in A
\otimes_K K_v$ such that $\Nrd(z_v) = t_v.$ Then $$x =
\tau(z_v)^{-1}g\varepsilon(s_v)^{-1}z_v^{-1}$$ is a $\tau$-symmetric
element in $A \otimes_K K_v$ of reduced norm one. So, using
Remark 4.4(1),  we conclude that $x$ can be written in the form
$\tau(h_v)h_v$ with $h_v \in (A \otimes_K K_v)^{\times},$ and then
the same is true for $g\varepsilon(a_v)^{-1},$ yielding a local
solution to (\ref{E:U11}) at $v.$

\vskip3mm

\noindent {\bf Remark 4.9.} It should be pointed out that the proof of the local-global
principle for embeddings of fields with involution in a central simple algebra with an involution of the second kind (Theorem \ref{T:U1}) depends in a very essential
way on the multinorm principle (i.e., (\ref{E:MN})). Proposition
\ref{P:U1} describes one situation in which this principle holds;
some other sufficient conditions are given in Proposition 6.11 of
\cite{PlR}. In fact, we are not aware of any examples where the
multinorm principle (for two fields) fails, and it is probably
safe to conjecture that it always holds if one of the fields
satisfies the usual Hasse norm principle and the extensions are linearly disjoint over $K$.
On the other hand, Lemma
\ref{L:U1} demonstrates that the multinorm principle may fail for
three fields, even when all the fields are quadratic extensions. It
would be interesting to complete the investigation of the multinorm
principle, and in particular, provide an explicit computation of the
obstruction, at least in the case where all fields are Galois
extensions.

After a preliminary version of this paper was circulated, J-L.\,Colliot-Th\'el\`ene informed us about an unpublished joint work of his with J-J.\,Sansuc in which they gave two proofs of a multinorm principle for a pair of extensions, one of which is cyclic.

\vskip5mm

 {\it In the remainder of this paper, we will
work exclusively with simple algebras $A$  endowed with an involution $\tau$
of the {\rm first kind}. The center of $A$, which is fixed point-wise under $\tau$,
will be denoted $K$ (instead of $L$) and  will be assumed to
be a global field of characteristic $\ne 2$. $E$ will be a commutative
\'etale algebra of dimension $n=\sqrt{\dim A}$ equipped with an involution $\sigma$.}

\vskip5mm

\section{Algebras with a symplectic involution}\label{S:Sym}

In this section, $A$ will denote a central simple
$K$-algebra, of dimension $n^2,$ with a {\it symplectic} involution
$\tau$ (then, of course, $n$ is necessarily even). Our goal is to
prove the local-global principle for embedding of an $n$-dimensional
commutative \'etale $K$-algebra $E$ given with an involutive
$K$-automorphism $\sigma$ (Corollary \ref{C:Sym1}). In fact, in this
case one has the following more convenient criterion for the
existence of an embedding.
\begin{thm}\label{T:Sym1}
With notations as above, assume that there exists an embedding
$\varepsilon \colon E \hookrightarrow A$ as algebras without
involutions, and that for each {\rm real} $v \in V^K$ there exists a
$K_v$-embedding
$$
\iota_v \colon (E \otimes_K K_v , \sigma \otimes \id_{K_v})
\hookrightarrow (A \otimes_K K_v , \tau \otimes \mathrm{id}_{K_v})
$$
of algebras with involutions. Then there exists a $K$-embedding
$$
\iota \colon (E, \sigma) \hookrightarrow (A , \tau)
$$
of algebras with involutions.
\end{thm}
The proof relies on the following lemma which is analogous to
Lemma~\ref{L:U10}. We will denote the involution $\tau \otimes \id_{K_v}$ of
$A\otimes_KK_v$ simply by $\tau$ in the following lemma and in the proof of Theorem \ref{T:Sym1}.
\begin{lemma}\label{L:Sym1}
Let $x \in A^{\times}$ be a $\tau$-symmetric element. Assume that
for every real $v \in V^K,$ there is $h_v \in (A \otimes_K
K_v)^{\times}$ such that $x = \tau(h_v)h_v.$ Then there is $h \in
A^{\times}$ such that $x = \tau(h)h.$
\end{lemma}
\begin{proof}
Since $\tau$ is symplectic, $G = \mathrm{U}(A , \tau) =
\mathrm{SU}(A , \tau)$ is a form of $\mathrm{Sp}_{n},$ hence it is
connected, absolutely almost simple and simply connected. This
implies that the map
$$
H^1(K , G) \stackrel{\rho}{\longrightarrow} \prod_{v \in V^K_r}
H^1(K_v , G)
$$
is bijective (cf.\,\cite{PlR}, Theorem 6.6, for number fields, and
\cite{Ha} for global fields of positive characteristic). Let
$K_{sep}$ be a fixed separable closure of $K$. Pick $y \in (A
\otimes_K K_{sep})^{\times}$ so that $x = \tau(y)y.$ Then the map
$$
\gamma \mapsto \xi_{\gamma} := y \gamma(y)^{-1}, \ \ \gamma \in
\Ga(K_{sep}/K),
$$
is a Galois 1-cocycle with values in $G.$ The fact that $x =
\tau(h_v)h_v$, with $h_v \in (A \otimes_K K_v)^{\times},$ for each
$v \in V^K_r,$ means that the corresponding cohomology class lies in
the kernel of $\rho.$ It follows from the injectivity of $\rho$ that
the class is trivial, i.e., there exist $z \in G(K_{sep})$ such that
$$
\xi_{\gamma} = y\gamma(y)^{-1} = z^{-1}\gamma(z) \ \ \text{for all}
\ \ \gamma \in \Ga(K_{sep}/K).
$$
Then $h := zy \in A^{\times}$ and $x = \tau(h)h,$ as required.
\end{proof}

\vskip2mm

{\it Proof of Theorem \ref{T:Sym1}.} By Proposition \ref{P:G1},
there exists an involution $\theta = \tau \circ \mathrm{Int}\: g$ on
$A,$ where $g \in A^{\times}$ is  $\tau$-symmetric, such that
$\varepsilon \colon (E , \sigma) \hookrightarrow (A , \theta)$ is an
embedding of algebras with involutions. Set $F = E^{\sigma}.$ It
follows from our assumptions and the equivalence ({\it i})
$\Leftrightarrow$ ({\it iii}) in Theorem \ref{T:G1} that for each $v
\in V^K_r$ there exists $b_v \in (F \otimes_K K_v)^{\times}$ such
that
$$
g\varepsilon_v(b_v) = \tau(h_v)h_v \ \ \text{for some} \ \ h_v \in
(A \otimes_K K_v)^{\times}.
$$
Since the subgroup $({F \otimes_K K_v)^{\times}}^{2} \subset (F
\otimes_K K_v)^{\times}$ is open, by weak approximation, there
exists $b \in F^{\times}$ such that
$$
b = b_vt_v^2 \ \ \text{with} \ \ t_v \in (F \otimes_K K_v)^{\times}
$$
for each $v \in V^K_r.$ Using the facts that $t_v$ is
$\sigma_v$-symmetric and that $\varepsilon$ intertwines $\sigma$ and
$\theta,$ one finds that $g\varepsilon_v(t_v) =
\tau(\varepsilon_v(t_v))g,$ so
$$
g\varepsilon(b) = \tau(\varepsilon_v(t_v))g\varepsilon_v(b_v)
\varepsilon_v(t_v) =
\tau(h_v\varepsilon_v(t_v))(h_v\varepsilon_v(t_v)).
$$
Then by Lemma \ref{L:Sym1}, we have $g\varepsilon(b) = \tau(h)h$ for
some $h \in A^{\times},$ and invoking Theorem \ref{T:G1}, we see
that there is an embedding $\iota \colon (E , \sigma)
\hookrightarrow (A , \tau).$

\vskip2mm

\begin{cor}\label{C:Sym1}
Let $A$ and $E$ be as above and assume that for every $v \in V^K$
there is a $K_v$-embedding
$$
\iota_v \colon (E \otimes_K K_v , \sigma \otimes \mathrm{id}_{K_v})
\hookrightarrow (A \otimes_K K_v , \tau \otimes \mathrm{id}_{K_v})
$$
of algebras with involutions. Then there exists a $K$-embedding
$$
\iota \colon (E , \sigma) \hookrightarrow (A , \tau)
$$
of algebras with involutions.
\end{cor}

Indeed, in view of Proposition \ref{P:E2}, the existence of
$\iota_v$ for all $v \in V^K$ implies the existence of an embedding
$\varepsilon \colon E \hookrightarrow A$ as algebras without
involutions.


\vskip5mm

\section{Algebras with orthogonal involutions: nonsplit case}\label{S:O}

Let $A$ be a central simple algebra over a global field
$K$ of characteristic $\ne 2$, of dimension $n^2,$ endowed with an involution $\tau$ of the
first kind. Then, if $A \simeq M_m(D),$ with $D$ a division
algebra, then the class $[D] \in \Br(K)$ has exponent $\leqslant 2,$
and therefore either $D = K,$ or $D$ is a quaternion central division
algebra over $K$ (cf.\:\cite{Pie}, \S 18.6). Thus, either $A =
M_n(K)$, or $A = M_m(D)$, where $D$ is a quaternion central division algebra
over $K$, and $n=2m$. We will refer to the first possibility as the {\it split
case,} and to the second as the {\it nonsplit case.} Henceforth, we
will work only with {\it orthogonal}\, involutions, and in this
section will focus on the nonsplit case. Thus, $n$ will be even
throughout the section, and $m = n/2$.

Now, let $E$ be an $n$-dimensional commutative \'etale
$K$-algebra given with a $K$-involution $\sigma$ such that $F =
E^{\sigma}$ is of dimension $m$ (so
(\ref{E:I1}) of  \S \ref{S:I} holds). Then, according to Proposition
\ref{P:E12} we can identify $E$ with $F[x]/(x^2 - d)$ for some $d
\in F^{\times}$ so that $\sigma$ is defined by $x \mapsto -x.$
Theorem \ref{T:O-101} below (which implies assertion (iii) of
Theorem A of the introduction) is formulated for the case where $F$
is a field extension of $K$ and $m$ is odd, however most of our
considerations apply to a much more general situation (cf., in
particular, Theorem \ref{T:O-102}). So, we will assume that $F =
\prod_{j = 1}^r F_j$, $F_j$ a separable field
extension of $K$, and in terms of this decomposition the element $d \in
F^{\times}$ that defines $E$ is written as $d = (d_1, \ldots ,
d_r).$
\begin{thm}\label{T:O-101}
In the above notations, assume that $F$ is a field extension of $K$ of
degree $m$, and $m$ is odd.  If for every $v \in V^K$ there exists a
$K_v$-embedding
$$
\iota_v \colon (E \otimes_K K_v , \sigma \otimes \id_{K_v})
\hookrightarrow (A \otimes_K K_v , \tau \otimes \id_{K_v}),
$$
then there exists a $K$-embedding $\iota \colon (E , \sigma)
\hookrightarrow (A , \tau).$
\end{thm}

\vskip2mm

{\bf Some facts about Clifford algebras.} The main difficulty in the
proof of Theorem \ref{T:O-101} is that orthogonal involutions on $A
= M_m(D),$ where $D$ is a quaternion division algebra, correspond to
(the similarity classes of) $m$-dimensional skew-hermitian forms
(with respect to the standard involution on $D$), and the Hasse
principle for (the equivalence of) such forms generally fails
(cf.\:\cite{K}, \S 5.11 or  \cite{Sch}, Ch.\:10, \S 4). However, one can
still use local-global considerations via an analysis of the associated
Clifford algebras. We refer the reader to \cite{BoI}, Ch.\:II, \S 8B, for
the notion and the structure of the Clifford
algebra $C(A , \nu)$ associated to a simple algebra $A$ with an
involution $\nu.$

\vskip1mm

We will crucially use a result of Lewis and Tignol \cite{LT}
which asserts that {\it for two orthogonal involutions $\tau_1$ and
$\tau_2$ of $A$ as above, $(A , \tau_1) \simeq (A , \tau_2)$}
(that is, $\tau_1$ and $\tau_2$ are conjugate in the
terminology of \cite{LT}) {\it if and only if they have the same
signature at every real place $v$ of $K$ (i.e., $(A \otimes_K K_v ,
\tau_1 \otimes \mathrm{id}_{K_v}) \simeq (A \otimes_K K_v , \tau_2
\otimes \mathrm{id}_{K_v})$), and the Clifford algebras $C(A ,
\tau_1)$ and $C(A , \tau_2)$ are $K$-isomorphic}. (This result
follows from Theorems A and B (see also Proposition 11) of \cite{LT}
since for a global field $K,$ the
fundamental ideal $I(K)$ of the Witt ring $W(K)$ has the property
that $I(K)^3$ (which is commonly denoted by $I^3(K)$ in the literature) is torsion-free,
and it is $\{ 0\}$ if $K$ does not embed in $\R$, cf., for example, \cite{Sch}, Theorem
14.6 in Ch.\:2 together with Corollary 6.6(vi) in
Ch.\:6.)

\vskip1mm

Another ingredient is the computation of classes in the Brauer group
corresponding to certain Clifford algebras. To formulate these
results, we need to make some preliminary remarks. If $\cE =
\prod_{j = 1}^r \cE_j$ is a commutative \'etale algebra over a field $\cK,$
where the $\cE_j$'s are finite separable field extensions of $\cK,$ then
$\Br(\cE)$ is defined to be $\bigoplus_{j = 1}^r \Br(\cE_j).$
Furthermore, the restriction and corestriction maps are defined by
$$
\Res_{\cE/\cK} \colon \Br(\cK) \to \Br(\cE), \ \ \alpha \mapsto
(\Res_{\cE_1/\cK}(\alpha), \ldots , \Res_{\cE_r/\cK}(\alpha)),
$$
and
$$
\Cor_{\cE/\cK} \colon \Br(\cE) \to \Br(\cK), \ \ (\alpha_1, \ldots ,
\alpha_r) \mapsto \Cor_{\cE_1/\cK}(\alpha_1) + \cdots +
\Cor_{\cE_r/\cK}(\alpha_r).
$$
For $a = (a_1, \ldots , a_r), b = (b_1, \ldots , b_r) \in
\cE^{\times},$ we define
$$
(a , b)_{\cE} = ((a_1 , b_1)_{\cE_1}, \ldots , (a_r , b_r)_{\cE_r})
\in \Br(\cE),
$$
where $(a_j , b_j)_{\cE_j}$ is the class in $\Br(\cE_j)$ of the
quaternion $\cE_j$-algebra defined by the pair $a_j , b_j.$ As
usual, if $\cE$ is a local field, then we identify $\Br(\cE)_2$ with
$\{ \pm 1 \},$ which makes $(a , b)_{\cE}$ into the Hilbert symbol.
(If $\cF$ is a global field and $v \in V^{\cF}$, then instead of
$(\cdot , \cdot)_{\cF_v}$ we will occasionally write $(\cdot ,
\cdot)_v$ if this is not likely to lead to  confusion.)
We note that if $\cK$ is a local field and $\cF$ is a quadratic field extension
of $\cK$, then $\Res_{\cF/\cK}(\Br(\cK)_2) = 0$ (cf.\:\cite{ANT}, Theorem 1.3 in Ch.\:VI).

\vskip1mm

Let now $A$ be a central simple $K$-algebra with an
orthogonal involution $\nu.$ Then the center $Z(C(A , \nu))$ of the
corresponding Clifford algebra $C(A, \nu)$ is a quadratic \'etale $K$-algebra
(cf.\:\cite{BoI}, Ch.\:II, Theorem 8.10), i.e., either a (separable) quadratic
field extension of $K$, or $K\times K  .$ Moreover, $C(A , \nu)$ is a
``simple" $Z(C(A , \nu))$-algebra, which in the case $Z(C(A , \nu))
= K \times K$ means that $C(A , \nu) = C_1 \times C_2$, where $C_1$
and $C_2$ are simple $K$-algebras. In all cases, one can consider
the corresponding class $[C(A , \nu)] \in \Br(Z(C(A , \nu))).$ Now,
fix a quadratic \'etale $K$-algebra $Z,$ and suppose that there
exists a $K$-isomorphism $\phi \colon Z \to Z(C(A , \nu)).$ Then
one can consider the simple $Z$-algebra $C(A , \nu, \phi)$ obtained
from $C(A , \nu)$ by change of scalars using $\phi,$ and also the
corresponding class $[C(A , \nu, \phi)] \in \Br(Z).$ Let
$\overline{\phi} \colon Z \to Z(C(A , \nu))$ be the other
$K$-isomorphism. Then
\begin{equation}\label{E:O-150}
[C(A, \nu, \overline{\phi})] - [C(A, \nu , \phi)] = \Res_{Z/K}([A])
\end{equation}
(cf.\:\cite{BoI}, (9.9) and Proposition 1.10). It follows that {\it if $\nu_1$ and
$\nu_2$ are two orthogonal involutions of $A$ such that the centers
of $C(A , \nu_i)$ are isomorphic to $Z$ for $i = 1 , 2,$ then $C(A ,
\nu_1) \simeq C(A , \nu_2)$ if and only if for some (equivalently,
any) isomorphisms $\phi_i \colon Z \to Z(C(A , \nu_i))$, one of the
following two conditions holds:
\vskip3mm

\hskip1cm$[C(A, \nu_1, \phi_1)] = [C(A, \nu_2, \phi_2)]$}

\vskip2mm

\noindent {or}

\hskip1cm $[C(A, \nu_1, \phi_1)] = [C(A, \nu_2, \phi_2)] +
\Res_{Z/K}([A]).$

\vskip3mm

After these recollections, we are ready to embark on our investigation of
the local-global principle in the situation described prior to Theorem
\ref{T:O-101}. First, we observe that the existence of
$K_v$-embeddings $\iota_v$, for all $v \in V^K$, as in the statement
of Theorem \ref{T:O-101} implies that

\vskip2mm

\noindent $\bullet$ \ \parbox[t]{11.5cm}{\it there exists a
$K$-embedding $\varepsilon \colon E \hookrightarrow A$ which may or
may not respect involutions.}

\vskip2mm

\noindent Next, using Proposition \ref{P:G1}, we can construct an
involution $\theta$ of $A$ for which (\ref{E:15}) holds. For $a \in
F^{\times},$ we let $\theta_a$ denote the involution $\theta \circ
\Int \varepsilon(a)$ (then
(\ref{E:15}), with $\theta$ replaced by $\theta_a$, holds). According to Theorem \ref{T:G1}, the existence of
$\iota_v$ is equivalent to the existence of $a_v \in (F \otimes_K
K_v)^{\times}$ such that
\begin{equation}\label{E:O-151}
(A \otimes_K K_v , (\theta \otimes \mathrm{id}_{K_v})_{a_v}) \simeq
(A \otimes_K K_v , \tau \otimes \mathrm{id}_{K_v}).
\end{equation}
We now observe that the centers of the Clifford algebras
$C(A \otimes_K K_v , (\theta\otimes \mathrm{id}_{K_v})_{a_v})$ and
$C(A \otimes_K K_v , \theta\otimes \mathrm{id}_{K_v})$ are isomorphic - this follows from the
description of the center given in \cite{BoI}, Theorem 8.10, the
definition of the discriminant of an orthogonal involution, {\it
loc.\:cit.,} \S 7A, and the fact that
$$
\mathrm{Nrd}_{A \otimes_K K_v/K_v}(a_v) = N_{E \otimes_K
K_v/K_v}(a_v) = N_{F \otimes_K K_v/K_v}(a_v)^2 \in {K_v^{\times}}^2,
$$
from which we deduce that
$$
Z(C(A , \theta)) \otimes_K K_v \simeq Z(C(A \otimes_K K_v , (\theta
\otimes \mathrm{id}_{K_v})_{a_v})) \simeq Z(C(A , \tau)) \otimes_K
K_v
$$
for all $v \in V^K.$ Using Tchebotarev's Density Theorem, we
conclude that

\vskip2mm

\noindent $\bullet$ $Z(C(A , \theta)) \simeq Z(C(A , \tau)).$

\vskip2mm

\noindent We will denote this quadratic \'etale $K$-algebra
by $Z,$ and fix isomorphisms $\phi \colon Z \to
Z(C(A, \theta))$ and $\psi \colon Z \to Z(C(A , \tau)).$ A
fundamental role in our analysis is played by the following
computation of the class of the Clifford algebra $C(A, \theta_a)$ valid over an arbitrary field $K$ of characteristic $\ne 2$
(cf.\:\cite{BC-KM}, Proposition 5.3):
\begin{equation}\label{E:O-155}
[C(A, \theta_a, \phi_a)] = [C(A, \theta, \phi)]  + \Res_{Z/K}
\Cor_{F/K}((a , d)_F).
\end{equation}
In our argument, we will not need the precise description of the
isomorphism $\phi_a$ involved in this equation, the only property
that will be used is that $\phi_a$ depends only on the coset
$aN_{E/F}(E^{\times}) \in F^{\times}/N_{E/F}(E^{\times})$,
cf.\:\cite{BC-KM}, p.\:99;  in particular, $\phi_a = \phi$ if $a \in
{F^{\times}}^{ 2}.$

\vskip1mm

According to Theorem \ref{T:G1}, the existence of $\iota: (E,\sigma)\hookrightarrow (A,\tau)$ is
equivalent to the existence of an $a \in F^{\times}$ such that $(A ,
\theta_a) \simeq (A , \tau),$ and we are now in a position to prove
the following local-global principle for that.
\begin{prop}\label{P:O-101}
Suppose that for each place $v \in V^K$ one can choose an element
$a_v \in (F \otimes_K K_v)^{\times}$ so that the following
conditions are satisfied:

\vskip3mm

\noindent {\rm (a)} $(A \otimes_K K_v , (\theta \otimes
\id_{K_v})_{a_v}) \simeq (A \otimes_K K_v , \tau \otimes \id_{K_v})$
for all $v \in V^K_r;$

\vskip2mm

\noindent {\rm (b)} \parbox[t]{12cm}{one of the following two families of
equalities in $\Br(Z \otimes_K K_v):$
$$
[C(A \otimes_K K_v , (\theta \otimes \id_{K_v})_{a_v}, \phi_{a_v})]
= [C(A, \tau, \psi) \otimes_K K_v] \hskip2.5cm
$$
and
$$
[C(A \otimes_K K_v , (\theta \otimes \id_{K_v})_{a_v}, \phi_{a_v})]=
[C(A , \tau, \psi) \otimes_K K_v]  \hskip2.5cm
$$
$$
\hskip7cm + \Res_{Z \otimes_K K_v/K_v}[A \otimes_K K_v],
$$
\vskip2mm
holds for all $v \in V^K.$}

\vskip3mm

\noindent Assume also that the following condition holds:

\vskip2mm

\noindent {$(*)$} \parbox[t]{12cm}{for any finite subset $V$ of  $V^K$,
there exists $v_0 \in V^K \setminus V$ such that for $j\leqslant r$, if $d_j \notin {F_j^{\times}}^{2},$
then $d_j\notin {(F_j\otimes_K K_{v_0})^{\times}}^{2}$,  and moreover, $Z \otimes_K
K_{v_0}$ is a field if $Z$ is a field.}

\vskip3mm

\noindent Then there exists an $a \in F^{\times}$ such that $(A ,
\theta_a) \simeq (A , \tau).$ Furthermore, condition~{$(*)$} holds
automatically if $F/K$ is a field extension of odd degree.
\end{prop}
\vskip2mm

For the proof of this proposition, we need the following two lemmas
about the Hilbert symbol. (In essence, these lemmas
are well-known, but we have not been able to locate suitable
references for them.)
\begin{lemma}\label{L:HS}
Let $\cF$ be a global field of characteristic $\neq 2,$ and $t \in
\cF^{\times}.$ Suppose that for each $v \in V^{\cF}$ we are given
$\alpha_v \in \{ \pm 1 \}$ and $s_v \in \cF_v^{\times}$ so that
$(s_v , t)_v = \alpha_v$ for all $v \in V^{\cF},$  $\alpha_v = 1$
for all but finitely many  $v \in V^{\cF},$ and $\displaystyle \prod_{v \in
V^{\cF}} \alpha_v = 1$ (here $(\cdot , \cdot)_v$ denotes the Hilbert
symbol on $\cF_v$). Then for any finite subset $\cS$ of
$V^{\cF},$ there exists $s \in \cF^{\times}$ such that $(s , t)_v =
\alpha_v$ for all $v \in V^{\cF}$, and $s \in s_v{\cF_v^{\times}}^{ 2}$
for all $v \in \cS.$
\end{lemma}
\begin{proof}
The existence of $s_0 \in \cF^{\times}$ satisfying $(s_0 , t)_v =
\alpha_v$ for all $v \in V^{\cF}$ follows from the result described
in the footnote in the proof of Lemma \ref{L:U1}.
So, we will only indicate how to modify $s_0$ so that the resulting
$s$ would also satisfy the additional condition $s \in
s_v{\cF_v^{\times}}^{2}$ for $v \in \cS.$ Let $\cE = \cF(\sqrt{t})$
and $\cE_v = \cF_v(\sqrt{t})$ for $v \in V^{\cF},$ and consider the
corresponding norm groups
$$
N = N_{\cE/\cF}(\cE^{\times}) \ \ , \ \ N_v = N_{\cE \otimes_{\cF}
\cF_v/\cF_v}((\cE \otimes_{\cF} \cF_v)^{\times}) =
N_{\cE_v/\cF_v}(\cE_v^{\times}).
$$
It follows from the weak approximation property that $N$ is dense in
$\prod_{v \in \cS} N_v,$ and therefore,
\begin{equation}\label{E:HS}
\prod_{v \in \cS} N_v = N \cdot \left(\prod_{v \in \cS}
{\cF_v^{\times}}^{2} \right)
\end{equation}
Since $(s_0 , t)_v = (s_v , t)_v$ for all $v \in \cS,$ we see that
$(s_0s_v^{-1})_{v \in \cS} \in \prod_{v \in \cS} N_v.$ So by
(\ref{E:HS}), there exists $z \in N$ such that $s_0s_v^{-1}z^{-1}
\in {\cF_v^{\times}}^{2}$ for all $v \in \cS.$ Then, for $s = s_0z^{-1}$

$$
(s , t)_v = (s_0 , t)_v = \alpha_v \ \  \text{for all} \ \  v \in
V^{\cF},
$$
and $s \in s_v{\cF_v^{\times}}^{2},$ as required.
\end{proof}

\begin{lemma}\label{L:HS1}
Let $\cF = \prod_{j = 1}^r \cF_j$ be a commutative \'etale
algebra over a global field $\cK,$ and $t = (t_1, \ldots , t_r) \in
\cF^{\times}.$ For $v\in V^{\cK}$, let $\cF_v
=\cF\otimes_{\cK}\cK_v$. Suppose we are given a finite subset $\cS
\subset V^{\cK},$ and for each $v \in \cS$, an element $s_v \in \cF^{\times}_v.$ Furthermore, let $v_0 \in V^{\cK}
\setminus \cS$ be such that for each $j\leqslant r$ with $t_j \notin
{\cF_j^{\times}}^{2}$, we have $t_j \notin {(\cF_j \otimes_{\cK}
\cK_{v_0})^{\times}}^2.$ Then there exists $s \in \cF^{\times}$ such
that $ss_v^{-1} \in {\cF_v^{\times}}^{ 2}$ for
all $v \in \cS$, and  $(s , t)_{\cF_v} = 1$ for all $v \in V^{\cK}
\setminus (\cS \cup \{ v_0 \}).$
\end{lemma}
\begin{proof}
It is enough to consider the case where $\cF$ is a field and $t
\notin {\cF^{\times}}^{2}$ (indeed, if $t \in {\cF^{\times}}^{ 2}$, then
everything boils down to proving the existence of an $s \in
\cF^{\times}$ such that $s \in s_v{\cF_v^{\times}}^{2}$ for all $v \in
\cS,$ which is obvious). We now define $\alpha_w \in \{ \pm 1 \}$
for all $w \in V^{\cF}$ as follows. For $v \in V^{\cK},$ we let
$w^{(1)}, \ldots , w^{(\ell_v)}$ denote all the extensions of $v$ to
$\cF.$ Then we have
$$
\cF_v=\cF \otimes_{\cK} \cK_{v} = \prod_{k = 1}^{\ell_v} \cF_{w^{(k)}}.
$$
In particular, for $v \in \cS,$ in terms of this decomposition, we write
$$
s_v = (s_{w^{(1)}}, \ldots , s_{w^{(\ell_v)}}),
$$
and we then set $\alpha_{w^{(k)}} = (s_{w^{(k)}} ,
t)_{\cF_{w^{(k)}}}$ for $k \leqslant \ell_v.$ Furthermore, if $w
\in V^{\cF}$ lies over $v \in V^{\cK} \setminus (\cS \cup \{ v_0
\}),$ we set $\alpha_w = 1.$ Finally, if $w_0^{(1)}, \ldots ,
w_0^{({\ell_0})}$ are the extensions of $v_0,$ then by our assumption,
there exists $k_0\leqslant \ell_0$ such that $t \notin
{\cF_{w_0^{(k_0)}}^{\times^2}}$. We then set $\alpha_{w_0^{(k)}} = 1$
for $k \neq k_0,$ and let $\alpha_{w_0^{(k_0)}} = \prod_{w \neq
w_0^{(k_0)}} \alpha_w$ where the product is taken over all $w \in
V^{\cF} \setminus \{ w_0^{(k_0)} \}$ (notice that the $\alpha_w$'s
for all these places have already been defined). Then $\prod_{w \in
V^{\cF}} \alpha_w = 1,$ and for each $w \in V^{\cF}$, there exists
$a_w \in \cF_w^{\times}$ such that $(a_w , t)_{\cF_w} = \alpha_w:$
indeed, if $w \vert v$, where $v \in \cS$, then one takes for $a_w$
the $w$-component of $s_v;$ for any $w \neq w_0^{(k_0)}$ lying over
$v \in V^{\cK} \setminus \cS$ we can takes $a_w = 1,$ and finally,
such $a_w$ exists for $w = w_0^{(k_0)}$ because $t \notin
{\cF_w^{\times}}^{2}.$ Now, our claim follows from Lemma~\ref{L:HS}.
\end{proof}

\vskip2mm

\noindent {\it Proof of  Proposition \ref{P:O-101}.} Let
$$
S_1 = \{ v \in V^K_f \ \vert \ A \otimes_K K_v \not\simeq M_n(K) \}
\cup V^K_r,
$$
$$
S_2 = \{ v \in V^K  \vert \ [C(A, \theta, \phi) \otimes_K K_v] \neq
[C(A, \tau, \psi) \otimes_K K_v] \ \ \text{in} \ \ \Br(Z \otimes_K
K_v) \},
$$
and $S = S_1 \cup S_2.$ Using $(*)$ for $V=S$, we can find $v_0 \in V^K
\setminus S$ with the properties described therein, and then it
follows from Lemma \ref{L:HS1} that there exists an $a \in F^{\times}$
such that
$$
aa_v^{-1} \in {F_v^{\times}}^2\ \ \text{for all} \ \ v
\in S \ \ \text{and} \  \ (a , d)_{F_v} = 1 \ \ \text{for all} \ \ v
\in V^K \setminus (S \cup \{ v_0 \}),
$$
where $F_v =F\otimes_K K_v$. We claim that $a$ is as required, i.e.,
\begin{equation}\label{E:O-101}
(A , \theta_a) \simeq (A , \tau) \ \ \text{as}\ K\text{-algebras
with involution.}
\end{equation}
According to the result of Lewis and Tignol mentioned above, to
establish (\ref{E:O-101}), it is enough to show that $\theta_a$ and
$\tau$ have the same signature at every real places of $K,$ i.e.,
\begin{equation}\label{E:O-102}
(A \otimes_K K_v , \theta_a \otimes \id_{K_v}) \simeq (A \otimes_K
K_v , \tau \otimes \id_{K_v}) \ \ \text{for all} \ \ v \in V^K_r,
\end{equation}
and
\begin{equation}\label{E:O-103}
C(A , \theta_a) \simeq C(A , \tau) \ \ \text{as}\ K\text{-algebras.}
\end{equation}
We notice that (\ref{E:O-102}) immediately follows from condition
(a) in the statement of the proposition and the fact that $aa_v^{-1}
\in {(F \otimes_K K_v)^{\times}}^{2}$ for all $v \in V^K_r.$ To
prove (\ref{E:O-103}), we set $\psi_0 = \psi$ if the first family of
equalities in condition (b) holds, and $\psi_0 = \overline{\psi},$
the other isomorphism between $Z$ and $Z(C(A , \tau)),$ if the
second family of equalities in condition (b) hold.  Then it follows
from (\ref{E:O-150}) that
\begin{equation}\label{E:O-104}
[C(A \otimes_K K_v , (\theta \otimes \id_{K_v})_{a_v}, \phi_{a_v})]
= [C(A, \tau, \psi_0) \otimes_K K_v] \ \ \text{for all} \ \ v \in
V^K.
\end{equation}
We now recall that by our construction, $v_0$ has the property that
if $Z/K$ is a quadratic field extension, then so is $Z \otimes_K
K_{v_0}/K_{v_0},$ which implies that the map of the Brauer groups
$$
\Br(Z) \longrightarrow \bigoplus_{v \neq v_0} \Br(Z \otimes_K K_v)
$$
is injective. So, to prove that $[C(A , \theta_a , \phi_a)] = [C(A,
\tau, \psi_0)]$ in  $\Br(Z),$ which will immediately yield
(\ref{E:O-103}), it is enough to show that
\begin{equation}\label{E:O-105}
[C(A , \theta_a , \phi_a) \otimes_K K_v] = [C(A, \tau, \psi_0)
\otimes_K K_v] \ \ \text{in} \ \ \Br(Z \otimes_K K_v),
\end{equation}
for all $v \in V^K\setminus \{ v_0 \}.$ If $v \in S$, then
$aa_v^{-1} \in {(F \otimes_K K_v)^{\times}}^{2},$ so
$$
[C(A , \theta_a , \phi_a) \otimes_K K_v] = [C(A \otimes_K K_v ,
(\theta \otimes \id_{K_v})_{a_v}, \phi_{a_v})],
$$
and (\ref{E:O-105}) follows from (\ref{E:O-104}). Now, suppose $v
\in V^K \setminus (S \cup  \{ v_0 \}).$ Since $v \notin S_2$, and by
our construction $(a , d)_{F_v} = 1,$ using (\ref{E:O-155}), we
obtain that
$$
[C(A , \theta_a, \phi_a) \otimes_K K_v] = [C(A, \theta, \phi)
\otimes_K K_v] = [C(A , \tau, \psi) \otimes_K K_v].
$$
On the other hand, since $v \notin S_1,$ according to
(\ref{E:O-150}), we have
$$ [C(A , \tau, \psi) \otimes_K K_v] = [C(A , \tau, \psi_0)
\otimes_K K_v],
$$
and again (\ref{E:O-105}) follows.

\vskip2mm

Finally, we will show that $(*)$ automatically holds if $F/K$ is a
field extension of odd degree. Indeed, if $d \in {F^{\times}}^{2}$
then all we need to prove is that there exists $v_0 \in V^K
\setminus V$ such that $Z \otimes_K K_{v_0}$ is a field if $Z$ is a
field, which immediately follows from Tchebotarev's Density Theorem.
Thus, we may suppose that $d \notin {F^{\times}}^{2},$ so that $E =
F(\sqrt{d})$ is a quadratic extension of $F,$ and then we let $L =
E$ if $Z = K \times K,$ and let $L = EZ$  if $Z/K$ is a quadratic
field extension. Then $L/F$ is a Galois extension with Galois group
isomorphic to $\Z/2\Z$ or $\Z/2\Z \times \Z/2\Z.$ In either case,
there exists $\phi \in \Ga(L/F)$ that acts {\it nontrivially} on
$E,$ and also on $Z$ if $Z/K$ is a quadratic extension (notice that
in this case $Z \not\subset F$ as $F$ has odd degree over $K$). By
Tchebotarev's Density Theorem, there exist infinitely many $w_0 \in
V^F_f$ such that $L/F$ is unramified at $w_0$ and the corresponding
Frobenius automorphism is $\phi.$ In particular, we can choose such
a $w_0$ which lies over some $v_0 \in V^K \setminus V,$ and then
this $v_0$ is as required.

\vskip3mm

We will derive Theorem \ref{T:O-101} from the following result which
applies also in the case where $m$ is even.
\begin{thm}\label{T:O-102}
Let $A = M_m(D)$, where $D$ is a quaternion division algebra over a
global field $K$ of characteristic $\ne 2$, and $\tau$ be an orthogonal involution of $A.$
Furthermore, let $F$ be a commutative \'etale $K$-algebra of degree
$m,$ and $E = F[x]/(x^2 - d)$ for some $d \in F^{\times}$ with the
involution $\sigma \colon x \mapsto -x.$ Assume that for every $v
\in V^K$ there exists a $K_v$-embedding
$$
\iota_v \colon (E \otimes_K K_v , \sigma \otimes \id_{K_v})
\hookrightarrow (A \otimes_K K_v , \tau \otimes \id_{K_v}).
$$
Moreover, assume that condition {$ (*)$} of Proposition
\ref{P:O-101} holds along with the following condition

\vskip2mm

\noindent {$(\#)$} \parbox[t]{11.5cm}{for all $v \in V^K$ such
that $A \otimes_K K_v \not\simeq M_n(K_v)$ and $Z \otimes_K K_v
\simeq K_v \times K_v,$ we have $d \notin {(F \otimes_K K_v)^{\times}}^
{2}.$}

\vskip2mm

\noindent Then there exists a $K$-embedding $\iota \colon (E ,
\sigma) \hookrightarrow (A , \tau).$ Furthermore, condition {$(\#)$} holds
automatically if $m$ is odd.
\end{thm}
\begin{proof}
We will keep the notations introduced earlier. By Theorem
\ref{T:G1}, the existence of $\iota_v$ is equivalent to the
existence of $a_v \in (F \otimes_K K_v)^{\times}$ such that
\begin{equation}\label{E:O-109}
(A \otimes_K K_v , (\theta \otimes \id_{K_v})_{a_v}) \simeq (A
\otimes_K K_v , \tau \otimes \id_{K_v}).
\end{equation}
On the other hand, in view of Proposition \ref{P:O-101}, to prove
our assertion, it suffices to exhibit, for each $v \in V^K,$ an
element ${c}_v \in (F \otimes_K K_v)^{\times}$ for which the
following two conditions hold:
\begin{equation}\label{E:O-110}
(A \otimes_K K_v , (\theta \otimes \id_{K_v})_{{c}_v}) \simeq
(A \otimes_K K_v , \tau \otimes \id_{K_v}) \ \  \text{for all}\ \ v
\in V^K_r;
\end{equation}
and
\begin{equation}\label{E:O-111}
[C(A \otimes_K K_v , (\theta \otimes \id_{K_v})_{{c}_v},
\phi_{{c}_v})] = [C(A, \tau, \psi) \otimes_K K_v] \ \
\text{for all} \ \ v \in V^K.
\end{equation}
We notice that (\ref{E:O-109}) implies that there is an isomorphism
of $K_v$-algebras
$$
C(A \otimes_K K_v , (\theta \otimes \id_{K_v})_{a_v}) \simeq C(A
\otimes_K K_v , \tau \otimes \id_{K_v}),
$$
so it follows from (9.9) and Proposition 1.10 of \cite{BoI} that either
\begin{equation}\label{E:O-112}
[C(A \otimes_K K_v , (\theta \otimes \id_{K_v})_{a_v}, \phi_{a_v})]
= [C(A, \tau, \psi) \otimes_K K_v]
\end{equation}
or
\begin{equation}\label{E:O-113}
[C(A \otimes_K K_v , (\theta \otimes \id_{K_v})_{a_v}, \phi_{a_v})]
= [C(A, \tau, \psi) \otimes_K K_v]
\end{equation}
$$
\hskip7.5cm +\Res_{Z
\otimes_K K_v/K_v}[A \otimes_K K_v]
$$
holds. In particular, if $A \otimes_K K_v \simeq M_n(K_v)$, then
(\ref{E:O-110}) and (\ref{E:O-111}) hold for ${c}_v = a_v.$

\vskip2mm

Assume now that $A \otimes_K K_v \not\simeq M_n(K_v).$ If such a $v$
is real, then there is only one equivalence class of involutions
(cf.\:\cite{Sch}, Theorem 3.7 in Ch.\,10), and therefore
(\ref{E:O-110}) holds for any choice of ${c}_v.$ Thus, in all
cases, it suffices to find ${c}_v$ satisfying only
(\ref{E:O-111}). If (\ref{E:O-112}) holds, we can take ${c}_v
= a_v.$ So, suppose that (\ref{E:O-113}) holds. We will  look for
${c}_v$ of the form ${c}_v = a_vb_v$ with $b_v \in (F
\otimes_K K_v)^{\times}.$ It follows from (\ref{E:O-155}) that then
$$
[C(A \otimes_K K_v , (\theta \otimes \id_{K_v})_{{c}_v},
\phi_{{c}_v})] = [C(A \otimes_K K_v , (\theta \otimes
\id_{K_v})_{a_v}, \phi_{a_v})]
$$
$$
\hskip6cm + \Res_{Z \otimes_{K} K_v/K_v} \Cor_{F \otimes_K K_v/K_v} (b_v ,
d)_{F \otimes_K K_v}.
$$
Comparing this with (\ref{E:O-113}), we see that it is enough to
find $b_v \in (F \otimes_K K_v)^{\times}$ such that
\begin{equation}\label{E:O-120}
\Res_{Z \otimes_K K_v/K_v}\Cor_{F \otimes_K K_v/K_v} (b_v , d)_{F\otimes_KK_v} =
\Res_{Z \otimes_K K_v/K_v} [A \otimes_K K_v]
\end{equation}
If $Z \otimes_K K_v/K_v$ is a quadratic field extension, then $\Res_{Z \otimes_K
K_v/K_v}(\Br(K_v)_2)=0$. So, in this case (\ref{E:O-120}) holds automatically for
any $b_v.$ Thus, it remains only to consider the case where $Z \otimes_K
K_v \simeq K_v \times K_v.$ Then (\ref{E:O-120}) amounts to finding
$b_v \in (F \otimes_K K_v)^{\times}$ such that
\begin{equation}\label{E:O-121}
\Cor_{F \otimes_K K_v/K_v} (b_v , d)_{F \otimes_K K_v} = [A
\otimes_K K_v],
\end{equation}
which we will do using condition $(\#)$. First, we observe that since
$[A \otimes_K K_v]$ is the only element of order 2 in $\Br(K_v),$ it
is enough to find $b_v$ for which $\Cor_{F \otimes_K K_v/K_v} (b_v ,
d)_{F\otimes_KK_v}$ is nontrivial. We have
\begin{equation}\label{E:O-130}
F \otimes_K K_v = \prod_{j = 1}^{\ell} F_{w_j},
\end{equation}
where $w_1, \ldots , w_{\ell}$ are the extensions of $v$ to $F.$ If
$d = (d_{w_1}, \ldots , d_{w_{\ell}})$ in terms of this
decomposition, then by $(\#)$ there exists $j_0 \in \{ 1, \ldots ,
\ell \}$ such that $d_{w_{j_0}} \notin F_{w_{j_0}}^{\times^{2}}.$
So, we can find $b_{w_{j_0}} \in F_{w_{j_0}}^{\times}$ such that
$(b_{w_{j_0}} , d_{w_{j_0}})_{F_{w_{j_0}}}$ is nontrivial. We claim
that $\Cor_{F_{w_{j_0}}/K_v} (b_{w_{j_0}} ,
d_{w_{j_0}})_{F_{w_{j_0}}}$ is also nontrivial. This is obvious for
$v$ real (because then $F_{w_{j_0}} = K_v=\R$), and follows from the
next lemma for $v$ nonarchimedean.
\begin{lemma}\label{L:cor}
Let $\cL/\cK$ be a finite extension of nonarchimedean local fields.
Then $\Cor_{\cL/\cK} \colon \Br(\cL) \to \Br(\cK)$ is an
isomorphism.
\end{lemma}

\vskip1.5mm

\noindent {\it Proof.} Cf.\,\cite{NSW}, Corollary 7.1.4.


\vskip1mm

We now see  that the element $b_v = (1, \ldots , b_{w_{j_0}}, \ldots
, 1)$ is as required, completing the proof of the first assertion of Theorem \ref{T:O-102}.

\vskip2mm

Finally, we will show that $(\#)$ holds automatically if $m$ is odd.
Let $v$ be a place of $K$ such that $A\otimes_K K_v \not\simeq M_n(K_v)$.
In the decomposition (\ref{E:O-130}), for some $j_0 \in \{ 1,
\ldots , \ell\},$ the degree $[F_{w_{j_0}} : K_v]$ is odd. We claim
that then the corresponding component $d_{w_{j_0}} \notin
F_{w_{j_0}}^{\times 2},$ and $(\#)$ will follow. Indeed, otherwise $E
\otimes_K K_v$ would have the following structure:
$$
\cdots \times F_{w_{j_0}} \times F_{w_{j_0}}\times \cdots,
$$
which would prevent it from being a maximal commutative \'etale
subalgebra of $A \otimes_K K_v$ as $(A \otimes_K K_v) \otimes_{K_v}
F_{w_{j_0}}$ is a nontrivial element of $\Br(F_{w_{j_0}})$
(cf.\:Proposition \ref{P:E1}).
\end{proof}

\vskip2mm

\begin{cor}\label{C:O-1001}
Let $(A , \tau)$ be as in Theorem \ref{T:O-102}, $Z$ be the center
of the Clifford algebra $C(A , \tau),$ and $E/K$ be a field
extension of degree $n = 2m$ with an automorphism $\sigma$ of order
two. Set $F = E^{\sigma},$ and write $E = F(\sqrt{d})$ with $d \in
F^{\times}.$ Assume that

\vskip1mm

\noindent $(\diamond)$ \parbox[t]{11cm}{if $Z$ is a field, then so is $F
\otimes_K Z,$}

\vskip2mm

\noindent and that condition $(\#)$ of Theorem \ref{T:O-102} holds.
Then the existence of $K_v$-embeddings $\iota_v \colon (E \otimes_K
K_v , \sigma \otimes \id_{K_v}) \hookrightarrow (A \otimes_K K_v ,
\tau \otimes \id_{K_v})$ for all $v \in V^K$ implies the existence
of a $K$-embedding $(E , \sigma) \hookrightarrow (A , \tau).$
\end{cor}
\begin{proof}
We only need to show that $(\diamond)$ implies condition $(*)$ of
Proposition \ref{P:O-101}. For this, we observe that the extension
$EZ/F$ admits an automorphism $\phi$ that restricts nontrivially to
both $E$ and $Z.$ Then the required fact is established by the
argument used in last paragraph of the proof of Proposition
\ref{P:O-101}.
\end{proof}

\vskip3mm

\noindent {\it Proof of Theorem \ref{T:O-101}.} If $F/K$ is a field
extension of odd degree, then conditions $(*)$ and $(\#)$ hold
automatically. So, our assertion follows from Theorem \ref{T:O-102}.
\hfill $\Box$

\section{Orthogonal involutions: split case}\label{S:OS}

In this section, we examine the local-global principle for
embeddings in the case where $A = M_n(K)$ with an orthogonal
involution $\tau.$ For $n$ even, these considerations, in principle,
can be built into the analysis given in \S \ref{S:O} for the
nonsplit case, however this would make the statements somewhat
cumbersome. In any case, one would still need to consider the case
of $n$ odd. It turns out that the theory of quadratic forms provides
a natural framework for treating both cases (i.e., $n$ even and $n$ odd)
and in fact all we need
in our analysis is the Hasse-Minkowski Theorem and the
classification of quadratic forms over the completions $K_v$ of a
global field $K$ of characteristic $\ne 2$.
For the reader's convenience, we recall that two
nondegenerate quadratic forms $q_1$ and $q_2$ of equal rank over
$K_v$ are equivalent if and only if (1) $v \in V^K_r$ and $q_1$ and $q_2$ have the same
signature over $K_v = \R;$ (2) $v \in V^K_f$ and $q_1$ and  $q_2$
have the same determinant and the same Hasse invariant (if $q =
a_1x_1^2 + \cdots + a_nx_n^2$, then the determinant and the Hasse
invariant are given by $d_v(q) = a_1 \cdots a_n{K_v^{\times}}^2$ (in
$K_v^{\times}/{K_v^{\times}}^2$) and $h_v(q) = \prod_{i < j} (a_i ,
a_j)_v$ respectively,  where $(\cdot , \cdot)_v \in \{ \pm 1 \}$ is the Hilbert
symbol over $K_v$), cf.\:\cite{Sch}, Ch.\:6, \S 4. Even though the
arguments in this section are considerably simpler than those
in \S \ref{S:O}, they use similar ideas, and the same
auxiliary statements. The fact that the
local-global principle holds for the equivalence of quadratic forms
(while it fails for the skew-hermitian forms over quaternion division 
algebras) is the reason why the split case is easier to analyze than the nonsplit case.

First, let us write $\tau$ in the form $\tau(x) = Q^{-1} x^t Q$ for
some nondegenerate symmetric matrix $Q$ (cf.\:\cite{BoI}, Proposition
2.7), and let $b(v , w) = v^t Q w$ be the corresponding bilinear
form on $K^n$ (notice that $b$ is determined, up to a scalar
multiple, by the property $b(xv , w) = b(v , \tau(x)w)$ for $x \in
A$ and all $v , w \in K^n$).  Let $q$ be the quadratic form
associated with $b.$

Now, let $E$ be a commutative \'etale $K$-algebra of dimension $n,$
with an involutive $K$-automorphism $\sigma.$ Set $F = E^{\sigma}.$
Then for any $a \in F^{\times},$ the bilinear form $b_a(v , w) :=
\mathrm{Tr}_{E/K}(av\sigma(w))$ on $E$ is symmetric and satisfies
$$
b_a(xv ,w) = b_a(v , \sigma(x)w)\ \ \text{for all} \ \ v, w, x \in
E.
$$
Let $q_a$ denote the corresponding quadratic form. The following
proposition is valid over an arbitrary field of characteristic $\ne 2$. It is essentially
Proposition 3.9 of \cite{BC-KM}
formulated in our context; it follows from Theorem \ref{T:G1}, however we give a simple direct proof.
\begin{prop}\label{P:OS-1}
An embedding $\iota \colon (E , \sigma) \hookrightarrow (A , \tau)$
as algebras with involution exists if and only if there is an $a \in
F^{\times}$ such that $(E , b_a)$ and $(K^n , b)$ are isometric.
\end{prop}
\begin{proof}
First, we observe that for a symmetric bilinear form $f$ on $E$
\begin{equation}\label{E:OS-1}
f(xv , w) = f(v , \sigma(x)w)  \ \ \text{for all} \ \ v, w, x \in E
\end{equation}
if and only if there is an $a \in F$ for which $f = b_a.$ Indeed,
suppose (\ref{E:OS-1}) holds. Since $E/K$ is \'etale, the trace form
$(v , w) \mapsto \Tr_{E/K}(vw)$ is nondegenerate and therefore we
can write $f(v , w) = {\Tr}_{E/K}(v\varphi(w))$ for some $\varphi
\in \mathrm{End}_K(E).$ Then (\ref{E:OS-1}) implies that
$$
{\Tr}_{E/K}(xv\varphi(w)) ={ \Tr}_{E/K}(v\varphi(\sigma(x)w))
$$
and consequently, $x\varphi(w) = \varphi(\sigma(x)w),$ for all $w ,
x \in E.$ It follows that for $\psi = \varphi \circ \sigma$ we have
$\psi(xw) = x\psi(w).$ Let $\psi(1) = a \in E$. Then $\psi(x) = a x$,
and hence, $\varphi(w) =
a\sigma(w).$ Thus,
$$
f(v , w) = {\Tr}_{E/K}(av\sigma(w)) = b_a(v , w).
$$
Finally, the fact that $f$ is symmetric implies that $\sigma(a) =
a.$ Conversely, for any $a \in F,$ the form $b_a$ is bilinear and
symmetric, and satisfies (\ref{E:OS-1}).

Now, we identify $E$ with $K^n$ as a $K$-vector space in some way, and
use the resulting identification of $\mathrm{End}(E)$ with
$\mathrm{End}(K^n) = A.$ Let $\lambda \colon E \longrightarrow
\mathrm{End}_K(E)$ be the left regular representation. Pick $\alpha
\in \mathrm{Aut}_K(E)$ and consider the embedding $\iota \colon E
\hookrightarrow \mathrm{End}_K(E)$ given by $\iota(x) = \alpha
\lambda(x) \alpha^{-1}.$ Set $\tilde{b}(v , w) = b(\alpha(v) ,
\alpha(w)).$ We claim that the following
\begin{equation}\label{E:OS-2}
\tilde{b}(xv , w) = \tilde{b}(v , \sigma(x)w)
\end{equation}
is equivalent to the fact that $\iota \colon (E , \sigma)
\hookrightarrow (A , \tau)$ respects involutions. We have $$
\tilde{b}(xv , w) = b(\alpha(xv) , \alpha(w)) = b(\iota(x) \alpha(v)
, \alpha(w)) = b(\alpha(v) , \tau(\iota(x))\alpha(w)).
$$
On the other hand, $$
\tilde{b}(v , \sigma(x)w) = b(\alpha(v) , \alpha(\sigma(x)w)) =
b(\alpha(v) , \iota(\sigma(x))(\alpha(w))),
$$
and our claim follows.

Suppose now that there exists an embedding $\iota \colon (E ,
\sigma) \hookrightarrow (A , \tau)$ of algebras with involution.
Then $\iota$ is of the form $\iota(x) = \alpha
\lambda(x)\alpha^{-1}$ for some $\alpha \in \mathrm{Aut}_K(E),$ and
(\ref{E:OS-2}) holds for the corresponding form $\tilde{b}.$ The
first part of the proof shows that $\tilde{b} = b_a$ for some $a \in
F^{\times}$ (notice that $\tilde{b}$ is nondegenerate), and then
$\alpha$ defines an isometry between $(E , b_a)$ and $(K^n , b).$
Conversely, if $\alpha$ yields such an isometry, then $b = b_a,$ and
consequently (\ref{E:OS-2}) holds. This implies that $\iota \colon E
\hookrightarrow A$ given by $\iota(x) = \alpha \lambda(x)
\alpha^{-1}$ respects the involutions.
\end{proof}

We will now use Proposition \ref{P:OS-1} to reduce the problem of
the existence of an embedding $(E , \sigma) \hookrightarrow (A ,
\tau)$ to the case of even $n.$
\begin{prop}\label{P:OS-101}
Let $A = M_n(K)$ with $n$ odd, and let $\tau$ be an orthogonal
involution of $A.$ Furthermore, let $(E , \sigma)$ be an
$n$-dimensional \'etale $K$-algebra with an involution $\sigma$
such that  {\rm (\ref{E:I1}) of \S \ref{S:I}} holds. Then

\vskip2mm

\noindent {\rm (i)} \parbox[t]{11.8cm}{$E = E' \times K$ for some
$\sigma$-invariant subalgebra $E'$ of  $E$ for which {\rm
(\ref{E:I1}) of \S \ref{S:I}} holds for $\sigma' = \sigma \vert E.$}

\vskip1mm

\noindent {\rm (ii)} \parbox[t]{11.8cm}{Assume that for each $v \in
V^K,$ there exists an embedding $$\iota_v \colon (E \otimes_K K_v ,
\sigma \otimes \id_{K_v}) \hookrightarrow (A \otimes_K K_v , \tau
\otimes \id_{K_v}).$$ Then there exists an involution $\tilde{\tau}$
on $A$ given by $\tilde{\tau}(x) = \widetilde{Q}^{-1}x^t \widetilde{Q}$ with
$\widetilde{Q}$ symmetric of the form $\widetilde{Q} = \mathrm{diag}\left(Q'
, \alpha\right),$ such that $(A , \tau) \simeq (A , \tilde{\tau})$
and for $A' = M_{n-1}(K)$ with the involution $\tau'(x) = (Q')^{-1}
x^t Q'$, there exists an embedding $$\iota'_v
\colon (E' \otimes_K K_v , \sigma' \otimes \id_{K_v})
\hookrightarrow (A' \otimes_K K_v , \tau' \otimes \id_{K_v})$$for all $v\in V^K.$}

\vskip1mm

\noindent {\rm (iii)} \parbox[t]{11.8cm}{With $\tau'$ as in \rm{(ii),
{\it the existence of an embedding $\iota \colon (E , \sigma)
\hookrightarrow (A , \tau)$ is equivalent to the existence of an
embedding $\iota' \colon (E' , \sigma') \hookrightarrow (A' ,
\tau').$}}}\end{prop}
\begin{proof}
(i) was actually established in the proof of Proposition
\ref{P:E10}(2). Set $F' = (E')^{\sigma'}.$ To prove (ii), given $a'
\in (F')^{\times},$ we let $b'_{a'}$ denote the
bilinear form on $E'$ defined by $b'_{a'}(x' , y') =
\Tr_{E'/K}(a'x'\sigma'(y')).$ It is easy to see that the determinant
$d'$ of $b'_{a'}$ is independent of $a'$ (cf.\:\cite{BC-KM},
Proposition 4.1), and we set $\alpha = d/d'$, where $d$ is the
determinant of $b.$ We claim that $\alpha$ is represented by $q$ over
$K.$ Indeed, by the Hasse-Minkowski Theorem, it is enough to show
that $\alpha$ is represented by $q$ over $K_v$ for all $v \in V^K.$
According to Proposition \ref{P:OS-1}, it follows from the existence
of $\iota_v$ that there is an $a_v = (a'_v , \alpha_v) \in (F \otimes_K
K_v)^{\times} = (F' \otimes_K K_v)^{\times} \times K_v^{\times}$
such that $b_{a_v} = b'_{a'_v} \perp \langle \alpha_v \rangle,$
where $\langle \alpha_v \rangle$ is the 1-dimensional form
corresponding to $\alpha_v,$ is $K_v$-equivalent to $b.$ As we
observed above, the determinant of $b'_{a'_v}$ is $d',$ so
$$
\det b_{a_v} = \det b'_{a'_v} \cdot \alpha_v = d' \cdot \alpha_v =
\det b = d \ \ \text{in} \ \ K_v^{\times}/{K_v^{\times}}^2,
$$
which implies that $\alpha/\alpha_v \in {K_v^{\times}}^2.$ So $b$, which is equivalent to
$b_{a_v} = b'_{a'_v}\perp \langle\alpha_v\rangle$, is equivalent to
$b'_{a'_v}\perp \langle\alpha\rangle$. Hence, $\alpha$ is a value assumed by $q$
over $K_v$ for all $v$, and therefore, also over $K$. This implies
that  $Q$ is
equivalent to a symmetric matrix $\widetilde{Q}$ of the form $\widetilde{Q}
= \mathrm{diag}\left(Q' , \alpha\right),$ and we will show that
the corresponding involution $\tilde{\tau}$ is as required. Since
$(A , \tau) \simeq (A , \tilde{\tau}),$ we can actually assume that
$Q = \widetilde{Q},$ and we let $b'$ denote the bilinear form
corresponding to $Q'.$

As $Q =\mathrm{diag}\left(Q',\alpha\right)$, $b$ is equivalent to $b'\perp \langle\alpha\rangle$.
We have seen above that it is also equivalent to $b'_{a'_{v}}\perp\langle\alpha\rangle$. Now, it
follows from the Witt Cancelation Theorem (cf.\:\cite{Sch}, Ch.\:I, \S
5) that $b'_{a'_v} \simeq b',$ and therefore by Proposition
\ref{P:OS-1} there exists an embedding $\iota'_v \colon (E'
\otimes_K K_v , \sigma' \otimes \id_{K_v}) \hookrightarrow (A'
\otimes_K K_v , \tau' \otimes \id_{K_v}).$

Finally, to prove (iii),
we observe that the existence of $\iota' \colon (E' , \sigma')
\hookrightarrow (A' , \tau')$ obviously implies the existence of
$\iota \colon (E , \sigma) \hookrightarrow (A , \tau).$ Conversely,
if $\iota$ exists, then by Proposition \ref{P:OS-1} there exists $a =
(a' , \beta) \in F^{\times} = (F')^{\times} \times K^{\times}$ such
that $b_{a} = b'_{a'} \perp \langle \beta \rangle$ is equivalent to
$b = b' \perp \langle \alpha \rangle.$ Taking determinants, we
obtain
$$
\det b_{a} = d' \cdot \beta = \det b = d = d' \cdot \alpha \ \
\text{in} \ \ K^{\times}/{K^{\times}}^2,
$$
so $\alpha/\beta\in {K^{\times}}^2$.
It follows that $b'_{a'} \perp \langle \alpha
\rangle$ is equivalent to $b = b' \perp \langle \alpha \rangle,$ so
by the Witt Cancelation Theorem $b'_{a'} \simeq b',$ implying the
existence of $\iota'.$
\end{proof}

\vskip2mm

Henceforth, we will assume that $n$ is even and $(E , \sigma)$ is an
$n$-dimensional \'etale $K$-algebra with involution satisfying
(\ref{E:I1}) of \S \ref{S:I}. Then, according to Proposition
\ref{P:E12}, we have $E \simeq F[x]/(x^2 - d)$ where $F =
E^{\sigma}$ is an \'etale $K$-algebra of dimension $m = n/2$ and $d
\in F^{\times}.$ We write $F = \prod_{j = 1}^r F_j,$ where $F_j$
is a separable extension of $K,$ and suppose that in terms of this
decomposition $d = (d_1, \ldots , d_r).$ The following result
contains assertion (ii) of Theorem A of the introduction as a particular
case.
\begin{thm}\label{T:M-1}
Assume that for every $v \in V^K$ there exists a $K_v$-embedding
$$
\iota_v \colon (E \otimes_K K_v , \sigma \otimes \id_{K_v})
\hookrightarrow (A \otimes_K K_v , \tau \otimes \id_{K_v}).
$$
If the following condition holds:

\vskip2mm

\noindent $(\Diamond)$ \parbox[t]{12cm}{for any finite subset $V
\subset V^K,$ there exists $v \in V^K\setminus V$ such that for $j \leqslant r$,
if $d_j \notin {F_j^{\times}}^{2}$, then $d_j
\notin {(F_j \otimes_K K_v)^{\times}}^{2}$;}

\vskip2mm

\noindent then there exists an embedding $\iota \colon (E , \sigma)
\hookrightarrow (A , \tau).$ Furthermore, $(\Diamond)$ automatically
holds if $F$ is a field.
\end{thm}
\begin{proof}
We need to show that if for every $v \in V^K,$ there exists an $a_v \in
(F \otimes_K K_v)^{\times}$ such that $q_{a_v}$ is equivalent to $q$
over $K_v$, then there exists an $a \in F^{\times}$ such that $q_a$ is
equivalent to $q$ over $K.$ Let $\tilde{q} = q_a$ for $a=1.$ For any
$v \in V^K,$ we have the following equalities of determinants
$$
d(\tilde{q}) = d(q_{a_v}) =  d(q) \ \ (\text{in} \ \
K_v^{\times}/{K_v^{\times}}^{2}).
$$
It follows that $d(\tilde{q}) = d(q)$ in
$K^{\times}/{K^{\times}}^{2},$ and therefore, $d(q_a) = d(q)$ for
all $a \in F^{\times}.$ So, our task is to find an $a \in F^{\times}$
such that

\vskip2mm

(1) $q_a$ is equivalent to $q$ over $K_v$ for all $v \in V^K_r,$

\vskip1mm

(2) $h_v(q_a) = h_v(q)$ for all $v \in V^K.$

\vskip2mm

\noindent We will use the following formula (written in the additive
notation) for the Hasse invariant (\cite{BC-KM}, Theorem 4.3):
\begin{equation}\label{E:OS27}
h_v(q_a) = h_v(\tilde{q}) + \Cor_{F \otimes_K K_v/K_v}(a , d)_{F \otimes_K
K_v} \ \ \text{for all} \ \ v \in V^K.
\end{equation}
Let $V$ be the (finite) set of places of $K$ containing all the
archimedean ones and those nonarchimedean $v$ for which
$h_v(\tilde{q}) \neq h_v(q),$ and choose $v_0$ as in $(\Diamond)$
for this $V.$ By Lemma \ref{L:HS1}, there exists $a \in
F^{\times}$ such that

\vskip1mm

\noindent (i) $a a_v^{-1} \in {(F \otimes_K K_v)^{\times}}^{2}$ for
all $v \in V,$ and

\vskip1mm

\noindent (ii) $(a , d)_{F \otimes_K K_v} = 1$ for all $v \in V^K
\setminus (V \cup \{ v_0 \}).$

\vskip2mm

\noindent Then (i) implies that $q_a \simeq q$ over $K_v,$ and in
particular, $h_v(q_a) = h(q),$ for all $v \in V.$ On the other hand,
it follows from (ii) and (\ref{E:OS27}) that for $v \in V^K
\setminus (V \cup \{ v_0 \})$ we have
$$
h_v(q_a) = h_v(\tilde{q}) = h_v(q).
$$
Thus, $h_v(q_a) = h_v(q)$ for all $v \neq v_0.$ But the product
formula for the Hilbert symbol implies that
$$
\prod_v h_v(q_a) = \prod_v h_v(q) = 1,
$$
whence $h_v(q_a) = h_v(q)$ holds also for $v = v_0.$ So, $a$ is as
required.

Finally, if $F$ is a field and $d \notin {F^{\times}}^{2}$, then
letting $L$ denote a finite Galois extension of $K$ containing
$F(\sqrt{d}),$ we can choose $\phi \in \Ga(L/F)$ which acts
nontrivially on $\sqrt{d}.$ Then by Tchebotarev's Density Theorem,
we can find $v_0 \in V^K \setminus V$ such that the Frobenius
automorphism of $L/K$ at $v_0$ is $\phi,$ and this $v_0$ is as
required.
\end{proof}

\vskip2mm

\begin{cor}\label{C:OS-1}
Let $(E , \sigma) = (E' , \sigma') \times (K , \id_K)$ where $E'/K$
is a field extension with a $K$-automorphism $\sigma'$ of order two,
$n = \dim_K E.$ Let $A = M_n(K)$ with an orthogonal involution
$\tau.$ Then the existence of embeddings $\iota_v \colon (E
\otimes_K K_v , \sigma \otimes \id_{K_v}) \hookrightarrow (A
\otimes_K K_v , \tau \otimes \id_{K_v})$ for all $v \in V^K$ implies
the existence of an embedding $\iota \colon (E , \sigma)
\hookrightarrow (A , \tau).$
\end{cor}

This follows from Theorem \ref{T:M-1} and Proposition
\ref{P:OS-101}.

\vskip3mm

\noindent {\bf Example 7.5.} We will now construct an example of an
\'etale $K$-algebra $E$ of dimension $n = 6$ with an involution
$\sigma$ satisfying (\ref{E:I1}) of \S \ref{S:I}, and an orthogonal
involution $\tau$ on $A = M_6(K)$ such that the local-global
principle for embeddings of $(E , \sigma)$ into $(A , \tau)$ fails.
(Notice that then Proposition \ref{P:OS-101} enables one to
construct a similar counter-example also for $n = 7$.)

We begin with the following general observation. Let $K$ be a number
field, and let $a , b \in K^{\times}$ be chosen so that $F =
K(\sqrt{a} , \sqrt{b})$ is a degree four extension of $K.$ Let $V$
denote the subset of $V^K$ consisting of all archimedean places, and those
nonarchimedean places which ramify in $F/K.$ Set $F_1 =
K,$ $F_2 = K(\sqrt{a}),$ and $d_1 = a,$ $d_2 = b.$ Let $v \notin V$
be such that $d_1 \notin {K_v^{\times}}^{2} = {(F_1 \otimes_K
K_v)^{\times}}^{2}.$ Then $[K_v(\sqrt{a}) : K_v] = 2.$ Since
$FK_v/K_v$ is unramified, hence cyclic, we conclude that
$K_v(\sqrt{a} , \sqrt{b}) = K_v(\sqrt{a}),$ i.e., $d_2 \in
{K_v(\sqrt{a})^{\times}}^{2} = {(F_2 \otimes_K K_v)^{\times}}^{2}.$
Thus, for every $v \notin V$, $d_j\in {(F_j
\otimes_K K_v)^{\times}}^{2}$ for at least one $j\in \{1,2\}$.

Let $K = \Q,$ and $p_1, p_2$ be two distinct primes of the form $4k
+ 1,$ with one of them of the form $8k + 1,$ such that $\displaystyle
\left(\frac{p_1}{p_2} \right) = 1$ (one can take, for example, $p_1
= 13$ and $p_2 = 17$). Set
$$F_1 =
\Q, \ \  F_2 = \Q(\sqrt{p_1}), \ \  F = F_1 \times  F_2, \ \  d =
(p_1 , p_2)
$$
and $E = F[x]/(x^2 - d)$ with the involution $\sigma$ defined by $x
\mapsto -x.$ Let $\tilde{q}$ be the 6-dimensional quadratic form on
$E$ corresponding to the bilinear form $\Tr_{E/\Q}(x\sigma(y)).$ Now,
Let $q$ be the quadratic form which is equivalent to $\tilde{q}$ over $\Q_v$
for all $v \neq v_{p_1}, v_{p_2}$ (including the unique real place), and which
has the Hasse invariant $h_v(q) = h_v(\tilde{q}) + 1/2$ for $v =
v_{p_1}, v_{p_2}$ (in the additive notation). It follows from
\cite{Sch}, Theorem 6.10 in Ch.\:6, or \cite{Serre}, Ch.\:IV, 3.3,  that such a
form exists, and we let $\tau$ denote the orthogonal involution on
$A = M_6(K)$ corresponding to (the matrix of) $q.$ We claim that for
each $v \in V^{\Q}$ there exists $a_v \in (F \otimes_{\Q}
\Q_v)^{\times}$ such that the quadratic form $q_{a_v},$
corresponding to the bilinear form
$\mathrm{Tr}_{E/K}(a_vx\sigma(y)),$ is equivalent to $q$ over
$\Q_v,$ but there is no $a \in F^{\times}$ such that $q_a$ is
equivalent to $q.$ (In view of Proposition \ref{P:OS-1}, this will
yield the existence of local embeddings $\iota_v$ for all $v \in
V^{\Q},$ but the absence of a~global embedding $\iota.$)

For the local assertion, we observe that we only need to consider $v \in
\{ v_{p_1} , v_{p_2} \}.$ For $v = v_{p_1},$ we pick $s \in
\Q_{p_1}^{\times}$ such that $(s , p_1)_{p_1} = - 1,$ and then
$a_{v_{p_1}} = (s , 1) \in
\Q^{\times}_{p_1}\times \Q_{p_1}(\sqrt{p_1})^{\times}=(F \otimes_{\Q} \Q_{p_1})^{\times}$ is as
required. Similarly, for $v = v_{p_2},$ we pick $t \in
\Q_{p_2}^{\times}$ so that $(t , p_2)_{p_2} = -1,$ and then
$a_{v_{p_2}} = (1 , t , 1) \in \Q_{p_2}^{\times} \times
\Q_{p_2}^{\times} \times \Q_{p_2}^{\times} = (F \otimes_{\Q}
\Q_{p_2})^{\times}$ is as required.

Now, suppose there exists $a = (a_1 , a_2) \in F^{\times}=F_1^{\times} \times
F_2^{\times}$ such that $q_a$ is equivalent to $q$ over $\Q.$ Then
$$
h_{v_{p_1}}(q_a) = h_{v_{p_1}}(\tilde{q}) + \Cor_{F \otimes
\Q_{p_1}/\Q_{p_1}} (a , d)_{F \otimes_{\Q} \Q_{p_1}} =
h_{v_{p_1}}(\tilde{q}) + 1/2,
$$
so $\Cor_{F \otimes \Q_{p_1}/\Q_{p_1}} (a , d)_{F \otimes_{\Q}
\Q_{p_1}} = 1/2.$ Since $p_2 \in {\Q_{p_1}^{\times}}^{2},$ we
necessarily have $(a_1 , p_1)_{p_1} = -1.$ So, by the product
formula, there exists a $v \neq v_{p_1}$ such that $(a_1 , p_1)_v =
-1.$ Since $p_1 \in {\Q_{p_2}^{\times}}^{2},\: {\R^{\times}}^{2},$ we
have $v \neq v_{p_2}, v_{\infty}.$ But it is easy to see that $F =
\Q(\sqrt{p_1} , \sqrt{p_2})$ is unramified outside $V = \{ v_{p_1} ,
v_{p_2} \},$ so according to the observation made earlier, since
$p_1 \notin {\Q_v^{\times}}^2,$ we necessarily have $p_2 \in {(F_2
\otimes \Q_v)^{\times}}^{2}.$ Then $\Cor_{F \otimes_{\Q}
\Q_v/\Q_v}(a , d)_v = 1/2,$ which contradicts $h_v(q_a) = h_v(q) =
h_v(\tilde{q}).$

\section{Invariant maximal subfields distinguish locally isomorphic
algebras, of degree a multiple of $4$, with orthogonal involutions}\label{S:M}

Let $A$ be a central simple algebra over a global field $K,$ of
dimension $n^2,$ and let $\tau$ be an orthogonal involution of $A.$
In this section, we will deal with the set $\cI = \cI(A , \tau)$ of all orthogonal
involutions $\eta$ of $A$  such that
\begin{equation}\label{E:M1001}
(A \otimes_K K_v , \eta \otimes \id_{K_v}) \simeq (A \otimes_K K_v ,
\tau \otimes \id_{K_v}).
\end{equation}
for all $v \in V^K.$ To put this notion in a more traditional
context, we recall that if $A = M_m(D),$ with $D$ being a division
algebra, then $D$ itself admits an involution $\bar{\ }$ (which may
be trivial) and then any involution $\nu$ of $A$ can be written in
the form $\nu(x) = Q_{\nu}^{-1} x^* Q_{\nu}$, where $(x_{ij})^* =
(\bar{x}_{ji})$ and $Q_{\nu}^* = \pm Q_{\nu}.$ In this case, we let
$h_{\nu}$ denote the corresponding $m$-dimensional (skew)-hermitian
form. Then, according to Proposition \ref{P:G111}, we have $(A ,
\eta) \simeq (A , \tau)$ if and only if the corresponding forms
$h_{\eta}$ and $h_{\tau}$ are {\it similar,} i.e., a scalar multiple of $h_{\eta}$ is
equivalent to $h_{\tau}$. So, the
elements of $\cI$ correspond to the (classes of proportional) forms
that are similar to $ h_{\tau}$ at every place of $K,$ and the
investigation of $\cI$ essentially boils down to the Hasse principle
for similarity of forms of a specific type. The analysis of the
latter was recently completed in \cite{LUG}.

For orthogonal
involutions $\nu,$ we either have $A = M_n(K)$, with $Q_{\nu}$
symmetric, making $h_{\nu}$ a quadratic form (split case),
or $A = M_m(D)$, with $D$ a quaternion division algebra,
$\bar{\ }$ being the canonical involution of $D$, and $Q_{\nu}$
satisfying $Q_{\nu}^* = - Q_{\nu}$, in this case $h_{\nu}$ is a skew-hermitian form
(nonsplit case). It is known (cf.\:references in \cite{LUG}, or
Proposition \ref{P:Ap25} below) that the Hasse
principle does hold for similarity of quadratic forms, which implies that in the split
case $\cI$ consists of a single isomorphism class. On the other hand,
in the nonsplit case, $\cI$ often contains more than one isomorphism
class (cf.\:\cite{LUG}), and therefore in this section we will
entirely focus on this case. In particular, unless stated
otherwise, $A$ will denote an algebra of the form $M_m(D)$, where $D$
is a quaternion division algebra, so that $n = 2m.$ (For the sake of
completeness, we mention that the Hasse principle is known to hold
for similarity of hermitian forms over quaternion division algebras
with the standard involution, and also for similarity of hermitian forms
over division algebras with an
involution of the second kind, cf.\:\cite{LUG} and references
therein, so the nonsplit case above is the {\it only} case where
$\cI$ may not reduce to a single isomorphism class.) Our goal is to
show that when $m$ is even, the isomorphism class of each $
\eta \in \cI$ is determined by the isomorphism classes of
$\eta$-invariant maximal fields in $A$. To give a precise statement of this result, we need to
make some preliminary remarks and introduce some notations. First,
we observe that the isomorphism (\ref{E:M1001}) leads to an
isomorphism
$$
C(A , \eta) \otimes_K K_v \simeq C(A , \tau) \otimes_K K_v
$$
of  the corresponding Clifford algebras for every $v \in V^K.$ In
particular,
$$
Z(C(A , \eta)) \otimes_K K_v \simeq Z(C(A , \tau)) \otimes_K K_v \ \
\text{for all} \ \ v \in V^K,
$$
so by applying Tchebotarev's Density Theorem we see that there
exists a quadratic \'etale $K$-algebra $Z$ such that the center
$Z(C(A , \eta))$ is isomorphic to $Z$ for every $\eta \in \cI.$
We let $V$ denote the finite set of all $v \in V^K$ such that
$$
A \otimes_K K_v \not\simeq M_n(K_v) \ \ \text{and} \ \ Z \otimes_K
K_v \simeq K_v \times K_v.
$$
The following theorem, together with Corollary \ref{C:M2}, implies
assertion (ii) of Theorem B (of the
introduction).

\begin{thm}\label{T:M10}
Assume that $m$ is even.

\vskip2mm

\noindent {\rm \ (i)} \parbox[t]{11.9cm}{Given $\eta \in \cI$,
there is an $n$-dimensional $\eta$-invariant commutative \'etale
subalgebra $E_{\eta}$ of $A$ such that $(E_{\eta} , \eta \vert E_{\eta})$
is isomorphic as algebra with involution to $(F_{\eta}[x]/(x^2 - d)
, \theta)$, where $F_{\eta} = (E_{\eta})^{\eta}$, $d \in
F_{\eta}^{\times}$ is such that $d \in {(F_{\eta} \otimes_K
K_v)^{\times}}^{2}$ for all $v \in V$, and $\theta$ is defined by
$\theta(x) = -x.$}

\vskip2mm

\noindent {\rm (ii)} \parbox[t]{11.9cm}{Let $ \eta \in \cI$ and
let $E_{\eta}$ be any commutative \'etale subalgebra of $A$ with the properties
described in {\rm (i)}. If $\nu \in \cI$ and there exists an
embedding $(E_{\eta} , \eta \vert E_{\eta}) \hookrightarrow (A ,
\nu)$, then $(A , \nu) \simeq (A , \eta).$}
\end{thm}

\vskip3mm

We begin by constructing the
required subalgebras over the completions $K_v$ for $v \in V.$
\begin{lemma}\label{L:M4}
Let $v \in V,$ and assume that $m$ is even. Then for any $\eta
\in \cI,$ the algebra $A_v = A \otimes_K K_v$ contains an
$n$-dimensional commutative \'etale $K_v$-subalgebra $E_v$ which is invariant
under $\eta_v = \eta \otimes \id_{K_v}$ and for which there is an
isomorphism of algebras with involution
$$
(E_v , \eta_v \vert E_v) \simeq (F_v[x]/(x^2 - 1) , \theta_v),
$$
where $F_v := E_v^{\eta_v}$, and $\theta_v$ is defined by $x \mapsto
-x.$
\end{lemma}
\begin{proof}
We have $A_v = M_m(D_v)$,  where $D_v = D \otimes_K K_v$ is a division
algebra as $v \in V.$ We will first construct certain $K_v$-algebras
and their embeddings into $M_2(D_v).$ Pick a maximal field $L_v$ in
$D_v,$ and let $g_v \in D_v^{\times}$ be an element such that
$\mathrm{Int}\: g_v$ induces the nontrivial automorphism of $L_v;$
notice that $\overline{g_v} = -g_v$ where $\bar{\ }$ denotes the
canonical involution of $D_v.$ Consider the algebra $C_v =
L_v[x]/(x^2 - 1)$ with the involution $\tau_v$ defined by $x \mapsto
-x.$ Then $(C_v , \tau_v)$ is isomorphic to $(L_v \times L_v ,
\varepsilon_v)$, where $\varepsilon_v$ is the involution $(a
, b) \mapsto (b , a).$ Let $^*$ be the (symplectic) involution of
$M_2(D_v)$ given by $(a_{ij})^* = (\overline{a_{ji}}).$ Then the matrix
$Q = \left(\begin{array}{cc} 0 & g_v \\
g_v & 0 \end{array} \right)$ obviously satisfies $Q^* = -Q,$ so
$\sigma$ given by $\sigma(a) = Q^{-1} a^* Q$ is an orthogonal
involution of $M_2(D_v).$ Now, it is easy to check that $(a , b)
\mapsto \mathrm{diag}(a , b)$ defines an embedding
$$\epsilon_v \colon (C_v , \tau_v) \simeq (L_v \times L_v ,
\varepsilon_v) \hookrightarrow (M_2(D_v) , \sigma)$$ of algebras
with involution.

By our assumption, $m$ is even, say $m = 2r.$ Let $S_v$ be the
direct product of $r$ copies of $L_v,$ and let $R_v = S_v[x]/(x^2 -
1)$ with the involution $\theta_v$ defined by $x \mapsto -x.$ Then,
obviously,
$$
(R_v , \theta_v) \simeq \prod_{i = 1}^{r} (C_v , \tau_v) \ \
\text{and} \ \ R_v^{\theta_v} = S_v.
$$
For $(a_1, \ldots , a_r) \in R_v,$ where $a_i \in C_v,$ we set
$$
\iota_v(a_1, \ldots , a_r) =
\mathrm{diag}\left(\epsilon_v(a_1), \ldots ,
\epsilon_v(a_r)\right) \in M_m(D_v).
$$
Then $\iota_v$ yields an embedding of algebras with involution
$$
(R_v , \theta_v) \hookrightarrow (M_m(D_v) , \mu_v) \ \ \text{with}
\ \ \mu_v(a) = M^{-1} a^* M,
$$
where, again, $^*$ is defined by $(a_{ij})^* = (\overline{a_{ji}})$,
and $M = \mathrm{diag}(Q, \ldots , Q).$ It follows from the definitions
that the discriminant of $\mu_v$ equals $\mathrm{discr}(\sigma) = 1 \cdot
{K_v^{\times}}^{ 2}$ (cf.\,\cite{BoI}, Ch.\,II, \S 7). On the other
hand, since $v \in V,$ we have $Z \otimes K_v = K_v \times K_v,$
which implies that $\mathrm{discr}({\eta_v}) = 1 \cdot
{K_v^{\times}}^ 2$ (cf.\,\cite{BoI}, Ch.\,II, Theorem 8.10). But
then $(A_v , \mu_v) \simeq (A_v , \eta_v)$ (cf.\,\cite{Sch},
Ch.\,10, Theorem 3.6 for the nonarchimedean case and Theorem 3.7 for
the real case). Thus, there exists an embedding $(R_v , \theta_v)
\hookrightarrow (A_v , \eta_v),$ the image of which furnishes a
subalgebra $E_v$ of  $A_v$ with the desired properties.
\end{proof}

\vskip2mm

{\it Proof of Theorem \ref{T:M10}}\,(i). For each $v\in V$, pick a  commutative
 \'etale subalgebra $E_v$ of  $A_v := A \otimes_K K_v$ as in the
preceding lemma. Using Proposition \ref{P:E50}, we find an
$n$-dimensional $\eta$-invariant commutative \'etale subalgebra $E$ of
$A$ which satisfies (\ref{E:I1}) of  \S \ref{S:I} and for which
$$
(E \otimes_K K_v , (\eta \vert E) \otimes \id_{K_v}) \simeq (E_v ,
\eta_v \vert E_v).
$$
By Proposition \ref{P:E12}, we have
$$
(E , \eta \vert E) \simeq (F[x]/(x^2 - d) , \theta)
$$
where $F = E^{\eta},$ $d \in F^{\times}$ and $\theta$ is defined by
$x \mapsto -x.$ Then by our construction, for every $v \in V$ we have
$$
(F \otimes_K K_v)[x]/(x^2 - d) \simeq (F \otimes_K K_v)[x]/(x^2 -
1),
$$
implying that $d \in {(F \otimes_K K_v)^{\times}}^{2},$ as required.
\hfill $\Box$

\vskip3mm

{\it Proof of Theorem \ref{T:M10}}\,(ii). The argument relies on
characterizing the isomorphism classes in $\cI$ as fibers of a
certain map $\delta$, which we will now construct. For each $\eta \in
\cI$ we fix an isomorphism $\phi_{\eta} \colon Z \to Z(C(A ,
\eta)),$ and then let $C(A , \eta , \phi_{\eta})$ denote $C(A ,
\eta)$ with the structure of $Z$-algebra defined using
$\phi_{\eta}.$ Consider the following subgroup
$$
\cB = \prod_{v \in V} \ \langle\: \Res_{Z \otimes_K K_v/K_v}([A
\otimes_K K_v]) \:\rangle \ \ \subset \ \ \prod_{v \in V} \Br(Z
\otimes_K K_v).
$$
Furthermore, let $\cB_0$ be the subgroup of $\cB$ generated by the
element
$$
(\: \Res_{Z \otimes_K K_v/K_v}([A \otimes_K K_v])\,)_{v \in V},
$$
and let $\overline{\cB} = \cB/\cB_0.$ This group will be the target of
the required map $\delta.$ To define it, we need to {\it fix} an
element of $\cI;$ to keep our notations simple, we will pick the
the involution $\tau$ used to define $\cI = \cI(A , \tau)$ as
the fixed element, but in fact any other element of $\cI$ can be
utilized equally well. Given $\eta \in \cI,$ for any $v \in V^K$ there is
an isomorphism as in (\ref{E:M1001}). Then with an appropriate choice
of an isomorphism $\psi \colon Z \otimes_K K_v \to Z(C(A \otimes_K
K_v , \eta \otimes \id_{K_v})),$ we will obtain an isomorphism
$$
C(A \otimes_K K_v, \eta \otimes \id_{K_v} , \psi) \ \simeq \ C(A
\otimes_K K_v , \tau \otimes \id_{K_v} , \phi_{\tau} \otimes
\id_{K_v})
$$
of $(Z \otimes_K K_v)$-algebras. Using (\ref{E:O-150}) of  \S \ref{S:O},
we see that in $\Br(Z \otimes K_v)$ the following difference
$$
\delta(\eta , \phi_{\eta} , v) :=
[C(A, \eta, \phi_{\eta})\otimes_K K_v] - [C(A, \tau,
\phi_{\tau})\otimes_K K_v]
$$
equals either $0$ or $\Res_{Z \otimes_K K_v/K_v}([A \otimes_K
K_v]).$ In fact, $\delta(\eta ,\phi_{\eta},  v) = 0$ for all $v
\notin V,$ which leads us to consider the element $(\delta(\eta ,
\phi_{\eta} , v))_{v \in V} \in \cB.$ Now, for a different
isomorphism $\phi'_{\eta} \colon Z \to Z(C(A , \eta)),$ again by
(\ref{E:O-150}) in \S \ref{S:O}, we have
$$
[C(A , \eta , \phi'_{\eta})] - [C(A , \eta , \phi_{\eta})] =
\Res_{Z/K}([A]).
$$
This means that the coset
$$
\delta(\eta) := (\delta(\eta , \phi_{\eta} , v))_{v \in V} + \cB_0 \
\in \ \overline{\cB}
$$
depends only on $\eta,$ not on the choice of $\phi_{\eta},$ and
therefore the map
$$
\delta \colon \cI \to \overline{\cB}, \ \ \eta \mapsto
\delta(\eta),
$$
is well-defined.
\begin{lemma}\label{L:M1}
For $\eta , \nu \in \cI,$ the condition $\delta(\eta) =
\delta(\nu)$ implies that $(A , \eta) \simeq (A , \nu).$
\end{lemma}
\begin{proof}
Indeed, $\delta(\eta) = \delta(\nu)$ means that after replacing
$\phi_{\eta}$ with another isomorphism $\phi'_{\eta} \colon Z \to
Z(C(A , \eta))$ if necessary, we can assume that
\begin{equation}\label{E:M1}
[C(A , \eta , \phi_{\eta}) \otimes_K K_v] = [C(A , \nu , \phi_{\nu})
\otimes_K K_v]
\end{equation}
in $\Br(Z \otimes_K K_v)$ for all $v \in V.$ At the same time, as we
observed above, the fact that $\eta, \nu \in \cI$
automatically implies (\ref{E:M1}) for $v \in V^K \setminus V.$
Using the injectivity of $ \Br(Z) \longrightarrow \bigoplus_{v \in
V^K} \Br(Z \otimes_K K_v),$  we conclude that
$$
[C(A , \eta , \phi_{\eta})] = [C(A , \nu , \phi_{\nu})],
$$
and in particular, $C(A , \eta) \simeq C(A , \nu)$ as $K$-algebras.
Since in addition we have
$$
(A \otimes_K K_v , \eta \otimes \id_{K_v}) \simeq (A \otimes_K K_v ,
\nu \otimes_K K_v) \ \  \text{for all}\ \   v \in V_r^K,
$$
by the result of Lewis and Tignol \cite{LT}, mentioned at the
beginning of \S \ref{S:O}, we have $(A , \eta) \simeq (A , \nu).$
\end{proof}

\vskip2mm

Let now $\eta \in \cI,$ and let $E_{\eta}$  be a
commutative \'etale subalgebra of $A$ as in Theorem \ref{T:M10}(i). Furthermore, let $\nu
\in \cI,$ and suppose that there is an embedding $\iota \colon
(E_{\eta} , \eta \vert E_{\eta}) \hookrightarrow (A , \nu).$ By
Lemma \ref{L:M1}, to show that $(A , \eta) \simeq (A , \nu)$ it is
enough to show that $\delta(\eta) = \delta(\nu).$ Observing that for
the involution $\theta$ in Theorem \ref{T:G1} which extends $\eta
\vert E_{\eta},$ one can take $\eta$ itself, we see that the
existence of $\iota$ implies that there is an $a \in F_{\eta}^{\times}$
such that $(A , \eta_a) \simeq (A , \nu),$ where $\eta_a = \eta
\circ \mathrm{Int}\: a.$ Then $ \eta_a \in \cI$ and
$\delta(\eta_a) = \delta(\nu).$ So, to prove Theorem \ref{T:M10}(ii), it remains only
to show that
\begin{equation}\label{E:M25}
\delta(\eta_a) = \delta(\eta).
\end{equation}
But according to (\ref{E:O-155}) in \S \ref{S:O}, for any $v \in V^K,$ we
have
\vskip1mm

\noindent $[C(A \otimes_K K_v , \eta_a \otimes \id_{K_v} , (\phi_{\eta})_a
\otimes \id_{K_v})] =$
$$
[C(A \otimes_K K_v , \eta \otimes \id_{K_v} , \phi_{\eta} \otimes
\id_{K_v})] + \Res_{Z \otimes_K K_v/K_v} \Cor_{F_{\eta} \otimes_K
K_v/K_v}(a , d)_{F_{\eta} \otimes_K K_v}.
$$
If now $v \in V$, then the assumption that $d \in {(F \otimes_K
K_v)^{\times}}^ 2$ implies that
$$
[C(A \otimes_K K_v , \eta_a \otimes \id_{K_v} , (\phi_{\eta})_a
\otimes \id_{K_v})] =[C(A \otimes_K K_v , \eta \otimes \id_{K_v} ,
\phi_{\eta} \otimes \id_{K_v})],
$$
i.e.,
$$
\delta(\eta_a , (\phi_{\eta})_a , v) = \delta(\eta , \phi_{\eta} ,
v),
$$
and (\ref{E:M25}) follows. \hfill $\Box$

\vskip3mm

\begin{cor}\label{C:M1}
Let $A = M_m(D)$, where $D$ is a quaternion division algebra over
$K$ and $m$ is even, and let $\tau$ be an orthogonal involution of
$A.$ Suppose we are given $\eta \in \cI = \cI(A , \tau),$ a
finite set $\mathscr{S} \subset V^K \setminus V,$ and for each $v
\in \mathscr{S}$, an $n$-dimensional (with $n = 2m$) commutative
\'etale subalgebra $E(v)$ of $A_v := A \otimes_K K_v$ invariant
under $\eta_v = \eta \otimes \id_{K_v}$ such that
$\dim_{K_v}E^{\eta_v}_v = m$. Then there exists an $n$-dimensional
$\eta$-invariant commutative \'etale subalgebra $E$ of $A$
with the properties described in Theorem \ref{T:M10}\,{\rm (i)} (with ``$E_{\eta}$" replaced by ``$E$" and ``$F_{\eta}$" by ``$F$"), and
such that for every $v \in \mathscr{S}$ we have
\begin{equation}\label{E:M171}
E(v) = g_v^{-1}(E\otimes_K K_v)g_v \ \ \text{for\ a} \ \ g_v
\in G_{\eta}(K_v),
\end{equation}
where $G_{\eta} = \mathrm{SU}(A , \eta).$
\end{cor}
\begin{proof}
Let $E_{\eta}$ be a commutative \'etale subalgebra of  $A$ as in Theorem
\ref{T:M10}(i), and for $v \in V,$  set $E(v) = E_{\eta} \otimes_K K_v.$
Applying Proposition \ref{P:E50} we can find an $n$-dimensional
$\eta$-invariant commutative \'etale subalgebra $E$ of $A$
such that
$$
E(v) = g_v^{-1}(E\otimes_K K_v)g_v \ \ \text{with} \ \ g_v
\in G_{\eta}(K_v) \ \ \text{for all} \ \ v \in \mathscr{S} \cup V.
$$
Then (\ref{E:M171}) holds automatically. On the other hand, writing
$E_{\eta}$ and $E$ in the form
$$
E_{\eta} = F_{\eta}[x]/(x^2 - d) \ \ \text{and} \ \ E= F[x]/(x^2 -
d'),
$$
where  $F_{\eta} = (E_{\eta})^{\eta},$ $F= E^{\eta},$  and $d\in
F_{\eta}^{\times}$, $d' \in F^{\times}$ (cf.\:Proposition
\ref{P:E12}), we observe that for $v \in V,$ the fact that the
isomorphism
$$\phi_v \colon E\otimes_K K_v \to E(v) = E_{\eta} \otimes_K K_v, \
\ a \mapsto g_v^{-1}a g_v,$$ commutes with $\eta_v,$ implies that
$\phi_v(F \otimes_K K_v) = F_{\eta}\otimes_K K_v$, and $\phi_v(d')
\in d \cdot {(F_{\eta} \otimes_K K_v)^{\times}}^{2}.$ Since by our
construction, $d \in {(F_{\eta} \otimes_K K_v)^{\times}}^{2},$ we obtain
that $d' \in {(F \otimes_K K_v)^{\times}}^{2},$ as required.
\end{proof}

\vskip1mm

Combining this corollary with the results of \cite{PR4}, we obtain
the following stronger assertion, which we will need in \S
\ref{S:Ap}.
\begin{cor}\label{C:M2}
Keep the notations of Corollary \ref{C:M1}. Then there exists an
$n$-dimensional $\eta$-invariant commutative \'etale subalgebra
$E$ of  $A$ which has the properties described in Theorem
\ref{T:M10}{\rm (i)} (with ``$E_{\eta}$" replaced by ``$E$" and ``$F_{\eta}$" by ``$F$"), satisfies (\ref{E:M171}) for all $v \in
\mathscr{S},$ and for which the corresponding maximal $K$-torus
$T_{\eta}$ of $G_{\eta} = \mathrm{SU}(A , \eta)$ is generic over $K$
(``generic" in the sense of \S \ref{S:E}). This algebra $E$
is automatically a field extension of $K.$
\end{cor}
\begin{proof}
The group $G_{\eta}$ is semisimple, and we let $r$ denote the number
of nontrivial conjugacy classes in the Weyl group of $G_{\eta}$. Using Tchebotarev's Density Theorem, we choose a subset
$S \subset V^K_f \setminus (\mathscr{S} \cup  V)$ of cardinality $r$
so that $G_{\eta}$ splits over $K_v$ for all $v \in S.$ Then,
according to Theorem 3 of \cite{PR4} (cf.\:also Theorem 3.1 in
\cite{PR2}), one can pick a maximal $K_v$-torus $T(v)$ of $G_{\eta}$,
for each $v \in S$, so that every maximal $K$-torus which is
conjugate to $T(v)$ by an element of $G_{\eta}(K_v),$ for all $v \in
S,$ is generic over $K.$ By Proposition~\ref{P:E11}, $T(v)$
corresponds to an $n$-dimensional $\eta_v$-invariant commutative
\'etale subalgebra $E(v)$ of $A_v$ satisfying (\ref{E:I1}) of \S
\ref{S:I}. Using Corollary \ref{C:M1}, we can find an
$n$-dimensional $\eta$-invariant commutative \'etale subalgebra
$E$ of $A$ which possesses the properties described in
Theorem \ref{T:M10}(i) (with ``$E_{\eta}$" replaced by ``$E$" and ``$F_{\eta}$" by ``$F$") and for
which $E\otimes_K K_v$ is
conjugate to $E(v)$ by an element of $G_{\eta}(K_v),$ for all $v \in
S \cup \mathscr{S}$ (in particular, yielding (\ref{E:M171}) for all
$v \in \mathscr{S}$). Let $T_{\eta}$ be the maximal
$K$-torus of $G_{\eta}$ corresponding to $E$.
Then $T_{\eta}$ is conjugate to $T(v)$ by
an element of $G(K_v),$ for all $v \in S,$ hence is generic. The
fact that $E$ is a field extension of $K$ now follows from
Proposition \ref{P:E55}.
\end{proof}

\addtocounter{thm}{1}

\noindent {\bf Remark 8.6.} Assume that $m$ is even, and let
$\eta$ and $\nu \in \cI.$ Then for any $\eta$-invariant
\'etale subalgebra $E$ of  $A$ and any $v \in V,$ there is an
embedding
$$
(E \otimes_K K_v  , \ (\eta \vert E) \otimes \id_{K_v}) \
\hookrightarrow \ (A \otimes_K K_v , \eta \otimes \id_{K_v}) \simeq
(A \otimes_K K_v  , \ \nu \otimes \id_{K_v}).
$$
Now, let $E_{\eta}$ be a subalgebra having the properties
described in Theorem \ref{T:M10}(i); notice that according to
Corollary \ref{C:M2} we can even choose $E_{\eta}$ to be a field
extension of $K.$ Then according to Theorem \ref{T:M10}(ii) there is an
embedding $(E_{\eta} , \eta \vert E) \hookrightarrow (A , \nu)$ if and only if  $(A
, \eta) \simeq (A , \nu).$ Since $\cI$ typically contains more than
one isomorphism class (cf.\,\cite{LUG} in conjunction with
Proposition \ref{P:G111} of this paper), we see that the local-global
principle for embeddings of fields with involution usually fails for even $m.$

\vskip3mm

We close this section with the Hasse principle for similarity of
quadratic forms. As we already mentioned earlier, an important
consequence of this result in our context is  that the set $\cI$ in
the split case reduces to a single isomorphism class. This fact will be
used in \S \ref{S:Ap}. The Hasse principle in question is known
(cf.\,\cite{O}, \cite{C}), but unfortunately it is not recorded in
the standard books on quadratic forms. So, we decided to sketch the
argument for the sake of completeness, especially since it uses
nothing more than Lemma \ref{L:HS1}.
\begin{prop}\label{P:Ap25}
Let $f$ and $g$ be two nondegenerate quadratic forms of the same
dimension $n$ over a global field $K$ of characteristic $\neq 2.$ If
for every $v \in V^K$ there exists $\lambda_v \in K_v^{\times}$ such
that $g$ is equivalent to $\lambda_v f$ over $K_v$, then there exists
$\lambda \in K^{\times}$ such that $g$ is equivalent to $\lambda f$
over $K.$
\end{prop}
\begin{proof}
We will use $d(\cdot)$ and $h_v(\cdot)$ to denote the determinant
and the Hasse invariant over $K_v,$ respectively (cf. \S
\ref{S:OS}).  It is easy to check that
\begin{equation}\label{E:M78}
d(\lambda f) = \lambda^n d(f) \ \ \text{and} \ \ h_v(\lambda f) =
(\lambda , \delta(f))_v \cdot h_v(f)
\end{equation}
where $\delta(f) = (-1)^{n(n-1)/2} \cdot d(f).$

Let now $n$ be odd. Set $\lambda = d(g)/d(f).$ Then $d(g) \equiv
d(\lambda f)$ in $K^{\times}/{K^{\times}}^{2}.$ For $v \in V^K,$
since $g$ and $\lambda_v f$ are equivalent over $K_v,$ by taking
determinants we obtain $\lambda \equiv \lambda_v$ in
$K_v^{\times}/{K_v^{\times}}^2.$ So, being equivalent to $\lambda_v
f,$ the form $g$ is equivalent to $\lambda f$ over $K_v,$ for any $v
\in V^K.$ Applying the Hasse-Minkowski Theorem, we obtain that $g$
is equivalent to $\lambda f.$

Now, we consider the case of even $n.$ Notice that for any
$v \in V^K$ we have $$d(g)/d(f) \equiv d(\lambda_v f)/d(f) \equiv 1
\ \ \text{in} \ \ K_v^{\times}/{K_v^{\times}}^{2}.
$$
So, $d(g)\equiv d(f)$ in $K^{\times}/{K^{\times}}^{2},$ and
therefore, $d(g) \equiv d(\lambda f)$ for any $\lambda \in
K^{\times}.$ First, assume that $\delta(f) \in {K^{\times}}^2.$ Then
it follows from (\ref{E:M78}) that for any $v \in V^K$ we have
$$
h_v(g) = h_v(\lambda_v f) = h_v(f),
$$
consequently
$$
h_v(g) = h_v(\lambda f)
$$
for any $\lambda \in K^{\times}.$ In particular, this means that $g$
and $\lambda f$ are equivalent over $K_v$ for any $\lambda \in
K^{\times}$ and any $v \in V^K_f.$ Now, choose $\lambda \in
K^{\times}$ so that $\lambda \lambda_v^{-1} \in {K_v^{\times}}^{2}$
for all $v \in V^K_r.$ Then $g$ is equivalent to $\lambda f$ over
$K_v$ for all $v \in V^K,$ hence over $K.$

Finally, we consider the case where $\delta(f) \notin
{K^{\times}}^2.$ Let $S$ be a finite set of places of $K$
containing all the archimedean ones and those nonarchimedean $v$ for which
$h_v(f) \neq h_v(g).$ By Tchebotarev's Density Theorem, we can find
$v_0 \in V^K \setminus S$ such that $\delta(f) \notin
{K_{v_0}^{\times}}^2.$ Then by Lemma \ref{L:HS1} there exists
$\lambda \in K^{\times}$ such that $\lambda \lambda_v^{-1} \in
{K_v^{\times}}^2$ for all $v \in S$ and $(\lambda , \delta(f))_v =
1$ for all $v \in V^K \setminus (S \cup \{ v_0 \}).$ Using
(\ref{E:M78}), we see that $h_v(g) = h_v(\lambda f)$ for all $v \neq
v_0.$ Since $$\prod_v h_v(g) = \prod_v h_v(\lambda f) = 1,$$ we infer  that
$h_{v_0}(g) = h_{v_0}(\lambda f)$ as well. Arguing as above, we
conclude that $g$ and $\lambda f$ are equivalent over $K.$
\end{proof}

\section{Application to weakly commensurable arithmetic subgroups}\label{S:Ap}

In this section, we will show how our previous results
(particularly, Theorem \ref{T:M10}) can be used to complete the
analysis of weakly commensurable arithmetic subgroups in the case
that was left open in the original version of \cite{PR2},
viz.\,where the ambient algebraic groups are of type $D_{2r},$ with
$r \geqslant 3$ (for obvious reasons, type $D_4$ requires a special
treatment, we hope to study groups of this type later).

We first recall the notion of weak commensurability introduced in
\cite{PR2}. Let $G_1$ and $G_2$ be two connected semi-simple
algebraic groups defined over a field $F$ of characteristic zero.
Semi-simple elements $\gamma_i \in G_i(F),$ where $i = 1, 2,$ are
said to be {\it weakly commensurable} if there exist maximal
$F$-tori $T_i$ of $G_i$ such that $\gamma_i \in T_i(F),$ and for
some characters $\chi_i$ of $T_i$ (defined over an algebraic closure
$\overline{F}$ of $F$) we have
$$
\chi_1(\gamma_1) = \chi_2(\gamma_2) \neq 1.
$$
Furthermore, (Zariski-dense) subgroups $\Gamma_i$ of $G_i(F)$ are
{\it weakly commensurable} if given a semi-simple element $\gamma_1
\in \Gamma_1$ of infinite order, there is a semi-simple element
$\gamma_2 \in \Gamma_2$ of infinite order which is weakly
commensurable to $\gamma_1$, and conversely, given a semi-simple
element $\gamma_2 \in \Gamma_2$ of infinite order, there is a
semi-simple element $\gamma_1 \in \Gamma_1$ of infinite order weakly
commensurable to $\gamma_2$. In the present paper, we will be
concerned exclusively with the situation where {\it both} groups
$G_1$ and $G_2$ are absolutely almost simple of the same type
$D_{2r}$ with $r \geqslant 3.$ (We note that in general, for
absolutely almost simple groups $G_i,$ the existence of finitely
generated weakly commensurable Zariski-dense subgroups 
$\Gamma_i$ of $G_i(F),$ for $i = 1, 2,$ implies that either $G_1$ and
$G_2$ are of the same type, or one of them is of type $B_n$ and the
other of type $C_n,$ cf.\:Theorem 1 in \cite{PR2}.) It is easy to
show (\cite{PR2}, Lemma 2.4) that weak commensurability of finitely
generated Zariski-dense subgroups is preserved by the $F$-isogenies
of ambient algebraic groups (\cite{PR2}, Lemma 2.4). So, by
replacing the given groups $G_i$'s with isogenous ones and enlarging
the field $F$ if necessary, we reduce our analysis of weakly
commensurable subgroups to the situation where $G_1 = G_2 = G,$ and
moreover, $G$ is a $F$-form of $\mathrm{SO}_n,$ hence
$F$-isomorphic to $\mathrm{SU}(A , \tau)$ for some central simple
$n^2$-dimensional $F$-algebra $A$ with an orthogonal involution
$\tau.$

One of the central issues in \cite{PR2} was to determine when weak
commensurability of $S$-arithmetic subgroups implies their
commensurability, which in turn led to some interesting results
about length-commensurable and isospectral locally symmetric spaces 
(see \cite{PR3} for a nontechnical exposition of these results). We
used the following definition of $S$-arithmeticity. Let $G$ be a
connected absolutely almost simple algebraic group defined over a
field $F$ of characteristic zero, $\overline{G}$ be its adjoint
group and $\pi: G \to \overline{G}$ be the natural isogeny. Two
subgroups $\Gamma' , \Gamma''$ of $G(F)$ are said to be {\it
commensurable up to an $F$-automorphism of} $\overline{G}$ if there
exists an $F$-automorphism $\sigma$ of $\overline{G}$ such that
$\sigma(\pi(\Gamma'))$ and $\pi(\Gamma'')$ are commensurable in the
usual sense. Now, suppose we are given a number field $K,$ an
embedding $K \hookrightarrow F,$ and a connected semi-simple 
algebraic $K$-group ${\cG}$ such that the $F$-group $_F\overline{\cG}$ obtained
from the adjoint group $\overline{\cG}$ of $\cG$ by extension of scalars $K \hookrightarrow F,$ is
$F$-isomorphic to $\overline{G}$ (in other words, $\overline{\cG}$
is an $F/K$-form of $\overline{G}$). Then we have an embedding
$\overline{\iota}: \: \overline{\cG}(K) \hookrightarrow
\overline{G}(F)$ which is well-defined up to an $F$-automorphism of
$\overline{G}.$ Now, given a subset $S$ of $V^K$ containing
$V^K_{\infty},$ but not containing any nonarchimedean places where
$\cG$ is anisotropic, a subgroup $\Gamma$  of $G(F)$ such that
$\pi(\Gamma)$ is commensurable with
$\overline{\iota}(\overline{\cG}(\cO_K(S)))$ (where $\cO_K(S)$ is
the ring of $S$-integers of $K$) up to an $F$-automorphism of
$\overline{G}$ is called a $(\overline{\cG}, K , S)$-{\it arithmetic
subgroup}. We note that in the situation at hand, i.e., where $G$ is
a form of $\mathrm{SO}_n$, with $n = 4r$ and $r \geqslant 3,$ for
every $F/K$-form $\overline{\cG}$ of $\overline{G},$ there exists a
unique $F/K$-form $\cG$ of $G$ admitting a $K$-isogeny 
$\cG \to \overline{\cG}$ compatible with $\pi.$ Furthermore,
two $F/K$-forms $\overline{\cG}_1$ and $\overline{\cG}_2$ of
$\overline{G}$ are $K$-isomorphic if and only if the corresponding
$F/K$-forms $\cG_1$ and $\cG_2$ of $G$ are $K$-isomorphic. Finally, a
subgroup $\Gamma$ of  $G(F)$ is $(\overline{\cG}, K,
S)$-arithmetic if and only if it is commensurable up to an
$F$-automorphism of $G$ with $\iota(\cG(\cO_K(S)))$ where $\iota: \:
\cG(K) \hookrightarrow G(F)$ is the natural embedding lifting
$\overline{\iota}$ (in view of this, $(\overline{\cG}, K,
S)$-arithmetic subgroups of $G(F)$ may be referred to as
$(\cG, K, S)$-arithmetic subgroups, as we will often do in this section).

We showed in \cite{PR2} for a general absolutely almost simple
algebraic $F$-group $G$ that if $\Gamma_i$ is a Zariski-dense
$(\overline{\cG}_i , K_i , S_i)$-arithmetic subgroup of $G(F)$ for
$i = 1 , 2$, then the weak commensurability of $\Gamma_1$ and
$\Gamma_2$ implies that $K_1 = K_2$ and $S_1 = S_2$ (Theorem 3), and
then their commensurability up to an $F$-automorphism of
$\overline{G}$ is equivalent to the assertion that $\overline{\cG}_1
\simeq \overline{\cG}_2$ over $K$ (Proposition 2.5 of \cite{PR2}).
Furthermore, we showed that the latter follows from weak
commensurability of $\Gamma_1$ and $\Gamma_2$ if $G$ is of type
different from $A_n$, $D_n$, and $E_6.$ On the other hand, we showed
that groups of types $A_n\: (n>1)$, $D_n$ ($n$ odd), and $E_6$
contain weakly commensurable, but not commensurable, $S$-arithmetic
subgroups (cf.\:Examples 6.5, 6.6 and \S 9 in \cite{PR2}). The only
unresolved question in the original version of  \cite{PR2} involved
the groups of type $D_n$, with $n$ even. We are now able to show
that, as far as weak commensurability is concerned, these groups
behave like ``good groups" if $n> 4.$
\begin{thm}\label{T:Ap1}
Let $G$ be an absolutely almost simple algebraic group of type
$D_{2r},$ with $r > 2,$  defined over a field $F$ of characteristic
zero, and let $\Gamma_i$ be a Zariski-dense $({\cG}_i, K,
S)$-arithmetic subgroup of $G(F)$ for $i = 1, 2.$ If $\Gamma_1$ and
$\Gamma_2$ are weakly commensurable, then $\overline{\cG}_1 \simeq
\overline{\cG}_2$ (hence $\cG_1 \simeq \cG_2$) over $K,$ and
consequently $\Gamma_1$ and $\Gamma_2$ are commensurable up to an
$F$-automorphism of $\overline{G}$.
\end{thm}

The proof of the theorem relies on Theorem \ref{T:M10} and a
connection, valid over arbitrary fields, between weak
commensurability of elements and isomorphism of commutative \'etale
subalgebras associated to the corresponding maximal tori, which we
will now describe. As we explained earlier, we may (and we will)
assume that $G = \mathrm{SU}(A , \tau)$ where $A$ is a central
simple $F$-algebra of dimension $n^2$ (if $G$ is of type $D_{2r}$, 
then $n = 4r,$ $r > 2,$ but some of our considerations are valid for
arbitrary $n \geqslant 3,$ $n \neq 8$) and $\tau$ is orthogonal
involution of $A.$ Then any $F/K$-form $\cG$ of $G$ equals 
$\mathrm{SU}(\mathscr{A} , \tau_{\mathscr{A}})$ for a suitable
central simple $n^2$-dimensional algebra $\mathscr{A}$ over $K$
equipped with an orthogonal involution $\tau_{\mathscr{A}}$ such
that $(\mathscr{A} \otimes_K F , \tau_{\mathscr{A}} \otimes
\mathrm{id}_F) \simeq (A , \tau)$ (cf. Lemma \ref{L:Ap1}(1) below).
So, as a preparation for the proof of Theorem \ref{T:Ap1}, we first
consider the following general situation. For $i=1, 2$, let $A_i$ be
two central simple algebras over the same (infinite) field $K,$ of
dimension $n^2,$ endowed with orthogonal involutions $\tau_i.$
Furthermore, let $F/K$ be a field extension such that
$$
(A_1 \otimes_K F , \tau_1 \otimes \id_F) \simeq (A_2 \otimes_K F ,
\tau_2 \otimes \id_F);
$$
we will denote this common $F$-algebra with involution by $(A ,
\tau).$ Then $\cG_i := \mathrm{SU}(A_i , \tau_i)$ is an $F/K$-form
of $G := \mathrm{SU}(A , \tau)$ for $i = 1, 2,$ and in the sequel,
we will view the groups $\cG_i(K)$ as subgroups of the group $G(F).$
We refer the reader to \S \ref{S:E} for the definition of a {\it
generic} maximal $K$-torus.
\begin{prop}\label{P:Ap1}
Assume that $n \geqslant 3,$ $n \neq 4, 8,$ and  let $L_i$ be the
minimal Galois extension of $K$ over which $\cG_i$ becomes an inner
form. Furthermore, let $E_i$ be a $\tau_i$-invariant maximal
commutative \'etale subalgebra of $A_i$ satisfying (\ref{E:I1}) of
\S \ref{S:I}, and let $\cT_i$ be the corresponding maximal $K$-torus
of $\cG_i.$ Assume that

\vskip2mm

{\rm (a)} $L_1 = L_2;$

\vskip1mm

{\rm (b)} $\cT_1$ is a generic maximal $K$-torus of $\cG_1.$

\vskip2mm

\noindent If there exists an element $\gamma_1 \in
\cT_1(K)$ of infinite order which is weakly commensurable to some $\gamma_2 \in \cT_2(K)$,
then $(E_1 , \tau_1 \vert E_1) \simeq (E_2 , \tau_2 \vert E_2)$ as
algebras with involution.
\end{prop}
\begin{proof}
We begin with the following lemma, which is valid for all $n
\geqslant 3$  (and also for symplectic involutions).
\begin{lemma}\label{L:Ap1}
{\rm (1)} Let $F/K$ be a field extension, and let $\varphi \colon
\cG_1 \to \cG_2$ be an $F$-isomorphism of algebraic groups. Then
$\varphi$ extends uniquely to an isomorphism
$$
\widetilde{\varphi} \colon (A_1 \otimes_K F , \tau_1 \otimes \id_F)
\longrightarrow (A_2 \otimes_K F , \tau_2 \otimes \id_F)
$$
of algebras with involution.

\vskip2mm

\noindent {\rm (2)} \parbox[t]{12cm}{For $i = 1,2$, let $\cT_i$ be a maximal
$K$-torus of $\cG_i$, and let $E_i$ be the corresponding maximal commutative \'etale
$K$-subalgebra of $A_i$. If $\varphi \colon \cG_1 \to \cG_2$
is a $\overline{K}$-isomorphism of algebraic groups such that
$\varphi(\cT_1) = \cT_2$, and the restriction $\varphi \vert \cT_1$ is
defined over $K$, then $(E_1 , \tau_1 \vert E_1) \simeq (E_2 ,
\tau_2 \vert E_2)$ as algebras with involution.}
\end{lemma}
\begin{proof}
(1): As the referee has pointed out, the first assertion follows from Theorem 26.15 of \cite{BoI},
but for the convenience of the reader we give the following direct proof. Since both
involutions are orthogonal, there exists an isomorphism
$$
\widetilde{\psi} \colon (A_1 \otimes_K \overline{F} , \tau_1 \otimes
\mathrm{id}_{\overline{F}}) \longrightarrow (A_2 \otimes_K \overline{F} ,
\tau_2 \otimes \mathrm{id}_{\overline{F}})
$$
of algebras with involution. We let $\psi \colon \cG_1
\longrightarrow \cG_2$ denote the induced isomorphism between the
special unitary groups, and observe that $\alpha := \psi^{-1} \circ
\varphi$ is an $\overline{F}$-automorphism of $\cG_1.$ But it is
well-known that any $\overline{F}$-automorphism of $\cG_1 =
\mathrm{SU}(A_1 , \tau_1)$ is conjugation by a suitable $h \in
\cH_1(\overline{F})$ where $\cH_1 := \mathrm{U}(A_1 , \tau_1).$
(Indeed, over $\overline{F},$ we have $\cG_1 \simeq \mathrm{SO}_n$
and $\cH_1 \simeq \mathrm{O}_n.$ If $n$ is odd, then $\cG_1$ is of
type $B_r,$ and every automorphism of $\cG_1$ is inner. For $n$
even, the group of outer automorphisms of $\cG_1$ has order two, and
conjugation by any element $h \in \cH_1(\overline{F}) \setminus
\cG_1(\overline{F})$ does give an outer automorphism of $\cG_1.$
Thus, any $\overline{F}$-automorphism of $\cG_1$ is conjugation by
an element of $\cH_1(\overline{F}).$) So, we can pick $h \in
\cH_1(\overline{F})$ such that $\varphi = \psi \circ \mathrm{Int}\:
h.$ Then $\widetilde{\varphi} := \widetilde{\psi} \circ
\mathrm{Int}\: h$ is an isomorphism $(A_1 \otimes_K \overline{F} ,
\tau_1 \otimes \mathrm{id}_{\overline{F}}) \longrightarrow (A_2
\otimes_K \overline{F} , \tau_2 \otimes \mathrm{id}_{\overline{F}})$
of algebras with involution. It is easy to check that
$\cG_i(\overline{F})$ spans $A_i \otimes_K \overline{F}$ as a
$\overline{F}$-vector space, so the Zariski-density of $\cG_i(F)$ in
$\cG_i$ (cf.\:\cite{Bor}, 18.3) implies that $\cG_i(F)$ spans $A
\otimes_K F$ as a $F$-vector space. Since $\varphi(\cG_1(F)) =
\cG_2(F),$ we see that $\widetilde{\varphi}(A_1 \otimes_K F) = A_2
\otimes_K F,$ as required.

\vskip2mm

(2): By (1), $\varphi$ extends to an isomorphism
$\widetilde{\varphi} \colon (A_1 \otimes_K \overline{K} , \tau_1 \otimes
\id_{\overline{K}}) \longrightarrow (A_2 \otimes_K \overline{K} , \tau_2
\otimes \id_{\overline{K}})$ of algebras with involution. Since
$\varphi(\cT_1(K)) = \cT_2(K)$ and $E_i$ coincides with the
$K$-subalgebra generated by $\cT_i(K)$ (cf.\:the proof of Proposition
\ref{P:E11}), we obtain that $\widetilde\varphi(E_1) = E_2,$ and assertion (2)
follows.
\end{proof}

\vskip1mm

To prove Proposition \ref{P:Ap1}, we pick simply connected coverings $\widetilde{\cG}_i
\stackrel{\pi_i}{\longrightarrow} \cG_i$ of $\cG_i$ defined over $K$,  and set $\widetilde{\cT}_i =
\pi_i^{-1}(\cT_i).$ In view of our assumptions (a) and (b), the fact
that $\gamma_1$ and $\gamma_2$ are weakly commensurable implies the
existence of a $\overline{K}$-isomorphism $\widetilde{\varphi}
\colon \widetilde{\cG}_1 \longrightarrow \widetilde{\cG}_2$ such that
$\widetilde{\varphi} \vert \widetilde{\cT}_1$ is an isomorphism of
$\widetilde{\cT}_1$ onto $\widetilde{\cT}_2$  defined over
$K$\,(cf.\,Theorem 4.2 and Remark 4.4 in \cite{PR2}). Since $n > 8,$
we automatically have $\widetilde{\varphi}(\ker \pi_1) = \ker
\pi_2,$ and therefore $\widetilde{\varphi}$ descends to a
$\overline{K}$-isomorphism $\varphi \colon \cG_1 \longrightarrow \cG_2$
such that $\varphi \vert \cT_1$ is defined over $K$. Then our
assertion follows from Lemma~\ref{L:Ap1}(2).
\end{proof}

\vskip1mm

The following proposition establishes assertion (i) of
Theorem B. As we already noted in \S 8, assertion (ii) of that theorem is implied by
Theorem \ref{T:M10} and Corollary \ref{C:M2}.
\begin{prop}\label{P:Ap2}
For $i=1, 2$, let $A_i$ be a central simple algebra over a
number field $K,$ of dimension $n^2$, with $n \geqslant 3,$ endowed with an orthogonal involution $\tau_i$, and let $\cG_i = \mathrm{SU}(A_i ,
\tau_i).$ Assume that {\rm either}

\vskip2mm

\noindent $(a)$ \parbox[t]{12cm}{$(A_1 ,\tau_1)$ and $(A_2 ,
\tau_2)$ have the same isomorphism classes of \mbox{$n$-dimensional}
commutative \'etale subalgebras invariant under the involutions and
satisfying {\rm (\ref{E:I1})} (i.e., for any $n$-dimensional
$\tau_1$-invariant commutative \'etale subalgebra $E_1$ of $A_1$
satisfying {\rm (\ref{E:I1})}, there exists an embedding $(E_1 ,
\tau_1 \vert E_1) \hookrightarrow (A_2 , \tau_2),$ and vice versa),}

\vskip1.7mm

\noindent {\rm or}

\vskip1.7mm

\noindent $(b)$ \parbox[t]{12cm}{$n\neq 4$, and for some finite $S
\subset V^K,$ for $i = 1,\,2$, any $(\cG_i, K,
S)$-arithmetic subgroup $\Gamma_i$ of $\cG_i(K)$ is Zariski-dense in $\cG_i$, and $\Gamma_1$ and $\Gamma_2$ are weakly commensurable.}

\vskip2mm

\noindent Then

\vskip2mm

\noindent \ {\rm (i)} \parbox[t]{12cm}{$A_1 \simeq A_2$ (in other
words, $A_1$ and $A_2$ involve the same division algebra in their
description);}

\vskip2mm

\noindent {\rm (ii)} $(A_1 \otimes_K K_v, \tau_1 \otimes \id_{K_v})
\simeq (A_2 \otimes_K K_v , \tau_2 \otimes \id_{K_v})$ for all $v
\in V^K.$
\vskip2mm

\noindent If $n$ is even, then the same conclusion holds if $A_1$
and $A_2$ just have the same isomorphism classes of maximal {\rm
subfields} invariant under the involutions.

\end{prop}
\begin{proof}
We begin by establishing the following two key properties of the
$K$-groups $\cG_i = \mathrm{SU}(A_i , \tau_i):$

\vskip2mm

\noindent $(\alpha)$ $\mathrm{rk}_{K_v}\,\cG_1 = \mathrm{rk}_{K_v}\,
\cG_2$ for all $v \in V^K;$

\vskip2mm

\noindent $(\beta)$ \parbox[t]{12cm}{$L_1 = L_2,$ where $L_i$ is the
minimal Galois extension of $K$ over which $\cG_i$ becomes an inner
form.}

\vskip2.5mm

\noindent These properties have been proven in \cite{PR2}, Theorems 6.2 and
6.3, if $(b)$ holds, so we will prove them assuming that
$(a)$ holds. (In condition $(b)$ we have assumed that $n \neq 4$
since if $n=4$, the corresponding special unitary groups are
semi-simple but not absolutely simple, which prevents us from using
the results of \cite{PR2}.) To prove $(\alpha)$ we basically repeat
the argument given in the proof of Theorem 6.2 in \cite{PR2}. More
precisely, by symmetry it is enough to show that
\begin{equation}\label{E:Ap250}
\mathrm{rk}_{K_v}\, \cG_1 \leqslant\mathrm{rk}_{K_v}\, \cG_2.
\end{equation}
Let $\cT_1(v)$ be a maximal $K_v$-torus of $\cG_1$ that contains a
maximal $K_v$-split torus, and let $E_1(v)$ be the corresponding
commutative \'etale subalgebra of $A_1 \otimes_K K_v.$ By Proposition \ref{P:E50}, there
exists a $\tau_1$-invariant commutative \'etale subalgebra $E_1$ of $A_1$ satisfying
(\ref{E:I1}) of \S \ref{S:I} such that the corresponding $K$-torus
$\cT_1$ is conjugate to $\cT_1(v)$ by an element of $\cG_1(K_v);$ in
particular, $\mathrm{rk}_{K_v}\: \cT_1 = \mathrm{rk}_{K_v}\: \cT_1(v).$
By our assumption, there exists an embedding $(E_1 , \tau_1 \vert
E_1) \hookrightarrow (A_2 , \tau_2),$ which implies that there is a
$K$-embedding $\cT_1 \hookrightarrow \cG_2,$ and (\ref{E:Ap250})
follows.

\vskip1mm

Next, we observe that the argument given in the proof of Theorem 6.3
in \cite{PR2} shows that $(\beta)$ is a consequence of $(\alpha).$
Indeed, there exists a finite subset $S$ of  $V^K$ such that $\cG_1$
and $\cG_2$ are quasi-split over $K_v$ for any $v \in V^K \setminus S$
(cf.\:\cite{PlR}, Theorem 6.7). Then $(\alpha)$ implies that a place
$v \in V^K \setminus S$ splits in $L_1$ if and only if it splits in
$L_2,$ and then $L_1 = L_2$ by Tchebotarev's Density Theorem.

\vskip2mm

(i): We will now use $(\alpha)$ and $(\beta)$ to prove (i).
For $n$ odd, we have $A_1 \simeq M_n(K) \simeq A_2,$ and there
is nothing to prove. So, we assume that $n$ is even and write $A_i =
M_m(D_i)$ for some quaternion central simple $K$-algebra $D_i,$ where $m = n/2.$ To
show that $D_1 \simeq D_2$ (which will prove our claim) it is enough
to show that $D_1$ and $D_2$ are ramified at exactly the same
places. By
symmetry it suffices to show that for $v \in V^K$ if ${D_1}_v := D_1\otimes_K K_v$ is a
division algebra, then ${D_2}_v := D_2\otimes_K K_v$ is also a division algebra. Assume
the contrary. First, let us show that $\cG_2$ is $K_v$-isotropic. This
is obvious if $n > 4$ and $v \in V^K_f.$ If $v \in V^K_r$, then our
assumption that ${D_1}_v$ is a division algebra implies that $\cG_1$
is $K_v$-isotropic (cf.\:\cite{Sch}, Ch.\:10, Theorem 3.7). But then,
by $(\alpha),$ $\cG_2$  must also be $K_v$-isotropic. It remains to
consider the case $n = 4$ and $v \in V^K_f.$ Here we need to use
$(\beta)$ and the description of $L_i$ in terms of discriminant
(\cite{BoI}, Ch.\:2, Theorem 8.10). The unique
anisotropic quadratic form in four variables over $K_v$ has
determinant (which coincides with its discriminant) in ${K_v^{\times}}^2,$ so
if $\cG_2$ happens to be $K_v$-anisotropic, then $v$ splits in $L_2$.
But then $v$ must split in $L_1$, which means that the
binary skew-hermitian form over ${D_1}_v$ corresponding to $\tau_1$
has determinant (discriminant) in ${K_v^{\times}}^{2}.$
However, it is known that any such form is necessarily isotropic
(\cite{Sch}, Ch.\:10, Theorem 3.6). So, $\cG_1$ is $K_v$-isotropic,
contradicting $(\alpha).$

Now, the assumption that $A_2 \otimes_K K_v = M_n(K_v)$ and $\cG_2$ is
isotropic means that $(A_2 \otimes_K K_v , \tau_2 \otimes
\id_{K_v})$ is isomorphic to $(M_n(K_v) , \sigma_2)$ where
$\sigma_2(x) = Q_2^{-1} x^t Q_2$ with $Q_2 = \mathrm{diag}(R , T)$
and $R = \left(\begin{array}{cc} 0 & 1 \\ 1 & 0 \end{array}
\right).$ Notice that if $\epsilon$ is the nontrivial $K_v$-automorphism of
$K_v \times K_v$,  then the map $(a , b) \mapsto \mathrm{diag}(a , b)$
defines an embedding $(K_v \times K_v , \epsilon) \hookrightarrow
(M_2(K_v) , \rho)$ where $\rho(x) = R^{-1} x^t R.$ Using Proposition
\ref{P:E50} we now see that there exists a $n$-dimensional $\tau_2$-invariant
commutative \'etale subalgebra $E_2$ of  $A_2$ satisfying (\ref{E:I1}) of  \S \ref{S:I}
such that $(E_2 \otimes_K K_v , (\tau_2 \vert E_2) \otimes
\id_{K_v})$ contains $(K_v \times K_v , \epsilon)$ as a direct
factor. By our assumption, $(E_2 , \tau_2 \vert E_2)$ can be
embedded into $(A_1 , \tau_1).$ But then $A_1 \otimes_K K_v$
contains an $n$-dimensional commutative \'etale subalgebra which has $K_v \times
K_v$ as a direct factor which, by Proposition \ref{P:E1},
contradicts the assumption that ${D_1}_v$ is a division algebra.

\vskip1mm

If $n$ is even, we will let $D$ denote the common quaternion central
{\it simple} $K$-algebra involved in the description of $A_1$ and
$A_2$ (thus, $D$ may be $M_2(K)$), and assume (as we may) in the
rest of the proof that $A_1$ and $A_2$ coincide with $A = M_m(D).$

\vskip2mm

(ii): In this paragraph, we treat the case where $n$ is odd.
Then $A_1 = A_2 = M_n(K),$ and $\tau_i(x) = Q_i^{-1} x^t Q_i$ with
$Q_i$ symmetric, $i = 1, 2.$ Let $q_i$ be the quadratic form with
matrix $Q_i.$ We need to show that for any $v \in V^K,$ the forms
$q_1$ and $q_2$ are similar over $K_v$ (cf.\:Proposition
\ref{P:G111}).
By $(\alpha),$ the groups $\cG_1$ and $\cG_2$ have the same
$K_v$-rank, and therefore the forms $q_1$ and $q_2$ have the same
Witt index over $K_v.$ For $v \in V^K_r,$ this immediately implies
that $q_1$ is equivalent to $\pm q_2,$ as required. Let now $v \in
V^K_f.$ Replacing one of the forms by a proportional form, we can
assume that $d(q_1) = d(q_2)$ in $K_v^{\times}/{K_v^{\times}}^{2}$
(cf.\:(\ref{E:M78})). We can write $q_i = q_i^h \perp q_i^a$ where
$q_i^h$ is hyperbolic and $q_i^a$ is anisotropic over $K_v.$ Then
$q_1^a$ and $q_2^a$ have the same dimension $s$ (which can only be 1
or 3) and the same determinant. But then $q_1^a$ and $q_2^a$ are
equivalent: for $s= 1,$ this is obvious, and for $s = 3$ it follows
from the fact that, up to equivalence, there is a unique anisotropic
ternary quadratic form of a given determinant. Thus, $q_1$ and $q_2$
are equivalent over $K_v,$ and the required isomorphism in (ii)
follows from Proposition \ref{P:G111}. \vskip1mm

Let now $n$ be even, $m = n/2$ and $A = M_m(D)$, where $D$ is a
quaternion central simple $K$-algebra.
Notice that it follows from $(\beta)$ that the involutions $\tau_1$
and $\tau_2$ have the same discriminant (equivalently, the same
determinant). Let now $v \in V^K$ be such that $D_v = D \otimes_K
K_v$ is a division algebra. Write $\tau_i$ in the form $\tau_i(x) =
Q_i^{-1} x^* Q_i,$ where $(x_{ij})^* = (\overline{x_{ji}})$ and
$\bar{\ }$ is the standard involution of $D_v,$ $Q_i \in M_m(D_v)$
is an invertible skew-hermitian matrix, and let $h_i$ be the
corresponding skew-hermitian form. Then $h_1$ and $h_2$ have the
same discriminant,  and therefore are equivalent over $D_v$: for $v$
nonarchimedean this follows from Theorem 3.6 of \cite{Sch}, Ch.\,10,
and for $v$ real it follows from Theorem 3.7 of loc.\,cit. As above,
this leads to the required isomorphism.

Next, we consider the case where $D_v \simeq M_2(K_v),$ and hence $A
\simeq M_n(K_v).$ Then the involutions $\tau_i \otimes \id_{K_v},$
which for simplicity we will denote by $\tau_i,$ can be
written in the form $\tau_i(x) = Q_i^{-1} x^t Q_i$, where $Q_i \in
M_n(K_v)$ is an invertible symmetric matrix. Let $q_i$ be the
quadratic form with matrix $Q_i.$ As above, we conclude that $q_1$
and $q_2$ have the same Witt index (over $K_v$) and the same
determinant: $d(q_1) = d(q_2),$ or, equivalently, the same
discriminant: $\delta(q_1) = \delta(q_2)$, where $\delta(q) =
(-1)^{n/2}\cdot d(q),$ and to establish our claim we need to show
that $q_1$ and $q_2$ are similar over $K_v.$ If $v \in V^K_r$, then
the mere fact that $q_1$ and $q_2$ have the same Witt index 
implies that $q_1$ is equivalent to $\pm q_2,$ yielding the required
fact. Let now $v \in V^K_f.$ 
First, suppose that the common discriminant $\delta \in
{K_v^{\times}}^{2}.$ Since binary forms whose discriminant is a
square, are isotropic, the common value of the Witt index of $q_1$
and $q_2$ can only be $n/2$ or $(n-4)/2.$ It is well-known that
there is a unique anisotropic quadratic form over $K_v$ in four
variables (viz., the reduced norm form of the unique quaternion
division algebra over $K_v$), so in both cases, $q_1$ and $q_2$ are
equivalent. It remains to consider the case where the common
discriminant $\delta \notin {K_v^{\times}}^{2}.$ Let $q$ be any
$n$-dimensional quadratic form with discriminant $\delta,$ and let
$\lambda \in K_v^{\times}$ be such that the Hilbert symbol $(\delta
, \lambda)_v = -1$ (which exists as $\delta \notin
{K_v^{\times}}^{2}$). Then it follows from (\ref{E:M78}) that the
Hasse invariant $h_v(\lambda q)$ equals  $- h_v(q),$ and therefore
the forms $q$ and $\lambda q$ represent the two equivalence classes
of $n$-dimensional forms of discriminant $\delta.$ This, clearly,
implies that in our situation $q_1$ and $q_2$ are similar, as
required. \vskip2mm

Finally, we note that arguing as in the proof of Corollary \ref{C:M2}, we see that the
subalgebras $E$ used in the above argument can be chosen so that the
corresponding $K$-torus $\cT$ is generic. If $n$ is even, then such an
$E$ is automatically a field extension of $K$ (Proposition
\ref{P:E55}), so effectively our argument only relies on the
assumption that $A_1$ and $A_2$ contain the same
isomorphism classes of maximal fields invariant under the
given involutions. \end{proof}

Using the fact that for $A = M_n(K)$ and any orthogonal involution
$\tau,$ the set $\cI = \cI(A , \tau)$ reduces to a single
isomorphism class (Proposition \ref{P:Ap25}), we obtain the
following interesting consequence of Proposition
\ref{P:Ap2}.
\begin{cor}\label{C:Ap2}
Let $A_i$,  $i = 1, 2,$ be central simple algebras over a number
field $K,$ of dimension $n^2$, where $n \geqslant 3,$ $n \neq 4,$
given with orthogonal involutions $\tau_i$, and let $\cG_i =
\mathrm{SU}(A_i , \tau_i).$ Assume that for some finite $S \subset
V^K,$ for $i = 1,\,2$, any $(\cG_i, K,
S)$-arithmetic subgroup $\Gamma_i$ of $\cG_i (K)$ is Zariski-dense in $\cG_i$, 
and $\Gamma_1$ and $\Gamma_2$ are weakly commensurable. Then, if one of
the algebras is isomorphic to $M_n(K),$ the groups  $\cG_1$ and
$\cG_2$ are $K$-isomorphic, and hence the $S$-arithmetic subgroups
$\Gamma_1$ and $\Gamma_2$ are commensurable.
\end{cor}

{\it Proof of Theorem \ref{T:Ap1}.} There exist central simple
algebras with orthogonal involutions $(A_1 , \tau_1)$ and $(A_2 ,
\tau_2)$ over $K,$ of dimension $n^2$, where $n = 4r$ and $r>
2,$ such that $\cG_i = \mathrm{SU}(A_i , \tau_i).$ We need to show
that the existence of Zariski-dense weakly commensurable
$S$-arithmetic subgroups $\Gamma_1$ of  $\cG_1(K)$ and $\Gamma_2$ of
$\cG_2(K)$ implies that $(A_1 , \tau_1) \simeq (A_2 , \tau_2).$
According to Proposition~\ref{P:Ap2}(i), $A_1$ and $A_2$ involve the
same division algebra $D$ in their description. If $D = K$ (i.e.,
$A_1 = A_2 = M_n(K)$), then the assertion of the theorem follows from
Corollary \ref{C:Ap2}. So, we can assume in the rest of the proof
that $D$ is a quaternion division algebra over $K$, and $A_1$ and $A_2$
coincide with $A = M_m(D)$, where $m = 2r.$ Let $\cI = \cI(A ,
\tau_1).$ Using Corollary \ref{C:M2}, one can find for each $\eta \in \cI,$
an $n$-dimensional $\eta$-invariant commutative
\'etale subalgebra $E_{\eta}$ of  $A$ satisfying (\ref{E:I1}) of \S 1
so that if  $\cT_{\eta}$ is the corresponding maximal $K$-torus of
$\cG_{\eta} := \mathrm{SU}(A , \eta)$, and  $V$ is the finite set of places of $K$
described just before the
statement of Theorem \ref{T:M10}, then the following conditions
hold:

\vskip2mm

\noindent (a) $E_{\eta}$ is as  in Theorem \ref{T:M10}(i);

\vskip2mm

\noindent (b) $\cT_{\eta}$ is generic (in the sense of \S \ref{S:E});

\vskip2mm

\noindent (c) $\displaystyle {\cT_{\eta}}_S := \prod_{v \in S}
\cT_{\eta}(K_v)$ is noncompact.

\vskip2mm

\noindent Indeed, first assume that there exists $v_0 \in S \cap V$.
Then applying Corollary \ref{C:M2} with
$\mathscr{S} = \emptyset$ we find a subalgebra $E_{\eta}$ such that
(a) and (b) hold. To see that (c) holds automatically in this case,
one needs to observe that since $(E_{\eta} , \eta \vert E_{\eta}) =
(F_{\eta}[x]/(x^2 - d) , \theta)$ (notations as in Theorem
\ref{T:M10}) and $d \in {(F_{\eta} \otimes_K
K_{v_0})^{\times}}^{2},$ there is a $K_{v_0}$-isomorphism $\cT_{\eta}
\simeq \mathrm{R}_{F_{\eta} \otimes_K K_{v_0}/K_{v_0}}(\mathrm{GL}_1),$
implying that $\cT_{\eta}(K_{v_0})$ is noncompact. It remains to
consider the case where $S \cap V = \emptyset.$ Since $\cG_1 (K)$ contains a
Zariski-dense $S$-arithmetic group, the group ${\cG_1}_S = \prod_{v\in S}\cG_1(K_v)$ is
noncompact, i.e., there exists $v_0 \in S$ such that $\cG_1(K_{v_0})$
is noncompact. But the groups $\cG_1$ and $\cG_{\eta}$ are isomorphic
over $K_{v_0},$ so $\cG_{\eta}(K_{v_0})$ is noncompact as
well.\footnote{This, in particular, shows that $S$-arithmetic
subgroups in $\cG_{\eta}$ are Zariski-dense, for any $\eta \in
\cI.$} Then $\cG_{\eta}$ contains a maximal $K_{v_0}$-torus $\cT_0$ such
that $\cT_0(K_{v_0})$ is noncompact, and we let $E(v_0)$ denote the
corresponding commutative \'etale subalgebra of $A \otimes_K K_{v_0}.$ Applying
Corollary \ref{C:M2} to $\mathscr{S} = \{ v_0 \},$ we can find
$E_{\eta}$ so that both (a) and (b) hold, and in addition
$$(E_{\eta} \otimes_K K_{v_0} , (\eta \otimes \id_{K_{v_0}}) \vert E_{\eta} \otimes_K
K_{v_0}) \simeq (E(v_0) , (\eta \otimes \id_{K_{v_0}}) \vert
E(v_0)).
$$
Then $\cT_{\eta} \simeq \cT_0$ over $K_{v_0},$ implying that
$\cT_{\eta}(K_{v_0})$ is noncompact and yielding (c).

\vskip1mm

Now, let $\cT_1 := \cT_{\tau_1}$ in the above notation. Then $\cT_1$ is
$K$-anisotropic, hence the quotient ${{\cT_1}}_S/\cT_1(\cO(S))$
is compact (\cite{PlR}, Theorem 5.7), where
${{\cT_1}}_S =\prod_{v\in S} \cT_1(K_v)$,
and $\cO(S)$ is the ring of $S$-integers in $K$.  Since, by (c), ${\cT_1}_S$
is noncompact, the group
$\cT_1(\cO(S))$ is infinite, and therefore there exists an element
$\gamma_1 \in \cT_1(K) \cap \Gamma_1$ of infinite order. By our
assumption, $\gamma_1$ is weakly commensurable to some semi-simple
$\gamma_2 \in \Gamma_2.$ Let $\cT_2$ be a maximal $K$-torus of $\cG_2$
containing $\gamma_2,$ and let $E_1$ and $E_2$ be the
$n$-dimensional commutative \'etale subalgebras of $A$ corresponding to $\cT_1$ and $\cT_2$
respectively. By Theorem 6.3 of \cite{PR2}, we have $L_1 = L_2,$
where $L_i$ is the minimal Galois extension of $K$ over which $\cG_i$
becomes an inner form. So, condition (b) above permits an
application of Proposition \ref{P:Ap1}, from which we get $(E_1 ,
\tau_1 \vert E_1) \simeq (E_2 , \tau_2 \vert E_2).$ In particular,
there is an embedding $(E_1 , \tau_1 \vert E_1) \hookrightarrow (A ,
\tau_2).$ Due to condition (a), we can apply Theorem
\ref{T:M10}(ii), to obtain $(A , \tau_1) \simeq (A , \tau_2)$.

\vskip3mm

\noindent {\bf Remark 9.6.}  Theorem \ref{T:Ap1} implies that if $K$
is a number field and $G$ is a connected absolutely simple $K$-group
of type $D_{2r}$ with $r> 2,$ then any $K$-form $G'$ of $G$
having the same set of isomorphism classes of maximal $K$-tori as
$G,$ is necessarily $K$-isomorphic to $G;$ see Theorem 7.5 in \cite{PR2}.
\vskip4mm

\noindent{\bf 9.7.} We take this opportunity to point out the
following corrections in \cite{PR2}. \vskip2mm

\noindent(i) In assertion (2) of Theorem 4.2, replace the condition ``{\it if}  $L_1=L_2,$"  by ``{\it if} $L_1 = L_2 =:L$, {\it and} $\theta_{T_1}(\rm{Gal} (L_{T_1}/L))\supset W(G_1,T_1),$".  (ii) In the proof of Proposition 5.6, after the proof of Lemma 5.7, replace  ``$G$", occurring without a subscript, with ``$G_2$"  everywhere.  (iii) In the fourth line of the proof of Theorem 4 (in \S 6), replace ``$G$" by ``$G_1$", and in the next line, replace ``obtained from $\mathscr{G}$" by ``obtained from $\overline{\mathscr{G}}$".

\vskip5mm

\centerline{\sc Appendix}

\vskip5mm

The goal of this appendix is to describe a Galois-cohomological approach to
the problem of embedding of a commutative \'etale algebra with an involutive
automorphism into a simple algebra with an involution, and also to
interpret the latter as a problem of finding rational points on
certain homogeneous spaces. Even though these methods do provide
some additional insight, it appears that neither of them is likely
to yield any simplification in the proofs of our embedding theorems,
nor can they be used to give an alternative proof of Theorem
\ref{T:M10}, which is one of the central results of the current
paper. For this reason, we chose to present the results in the main
body of the paper in the set-up of simple algebras with involution and
their subalgebras, and confine a discussion of relevant
Galois-cohomological techniques to this appendix.

As in the main body of the paper, we let $(A , \tau)$ denote a
central simple $L$-algebra, with $\dim_L A = n^2,$ endowed with an
involution $\tau.$ Furthermore, we let $(E , \sigma)$ be an
$n$-dimensional commutative  \'etale $L$-algebra with an involutive
automorphism $\sigma$ that leaves $L$ invariant and satisfies
$\sigma \vert L = \tau \vert L$ and also condition (\ref{E:I1}) \S
\ref{S:I}. Set $K = L^{\tau}.$ To streamline the exposition, we will
leave out the case where $\tau$ is of the first kind and $n$ is odd
as otherwise we find ourselves in the split case which is
well-understood in terms of the classical results of the theory of
quadratic forms, cf.\:\S \ref{S:OS}. So we will assume that either
$\tau$ is of the second kind, or $\tau$ is of the first kind (hence
$K = L$), $n$ is even and $\dim_K F = n/2$ where $F = E^{\sigma}.$
Then it follows from Propositions \ref{P:E10} and \ref{P:E12} that
$E$ is a 2-dimensional free $F$-module, and hence the corresponding
unitary group $\mathrm{U}(E , \sigma)$ is a torus which we will
denote by $T.$ Clearly,
\begin{equation}\tag{A1}
T \simeq \mathrm{R}_{F/K}(\mathrm{R}_{E/F}^{(1)}(\mathrm{GL}_1))
\end{equation}
in the standard notations. Furthermore, we let $H$ denoted the
unitary group $\mathrm{U}(A , \tau)$ regarded as an algebraic
$K$-group.

Next, we assume that there is an embedding $\varepsilon \colon E
\hookrightarrow A$ which may or may not respect involutions. In the
sequel, we will use the same notations $\varepsilon,$ $\tau$ for the
natural extensions of these maps to $E \otimes_K K_{sep},$ $A
\otimes_K K_{sep}$ etc. According to Proposition \ref{P:G1}, there
exists a $\tau$-symmetric $g \in A^{\times}$ such that
\begin{equation}\tag{A2}
\varepsilon(\sigma(x)) = g^{-1} \tau(\varepsilon(x)) g \ \ \text{for
all} \ \ x \in E.
\end{equation}
Pick $s \in (A \otimes_K K_{sep})^{\times}$ so that
\begin{equation}\tag{A3}
g = \tau(s)s.
\end{equation}
In the sequel, we will use the standard notation and conventions
from Galois cohomology of algebraic groups (cf., for example,
\cite{PlR}, Ch.\:VI, or \cite{SerreG}, Ch.\:III); in particular, for
an algebraic $K$-group $G$ we let $Z^1(K , G)$ denote the set of
1-cocycles on $\Ga(K_{sep}/K)$ with values in $G(K_{sep}),$ and let
$H^1(K , G)$ denote the corresponding cohomology set.

\vskip3mm

\noindent {\bf Proposition A.} (i) {\it Given $\xi = \{
\xi_{\theta} \} \in Z^1(K , T),$ set $\zeta_{\theta} =
s\varepsilon(\xi_{\theta}) \theta(s)^{-1}.$ Then $\zeta = \{
\zeta_{\theta} \} \in Z^1(K , H).$ Furthermore, the correspondence
$\xi \mapsto \zeta$ yields a well-defined map
$$
\varphi \colon H^1(K , T) \longrightarrow H^1(K , H).
$$}

\noindent {(ii)} \parbox[t]{12cm}{\it The equation
$g\varepsilon(b) = \tau(h)h$ has a solution $(b , h) \in F^{\times}
\times A^{\times}$ (which is equivalent to the existence of an
embedding $(E , \sigma) \hookrightarrow (A , \tau)$ as algebras with
involutions- cf.\:Theorem \ref{T:G1}) if and only if $\mathrm{Im}\:
\varphi$ contains the trivial element of $H^1(K , H).$}

\vskip3mm

\begin{proof}
{ (i):} First, we observe that
\begin{equation}\tag{A4}
s\varepsilon(T)s^{-1} \subset H.
\end{equation}
Indeed, for any $x \in T(K),$ using (A2) and (A3), we obtain
$$
\tau(s\varepsilon(x)s^{-1})(s\varepsilon(x)s^{-1}) =
\tau(s)^{-1}\tau(\varepsilon(x))g\varepsilon(x)s^{-1} =
\tau(s)^{-1}g\varepsilon(\sigma(x)x)s^{-1}
$$
$$\ \ \ \ \ \ \ \ \ \ \ \ \ \ \ \ \ \ \ \ \ \ \ \ \ \
= \tau(s)^{-1}gs^{-1} = 1.
$$
It follows that for any $\theta \in \Ga(K_{sep}/K),$ we have
$$
\zeta_{\theta} = (s\varepsilon(\xi_{\theta})s^{-1})(s\theta(s)^{-1})
\in H(K_{sep})
$$
as
$$
g = \tau(s)s = \theta(g) = \theta(\tau(s))\theta(s) =
\tau(\theta(s))\theta(s),
$$
hence $s\theta(s)^{-1} \in H(K_{sep}).$ Furthermore, for any
$\theta_1, \theta_2 \in \Ga(K_{sep}/K),$ we have
$$
\zeta_{\theta_1}\theta_1(\zeta_{\theta_2}) =
s\varepsilon(\xi_{\theta_1\theta_2})(\theta_1\theta_2)(s)^{-1} =
\zeta_{\theta_1\theta_2},
$$
proving that $\zeta = \{ \zeta_{\theta} \} \in Z^1(K , H).$ Finally,
we show that the correspondence $\xi \mapsto \zeta$ takes cohomologus
cocycles into cohomologus cocycles. Indeed, for any $t \in
T(K_{sep})$ we have
$$
s\varepsilon(t\xi_{\theta}\theta(t)^{-1})\theta(s)^{-1} =
(s\varepsilon(t)s^{-1})\zeta_{\theta}\theta(s\varepsilon(t)s^{-1})^{-1},
$$
which defines a cocycle cohomologus to $\zeta_{\theta},$ in view of
(A4). So, the correspondence $\xi \mapsto \zeta$ gives rise to a
well-defined map $\varphi \colon H^1(K , T) \longrightarrow H^1(K ,
H).$

\vskip2mm

{(ii):} First, recall that if $a \in A^{\times}$ is
$\tau$-symmetric and $a = \tau(x)x$ with $x \in (A \otimes_K
K_{sep})^{\times}$, then $\zeta = \{ \zeta_{\theta} \}$, where
$\zeta_{\theta} = x\theta(x)^{-1},$ is a cocycle in $Z(K , H),$
which is cohomologus to the trivial cocycle if and only if the equation $a = \tau(x)x$ has a
solution in $A^{\times}.$ Next, it follows from (A1) that
\begin{equation}\tag{A5}
H^1(K , T) \simeq F^{\times}/N_{E/F}(E^{\times}),
\end{equation}
and the inverse of this isomorphism can be described as follows.
Given $b \in F^{\times},$ pick $c \in (E \otimes_K
K_{sep})^{\times}$ so that $b = \sigma(c)c$ ($=N_{E/F}(c)$), and for
$\theta \in \Ga(K_{sep}/K)$ set $\xi_{\theta} = c\theta(c)^{-1}.$
Then $\xi = \{ \xi_{\theta} \} \in Z^1(K , T),$ and the
correspondence
$$
bN_{E/F}(E^{\times}) \ \mapsto \ (\text{class of}\ \xi)
$$
gives the inverse of the isomorphism (A5).

Now, suppose that the equation $g\varepsilon(b) = \tau(h)h$ has a
solution $(b , h) \in F^{\times} \times A^{\times}.$ We then choose
$c \in (E \otimes_K K_{sep})^{\times}$ and construct $\xi \in Z^1(K
, T)$ as in the previous paragraph, for that $b.$ Then
with $s$ as in (A3), we have
\begin{equation}\tag{A6}
g\varepsilon(b) = g\varepsilon(\sigma(c)c) =
\tau(\varepsilon(c))g\varepsilon(c) =
\tau(s\varepsilon(c))(s\varepsilon(c)).
\end{equation}
So, $x := s\varepsilon(c)$ is a solution to $g\varepsilon(b) =
\tau(x)x,$ which also has the solution $h \in A^{\times}.$ By the
remark above, this means that the cocycle $\zeta \in Z^1(K , H)$, corresponding to $x$,
given by
\begin{equation}\tag{A7}
\zeta_{\theta} = x\theta(x)^{-1} =
s\varepsilon(c\theta(c)^{-1})\theta(s)^{-1},
\end{equation}
lies in the trivial class in $H^1(K , H).$ On the other hand,
$\varphi(\xi) = \zeta,$ and therefore $\mathrm{Im}\: \varphi$
contains the trivial element of $H^1(K , H).$ Conversely, suppose
$\xi \in Z^1(K , T)$ is such that $\varphi(\xi)$ represents the trivial element of
$H^1(K , H)$. Using (A5) and subsequent remarks, we can write $\xi = \{
\xi_{\theta} \}$, where
$$
\xi_{\theta} = c\theta(c)^{-1} \ \text{for} \ \text{some} \ c \in (E
\otimes_K K_{sep})^{\times} \ \text{such that} \ b:= \sigma(c)c \in
F^{\times}.
$$
Then (A6) shows that $x = s\varepsilon(c)$ satisfies
$g\varepsilon(b) = \tau(x)x,$ and (A7) combined with the definition
of $\varphi$ implies that the class in $H^1(K , H)$ corresponding
to $x,$ coincides with $\varphi(\xi),$ hence is trivial. So, the
equation $g\varepsilon(b) = \tau(h)h$ has a solution $h \in
A^{\times},$ as required.
\end{proof}

\vskip2mm

We can now reformulate the question about the local-global principle
for the existence of an embedding $(E , \sigma) \hookrightarrow (A ,
\tau)$ as algebras with involutions as follows. For $v \in V^K,$
define the corresponding local map $\varphi_v \colon H^1(K_v , T)
\longrightarrow H^1(K_v , H)$ just as we defined $\varphi$ in Proposition A1(i). {\it
Does the fact that $\mathrm{Im}\: \varphi_v$ contains the trivial
element of $H^1(K_v , H)$ for all $v \in V^K$ imply that
$\mathrm{Im}\: \varphi$ contains the trivial element of $H^1(K ,
H)$?} To analyze this question, we consider the following diagram
\begin{equation}\tag{A8}
\begin{array}{ccc}
H^1(K , T) & \stackrel{\varphi}{\longrightarrow} & H^1(K , H)\\
\alpha \downarrow &  & \downarrow \beta \\
\displaystyle \prod_{v \in V^K} H^1(K_v , T) &
\stackrel{\Phi}{\longrightarrow} &\displaystyle \prod_{v \in V^K}
H^1(K_v , H),
\end{array}
\end{equation}
in which $\Phi = \prod \varphi_v$ and $\alpha,$ $\beta$ are induced
by restrictions. Clearly, the above question is much more tractable
if $\beta$ is injective, i.e., $H$ satisfies the Hasse principle for
Galois cohomology. The Hasse principle may fail for orthogonal
involutions in the non-split case - see below, but it is valid in
all other cases at hand. We will now use this to explain why the
proof of Theorem \ref{T:Sym1}, which yields the unconditional
local-global principle for embeddings if $\tau$ is symplectic, was
so easy. In this case, $H$ is connected and simply connected (of
type $C_{\ell},$ for $\ell = n/2$), so $H^1(K_v , H) = 1$ for all $v
\in V^K_f$ (cf.\:\cite{PlR}, Theorem 6.4). So, instead of (A8), we
can work with the following:
\begin{equation}\tag{A9}
\begin{array}{ccc}
H^1(K , T) & \stackrel{\varphi}{\longrightarrow} & H^1(K , H)\\
\alpha \downarrow &  & \downarrow \beta \\
\displaystyle \prod_{v \in V^K_{\infty}} H^1(K_v , T) &
\stackrel{\Phi}{\longrightarrow} &\displaystyle \prod_{v \in
V^K_{\infty}} H^1(K_v , H).
\end{array}
\end{equation}
It is known that $\alpha$ is surjective (\cite{PlR}, Proposition 6.17), and $\beta$ is injective (in fact,
bijective) (\cite{PlR}, Theorem 6.6). So, a simple diagram chase
shows that if $\mathrm{Im}\: \varphi_v$ contains the trivial class
in $H^1(K_v , H)$ for all $v \in V^K_{\infty}$, then $\mathrm{Im}\:
\varphi$ contains the trivial class in $H^1(K , H),$ as required.
(Notice that this is not an alternative proof of Theorem
\ref{T:Sym1}, but rather a cohomological interpretation of the
argument given in \S \ref{S:Sym}.)

Next, we consider the case where $\tau$ is of the second kind. Then
$H$ is a connected reductive group, whose commutator subgroup $G =
\mathrm{SU}(A , \tau)$ is simply connected. So, $H^1(K_v , G) = 1$
for all $v \in V^K_f,$ however $H^1(K_v , H) \neq~1$ for $v \in
V^K_f$ that do not split in $L,$ and therefore it is not enough to work with (A9) in this case.
To study cohomology of $H$ we consider the
exact sequence
\begin{equation}\tag{A10}
1 \to G \longrightarrow H \stackrel{\det}{\longrightarrow} S \to 1,
\end{equation}
where $S = \mathrm{R}_{L/K}^{(1)}(\mathrm{GL}_1)$, and $\det$ is the
homomorphism of reduced norm, and the corresponding sequence of
cohomology
$$
H^1(K , G) \stackrel{\gamma}{\longrightarrow} H^1(K , H)
\stackrel{\delta}{\longrightarrow} H^1(K , S).
$$
We have an isomorphism $H^1(K , S) \simeq
K^{\times}/N_{L/K}(L^{\times})$ similar to (A5), and it is easy to
compute that in terms of these isomorphisms the composite map
$\delta \circ \varphi$ can be described as follows
$$
H^1(K , T) \ni bN_{E/F}(E^{\times}) \ \stackrel{\delta \circ
\varphi}{\longrightarrow} \ (\mathrm{Nrd}_{A/L}(g) \cdot N_{F/K}(b))
N_{L/K}(L^{\times}) \in H^1(K , S).
$$
The compositions $\delta_v \circ \varphi_v,$ where $\delta_v
\colon H^1(K_v , H) \to H^1(K_v , S)$ is obtained from (A10) over
$K_v,$ have a similar description. For every $v \in V^K,$ there
exists $b_v \in (F \otimes_K K_v)^{\times}$ such that for the
corresponding cocycle $\xi_v \in H^1(K_v , T),$ the element
$\varphi_v(\xi_v) \in H^1(K_v , H)$ is trivial. Applying $\delta_v$
and using the above description, we obtain that
$$
\mathrm{Nrd}_{A/L}(g) \cdot N_{F \otimes_K K_v/K_v}(b_v) \in N_{L
\otimes_K K_v/K_v}((L \otimes_K K_v)^{\times}).
$$
Now, assuming that $E/L$ is a field extension, which enables us to
use the multinorm principle (Proposition \ref{P:U1}) and the
subsequent argument in \S \ref{S:U}, we conclude that there exists
$b \in F^{\times}$ such that
\begin{equation}\tag{A11}
\mathrm{Nrd}_{A/L}(g) \cdot N_{F/K}(b) \in N_{L/K}(L^{\times})
\end{equation}
and
\begin{equation}\tag{A12}
b \in b_v N_{E \otimes_K K_v/F \otimes_K K_v}((E \otimes_K
K_v)^{\times}) \ \ \text{for all} \ \ v \in V^K_{\infty}.
\end{equation}
We claim that if $\xi \in H^1(K , T)$ is the cocycle corresponding
to $b$ then $\zeta := \varphi(\xi)$ is trivial. Indeed, (A11)
implies that $\delta(\zeta) = 1,$ and therefore $\zeta \in
\gamma(H^1(K , G)).$ But for any $v \in V^K_f$ we have  $H^1(K_v ,
G) = 1,$ which yields that the image of $\zeta$ in $H^1(K_v , G)$ is
trivial. On the other hand, due to (A12), for any $v \in
V^K_{\infty},$ the image of $\zeta$ in $H^1(K_v , H)$ coincides with
that $\varphi_v(\xi_v),$ hence is also trivial. Thus, $\beta(\zeta)
= 1,$ so the injectivity of $\beta$ (which is equivalent to
Landherr's theorem, cf.\:\cite{PlR}, \S 6.7, implies that $\zeta =
1,$ as required. (Again, this argument is simply the cohomological
version of the proof of Theorem \ref{T:U1}.)

For an orthogonal involution $\tau,$ the group $H$ is no longer
connected, and more importantly, may fail to satisfy the Hasse
principle for Galois cohomology, i.e., $\beta$ need not be injective,
in the nonsplit case (cf.\:\cite{K}, \S 5.11, or \cite{PlR}, \S 6.6).
This is a serious obstacle to obtaining a purely cohomological proof of
Theorem \ref{T:O-101}. To overcome this obstacle, we were forced to
introduce some new techniques in \S \ref{S:O} and study the classes $[C(A , \nu , \phi)]$.

\vskip2mm

Finally, one can view Theorem \ref{T:G1} as the assertion that the
existence of an embedding $(E , \sigma) \hookrightarrow (A , \tau)$
is equivalent to the existence of a $K$-rational point on the
variety
$$
Y := \{ (b , h) \ \in \mathrm{R}_{F/K}(\mathrm{GL}_1) \times
\mathrm{GL}_{1 , A} \ \vert \ g\varepsilon(b) = \tau(h)h \}.
$$
So, we would like to point out that $Y$ is in fact a homogeneous
space of the group ${\cG} := H \times
\mathrm{R}_{E/K}(\mathrm{GL}_1)$ under the following action
$$
(x , z) \cdot (b , h) = (\sigma(z) b z, \: xhz).
$$
Furthermore, in our previous notations, $(1 , s) \in Y,$ and the
stabilizer of this point is the torus $\{ (st^{-1}s^{-1} , t) \:
\vert \: t \in T \}.$ Thus, the question about the local-global
principle for the existence of an embedding $(E , \sigma)
\hookrightarrow (A , \tau)$ fits into the general framework of the
Hasse principle for homogeneous spaces of linear algebraic groups.
Among early results in this area one can mention the validity of the
Hasse principle for projective homogeneous  varieties (Harder
\cite{Ha0}) and for symmetric spaces of absolutely simple simply
connected groups (Rapinchuk \cite{R}). Later, Borovoi in a series of
papers developed cohomological methods for analyzing the Hasse
principle for homogeneous spaces with connected stabilizers, of an
arbitrary connected group whose maximal semi-simple subgroups are
simply connected. In particular, in \cite{Bor1}, he proved that the
Brauer-Manin obstruction is the only obstruction to the Hasse
principle in this situation, and in \cite{Bor2}, computed this
obstruction in terms of Galois cohomology (some methods for
computing the Brauer group of a compactification of a given
homogeneous space are given in \cite{C-TK}). It would probably be
interesting to use these techniques to show that the Brauer-Manin
obstruction for $Y$ is trivial if $\tau$ is a symplectic involution,
and to compute it precisely when $\tau$ is of the second kind
(apparently, it is related to the Tate-Shafarevich group of the
multinorm torus associated with the pair of \`etale algebras $(F ,
L)$). However, because of the  concrete description of $Y$, one can
give a direct Galois-cohomological analysis of the existence of a
$K$-rational point on it which results in the condition described in
Proposition A(ii). We feel that the general results on homogeneous
spaces are unlikely to lead to an alternative proof of our results.
Moreover, for an orthogonal involution $\tau,$ the group ${\cG}$ is
disconnected, which makes Borovoi's results inapplicable, but this
can serve as  a motivation to extend these results to some class of
disconnected groups which includes~${\cG}.$

\vskip7mm

\bibliographystyle{amsplain}

\end{document}